\documentclass[a4paper,reqno,11pt]{amsart}

\calclayout

\usepackage{amssymb}
\usepackage{amsmath}
\usepackage{amsthm}
\usepackage{enumitem}
\usepackage[table]{xcolor}
\usepackage{latexsym, mathrsfs, graphicx, fancyhdr,bm}
\usepackage{multirow}
\usepackage[final]{listings}
\usepackage{longtable}
\usepackage[leftmargin=3em]{quoting}
\usepackage[hidelinks]{hyperref}

\newtheorem{mainthA}{Theorem}

\newtheorem{Th}{Theorem}[section]
\newtheorem{Lem}[Th]{Lemma}
\newtheorem{Prob}[Th]{Problem}
\newtheorem{Cor}[Th]{Corollary}
\newtheorem{Que}[Th]{Question}
\newtheorem{Prop}[Th]{Proposition}
\newtheorem*{T1}{Theorem~\ref{theorem}}

\newtheorem*{T4}{Theorem~\ref{theoremGR}}

\theoremstyle{definition}
\newtheorem{Def}[Th]{Definition}
\newtheorem*{Def*}{Definition}
\newtheorem{example}[Th]{Example}

\newtheorem{Not}[Th]{Notation}

\theoremstyle{remark}
\newtheorem{Rem}[Th]{Remark}

\numberwithin{equation}{section}

\newcommand{\fpr}{{\mathrm{fpr}}}

\newcommand{\Sym}{\mathrm{Sym}}

\newcommand{\diag}{\mathrm{diag}}

\newcommand{\Stab}{\mathrm{Stab}}
\newcommand{\sgn}{\mathrm{sgn}}
\newcommand{\per}{\mathrm{perm}}
\newcommand{\Det}{\mathrm{Det}}
\newcommand{\Reg}{\mathrm{Reg}}
\newcommand{\Aut}{\mathrm{Aut}}

\newcommand{\GaL}{{\Gamma \mathrm{L}}}
\newcommand{\PGaL}{{\mathrm{P}\Gamma \mathrm{L}}}
\newcommand{\SL}{{\mathrm{SL}}}
\newcommand{\PSL}{{\mathrm{PSL}}}
\newcommand{\GL}{{\mathrm{GL}}}
\newcommand{\PGL}{{\mathrm{PGL}}}

\newcommand{\Sp}{{\mathrm{Sp}}}

\newcommand*{\Scale}[2][4]{\scalebox{#1}{$#2$}}

\setlist[enumerate,1]{label={(\arabic*)}}

\title{Base sizes for finite linear groups with solvable stabilisers}

\author{Anton A. Baykalov}
\email{a.a.baykalov@gmail.com}

\address{Department of Mathematics and Statistics, The University of Western Australia, Perth, Australia}

\begin{document}

\begin{abstract}
Let $G$ be a transitive permutation group on a finite set with solvable point stabiliser and assume that the  solvable radical of $G$ is trivial. In 2010, Vdovin  conjectured that the base size of $G$ is at most 5. Burness proved this conjecture in the case of primitive $G$. The problem was reduced by Vdovin in 2012 to the case when $G$ is an almost simple group. Now the problem is further reduced to groups of Lie type through work of Baykalov and Burness. In this paper, we prove the strong form of the conjecture for all almost simple groups with socle isomorphic to $\PSL_n(q),$ and the remaining classical groups will be handled in two forthcoming papers.   
\end{abstract}

\maketitle

\section{Introduction}
\label{intro}

\subsection{Main problem and results}

Consider a permutation group $G$ on a finite set $\Omega.$ A  \emph{base} of $G$ is a subset of $\Omega$ such that its pointwise stabiliser is trivial. The \emph{base size} of $G$ is the minimal size of a base. The study of bases and base sizes of permutation groups is an active research area with a rich history. These concepts have applications in abstract group theory and computational algebra; for a summary, see \cite{bailey} and  \cite{sfa}. In the case of a transitive group $G$ with point stabiliser $H$, we write $b_H(G)$ for the base size, noting that the action of $G$ on $\Omega$  is permutation isomorphic to the action of $G$ on right $H$-cosets by right multiplication.

In the past few decades, significant progress has been made in determining the base sizes of groups, with a particular focus on the primitive actions of almost simple groups. For example, Liebeck and Shalev \cite{lieb} proved the following conjecture of Cameron and Kantor \cite{Camer}:  if $G$ is an almost simple finite group and $H\le G$ is maximal and core-free, then there exists an absolute constant $c$ such that $b_H(G)$  is at most   $c$ unless $(G,H)$ lies in a prescribed list of exceptions. The exceptions arise when the action of $G$ on the set of right cosets of $H$  is \emph{standard} (see \cite[Definition 1.1]{burness}). Results specifying bounds for $b_H(G)$ include \cite{burness}, \cite{BuGuSa}, \cite{bls}, \cite{spor} and \cite{mon}. In particular, if $G$ is an almost simple group in a primitive non-standard action, then $b_H(G) \le 7$ with equality if and only if $G=\mathrm{M}_{24}$ and $H=\mathrm{M}_{23}.$

This paper is the first in a series of three devoted to the study of  base sizes of finite transitive  permutation groups with solvable point stabilisers.  In \cite[Problem 17.41(b)]{kt}, Vdovin conjectures (in a slightly different notation) that $b_S(G) \le 5$ for a transitive permutation group $G$ with a solvable point stabiliser $S$ and trivial solvable radical (the unique maximal solvable normal subgroup). It is easy to see that we obtain an equivalent statement if we omit the condition on the solvable radical and force $S$ to be a maximal solvable subgroup; that is, $S$ is not contained in any larger solvable subgroup of $G$.  Burness \cite{burPS} proved this conjecture in the case of primitive $G$, building on important earlier work of Seress \cite{seress}, who established the bound $b_S(G)\le 4$ for every solvable primitive group. Notice that the bound in the conjecture is the best possible since $b_S(G)=5$ if $G=\Sym(8)$ and $S=\Sym(4) \wr \Sym(2).$  In fact, there are infinitely many examples with $b_S(G)=5$, for example see \cite[Remark 8.3]{burPS}.

 Vdovin reduces the above conjecture to almost simple groups in \cite[Theorem 1]{vd}. In order to explain how exactly the reduction works, let us define $\Reg(G,k)$ to be the number of distinct regular orbits in the  action of a permutation group $G \le \Sym(\Omega)$ on the Cartesian product $\Omega^k$; here $G$ acts on $\Omega^k$ by $$(\alpha_1, \ldots,\alpha_k)g = (\alpha_1g,\ldots,\alpha_kg).$$ Analogous to $b_H(G)$, we write $\Reg_H(G,k)$ for $\Reg(G,k)$ in the case when $\Omega$ is the set of right cosets of a subgroup $H$. Since a regular point in $\Omega^k$ forms a base for $G$ acting on $\Omega$, $\Reg_H(G,k)=0$ for $k < b_H(G).$    Now we return to the reduction: \cite[Theorem 1]{vd} implies that, in order to prove Vdovin's conjecture,   
it is sufficient to show $$\Reg_S(G,5) \ge 5$$ for every almost simple group $G$ and each of its maximal solvable subgroups $S$. This inequality is established for almost simple groups with alternating and sporadic socle in \cite{bay} and \cite{burspor}, respectively. Therefore, Vdovin's conjecture is now reduced to the case of almost simple groups of Lie type.

 This paper is the first one to consider the general form of the conjecture for groups of Lie type. Here we  consider groups with socle isomorphic to $\PSL_n(q)$ and prove the following. 

  \begin{Th}
\label{fulltheoremPSL}
Let  $G\le \Sym(\Omega)$ be a transitive almost simple group with  solvable point stabiliser $S$ and socle isomorphic to $\PSL_n(q)$. Then $\Reg_S(G,5)\ge 5$ and thus $b_S(G) \le 5$.
\end{Th}

\begin{Rem}\label{remSleH}
\begin{enumerate}

\item Note that the statement of Theorem 1  in \cite{vd} is incorrect and this is fixed in the latest arXiv version (see the link in the Bibliography). 

\item The bound on $b_S(G)$ in Theorem \ref{fulltheoremPSL} is the best possible. It is achieved, for example, when $(2)$ of Theorem \ref{theoremGR} (see below) holds.

\item Even though $S$ is not stated to be maximal solvable in Theorem~\ref{fulltheoremPSL}, it is sufficient for the proof to consider only maximal solvable subgroups. Indeed, if $H$ is a core-free (so $\cap_{g\in G} H^g=1$) subgroup of $G$  and $S \le H$, then $b_S(G) \le b_H(G)$.

\item As already mentioned above, this paper is the first in a series of three where our goal is to prove Vdovin's conjecture for the almost simple classical groups.  In a forthcoming second paper (based on \cite{thesis}) we deal with unitary and symplectic groups. The final third paper dealing with orthogonal groups is being prepared. 
\end{enumerate}
\end{Rem}

The following elementary lemma will be a useful observation.
  \begin{Lem}[{\cite[Lemma 3]{bay}}] \label{base4}
Let $G$ be a finite permutation group with point stabiliser $H$. If  $b_H(G)\le 4$, then $\Reg_H(G,5)\ge 5.$
\end{Lem}

  By Remark \ref{remSleH}(3), if $G$ is almost simple, $S$ is a maximal solvable subgroup of $G$, and  $S \le H \le G$ (and $H$ does not contain the socle of $G$), then $b_S(G) \le b_H(G).$ Hence, results on the base size of primitive almost simple groups can provide useful information.
As we mention above, many results specifying bounds on base sizes of almost simple groups in primitive actions are available. 
In particular, Burness \cite{burness} proves the following.
\begin{Th}
\label{bernclass}
Let $G \le \Sym(\Omega)$ be a non-standard finite almost simple  primitive permutation   group with classical socle and  point stabiliser $H$. Then either $b_H(G) \le 4$, or $G = \mathrm{U}_6 (2) \cdot 2$, $H = \mathrm{U}_4 (3) \cdot 2^2$ and $b_H(G) = 5$.
\end{Th}

Roughly speaking, the non-standard hypothesis in Theorem \ref{bernclass} is equivalent to the condition $H \notin \mathcal{C}_1$. Here $\mathcal{C}_i$ for $i=1, \ldots, 8$ denote Aschbacher's subgroup collections introduced in \cite{asch}   and described in \cite[\S 2.1]{maxlow}  and \cite[Chapter 4]{kleidlieb}.  If $H \in \mathcal{C}_1,$ then it stabilises a subspace (or a pair of subspaces) of the natural module of $G$. Tables 2 and 3 in \cite{burness} contain detailed information on $b_H(G)$ for $n\le 5$, where $n$ is the dimension of the natural module.

If $G$ is classical and $H \in \mathcal{C}_1$, then $b_H(G)$ can be arbitrarily large since the order
of $G$ is not always bounded by a fixed polynomial function of the degree $|G:H|$ of the action.  In particular, by Lemma \ref{logd}, $b_H(G)$ is not bounded by a constant. For example, if $q\ge 3$ and $G=\PGL_n(q)$ acting on the set of $1$-dimensional subspaces of the natural module, then the point stabiliser $H$ is a parabolic subgroup lying in  $\mathcal{C}_1$ and $b_H(G)=n+1$.  

 Therefore,  if a maximal solvable subgroup $S$ lies only in a $\mathcal{C}_1$-subgroup of $G,$ then one cannot find a bound for $b_S(G)$ simply by studying the corresponding problem for maximal subgroups.
 
 \medskip
 
 Before we describe some techniques and ideas we use to prove  Theorem \ref{fulltheoremPSL},  
 let us briefly describe the structure of an almost simple group $G$ with socle $\PSL_n(q)$ and state more detailed results which also illustrate the cases that arise in the proof.
 
 Let $\GaL_n(q)$ be the group of invertible $\mathbb{F}_q$-semilinear transformations of $V=\mathbb{F}_q^n$, let  $\iota$ be the inverse-transpose map of $\GL_n(q)$, and let $Z$ be the group of all scalar matrices in $\GL_n(q)$ (see Section~\ref{secnot} for details). If $\PSL_n(q)$ is simple, so $n \ge 2$ and $(n,q)$ is neither $(2,2)$ nor $(2,3)$, then $$\Aut(\PSL_n(q)) \cong 
 \begin{cases}
 \PGaL_n(q) \rtimes \langle \iota \rangle =  ( \GaL_n(q) \rtimes \langle \iota \rangle)/Z & \text{ if } n \ge 3 \\
  \PGaL_n(q) =  \GaL_n(q) /Z & \text{ if } n = 2.
 \end{cases}$$
 
See \cite{Dieu} for details.  We achieve Theorem \ref{fulltheoremPSL} by proving the following two statements. 
\begin{mainthA}
\label{theorem}
Let $n \ge 2$ and $(n,q)$ is neither $(2,2)$ nor $(2,3).$ If $S$ is a maximal solvable subgroup of $\PGaL_n(q)$, 
 then $\Reg_S(S \cdot \PSL_n(q),5)\ge 5$. In particular $b_S(S \cdot \PSL_n(q)) \le 5.$
\end{mainthA}

\begin{mainthA}
\label{theoremGR}
Let $n\ge 3$. If $S$ is a maximal solvable subgroup of $\PGaL_n(q) \rtimes \langle \iota \rangle$ not contained in $\PGaL_n(q),$ then one of the following holds:
\begin{enumerate}
\item[$(1)$] $b_S(S \cdot \PSL_n(q))\le 4$;
\item[$(2)$] $n=4,$ $q\in \{2,3\}$, $S=P_2$, $b_S(S \cdot \PSL_4(q))=5$ and $\Reg_S(S \cdot \PSL_4(q),5)\ge 5.$
\end{enumerate}
\end{mainthA}

Here $P_2$ is the stabiliser in $\PGaL_n(q) \rtimes \langle \iota \rangle$ of a $2$-dimensional subspace of the natural module for $\PSL_n(q).$ Note that if $q=2$ in Theorem \ref{theoremGR}$(2)$, then $S \cdot \PSL_4(q) \cong \Sym(8)$ and $S \cong \Sym(4) \wr \Sym(2)$ which is the special case highlighted above.  The combination of the two theorems above establishes a stronger version of Theorem \ref{fulltheoremPSL}: if $\PSL_n(q)$ is simple and $S$ is a maximal solvable subgroup of $\Aut(\PSL_n(q))$, then $\Reg_S(S \cdot \PSL_n(q)) \ge 5.$ Indeed, Theorem \ref{fulltheoremPSL} follows from this statement via Lemma \ref{SdotG0}. 

\subsection{Methods and ideas.}\label{secmeth}

 It is more convenient for us to work with groups of matrices which makes the action  on the set of right cosets of a subgroup not faithful in most cases.  It is possible to extend the notion of $b_H(G)$ and $\Reg_H(G,k)$ to an arbitrary finite group $G$ by defining
  $$b_H(G):=b(G/H_G) \text{ and } \Reg_H(G,k):=\Reg(G/H_G,k).$$
  Here $H_G = \cap_{g \in G} H^g$ and the action of $G/H_G$ on the set $\Omega$ of right cosets of $H$ in $G$ is induced from the natural action of $G$. 
  
  Another straightforward but important observation is that the statement $b_H(G) \le k$ is equivalent to the existence of $k$ conjugates to $H$  such that their intersection is equal to $H_G.$ Indeed, $H^{g_1} \cap \ldots \cap H^{g_k}$ is the pointwise stabiliser of $(Hg_1, \ldots, Hg_k) \in~\Omega.$   In particular, in the proofs of Theorems \ref{theorem} and \ref{theoremGR}, we work with $\GaL_n(q)$ and $ \GaL_n(q) \rtimes \langle \iota \rangle $ rather than $\PGaL_n(q)$ and $\PGaL_n(q) \rtimes \langle \iota \rangle $.

 Now we are ready to outline the main techniques and ideas of this paper.  We use a combination of  probabilistic, constructive and computational methods to establish our results. The probabilistic approach is presented in Section \ref{fprsec}. It uses  fixed point ratio estimates to obtain an upper bound on the probability $Q(G,c)$ that a randomly chosen $c$-tuple of points in $\Omega$ is not a base. By showing $Q(G,c)<1$, one concludes that there exists a base of size $c$. This method was first introduced by Liebeck and Shalev in \cite{lieb}; it is used in most works related to base sizes of primitive permutation groups including  \cite{burness, sfa, bls, spor}. We use this method to obtain bounds for $b_S(S \cdot \SL_n(q))$ for irreducible maximal solvable subgroups of $\GL_n(q).$  Our results  are refinements of Theorem \ref{bernclass} in the sense that they provide better estimates for $b_H(G)$  for solvable $H$ not lying in a $\mathcal{C}_1$-subgroup of $G \le \PGL_n(q)$.  In particular, with an explicit list of exceptions, we show that $b_S(S \cdot \SL_n(q)) =2$ for an irreducible maximal solvable subgroup $S$ of $\GL_n(q)$  (see Theorem \ref{irred}).
 
  The probabilistic method does not work for us in the general case since the fixed points ratios are more difficult to estimate when $S$ is reducible and the corresponding bounds may be too large to be useful. Our reduction of the general case to  irreducible subgroups of $\GL_n(q)$ is constructive.
  
   We illustrate this for  a reducible maximal solvable subgroup $S$ of $\GL_n(q)$; the analysis for $S$ not contained in $\GL_n(q)$ is much more technical and splits into a number of smaller cases, although the main ideas are similar. So, let $S$ be a reducible maximal solvable subgroup of $\GL_n(q)$. 
    By Lemma \ref{supreduce}, in a suitable basis, matrices in $S$ have $k$ square $n_i \times n_i$ blocks ($i=1, \ldots, k$) on their diagonal and all the entries below these blocks are zero. For $g \in S$, each block $g_i$ lies in $\GL_{n_i}(q)$ and $\{g_i \mid g \in S\}$ forms an irreducible solvable subgroup of $\GL_{n_i}(q)$
     for which we can apply Theorem \ref{irred}.
           As a result, we obtain that the intersection of two (in most cases) conjugates of $S$ consists of upper triangular matrices. Further, using this ``anti-symmetry'' we explicitly construct elements of $\SL_n(q)$ as in Lemma \ref{irrtog}
      that, as a result of using them as conjugating elements, give us a subgroup of diagonal matrices in the intersection of four conjugates of $S$. Finally, we use Lemma \ref{diag} to construct a conjugate of $S$ whose intersection with the other four conjugates is the subgroup of scalar matrices, so $b_S(S \cdot \SL_n(q))\le 5$.  
       Five distinct regular orbits are then identified by  explicitly adjusting these conjugating elements (see the proof of Proposition \ref{case5prop}).
  
  Our approach requires detailed information of the structure of maximal solvable subgroups of $\Aut(\PSL_n(q)).$  Here we heavily rely on the classical work of Suprunenko \cite{sup} on solvable matrix groups. Many useful details also can be found in \cite{short}.  
  
  We use the  computer algebra systems {\sf GAP} \cite{GAP4} and {\sc Magma} \cite{magma} to compute  $b_S(S \cdot \SL_n(q))$ for some small values of $n$ and $q$ (usually $n \le 6$ and $q\le 19$) when theoretical methods are ineffective. In these cases, the index of an irreducible maximal solvable subgroup $S$ of $\GL_n(q)$ is small which makes it harder to obtain a suitable estimate for $Q(S \cdot \SL_n(q), c).$ For details see Section \ref{gapsec}.

\subsection{Further comments and implications}

Vdovin's conjecture is originally stated in terms of  subgroup intersections in \cite[Problem 17.41 (b)]{kt}:

\begin{Prob}\label{prob}
Let $S$ be a solvable subgroup of a finite group $G$ with trivial solvable radical.
 Do there always exist five conjugates of $S$ whose intersection is trivial?
\end{Prob}

As we mentioned above, since the stabiliser of $Sg$ in $G$ is $S^g$ for $g \in G$, the existence of five conjugates of $S$ whose intersection is trivial is equivalent to the statement that $b_S(G)\le 5.$ Given a subgroup with a certain property, the minimal number of conjugates one needs to obtain (by intersecting) a normal subgroup often  provides useful information. In particular, as we discuss below, it gives an estimate on the index of such a normal subgroup. For example, in a pioneering work  from 1966, Passman \cite{passman} proves that for a Sylow $p$-subgroup $P$ of a $p$-solvable finite group $G$ there exist $x,y \in G$ such that $P \cap P^x \cap P^y = O_p(G)$ where $O_p(G)$ is the unique maximal normal $p$-subgroup of $G$. In later work, Zenkov \cite{ZenP} extended this result to an arbitrary finite group $G.$  Similar results are achieved by Zenkov \cite{Zen93}, Vdovin \cite{vdovinReg} and Dolfi \cite{dolfi} for solvable Hall $\pi$-subgroups of a $\pi$-solvable group $G$.  We use the following well known result of Zenkov in our proofs on multiple occasions.   

\begin{Th}[\cite{zen}]
\label{zenab}
If $A$ and $B$ are abelian subgroups of a finite group $G$, then there exists $x \in G$ such that $A \cap B^x \le {\bf F}(G).$
\end{Th}
Here ${\bf F}(G)$ is the Fitting subgroup of $G$ (the maximal normal nilpotent subgroup).
In a recent paper \cite{zen21}, Zenkov shows that for any nilpotent subgroups $A,B,C$ of a finite group $G$ there exist $x,y \in G$ such that $A \cap B^x \cap C^y \le {\bf F}(G)$.


The study of intersections of subgroups is connected directly to a more general and abstract question of interest. Consider some property $\Psi$ of a finite group inherited by all its subgroups.  Important examples of  such a property are  the following: cyclicity, commutativity, nilpotence, solvability.
A natural question arises: how large is a normal $\Psi$-subgroup in an arbitrary finite group $G$? A more precise formulation of this question is the following:

\begin{Que}\label{que1}
Given a finite group $G$ with a $\Psi$-subgroup $H$ of index $n$, is it true that $G$ has a normal $\Psi$-subgroup whose index is bounded by some function $f(n)?$
\end{Que}

Since $H_G$, which is a normal $\Psi$-subgroup, is the kernel of the action of $G$ on the set of right cosets of $H$, and such an action provides a homomorphism to the symmetric group $\Sym(n),$   it always suffices to take $f(n)=n!$ for every such $\Psi$. We are interested in  stronger bounds, in particular those of the form $f(n)=n^c$ for some constant $c.$

Babai, Goodman and Pyber \cite{bab} prove some related results and state several conjectures.
In particular, they prove that
if a finite group $G$ has a cyclic subgroup ${C}$ of index $n$, then $ \cap_{g \in G} {{C}}^g$ has index at most $n^7.$
They also conjectured that the bound $n^2-n$ holds  and showed that it is best possible. Lucchini \cite{luccini} and, independently,  Kazarin and  Strunkov \cite{kaz} proved this. It follows from 
results  by Chermak and  Delgado \cite{cher} that the bound $n^2$ suffices for commutativity, and from results  by Zenkov \cite{zen21} that the bound $n^3$ suffices for nilpotence. We are interested in the case when $\Psi$ is solvability.

Babai, Goodman and Pyber \cite{bab}   prove the following statement.
\begin{Th}\label{indexBab}
There is an absolute constant $c$ such that, if a finite group $G$ has a
solvable subgroup $S$ of index $n$, then $G$ has a solvable normal subgroup of index
at most $n^c$.
\end{Th}

\noindent Although their proof does not yield an explicit value, they  conjectured  that $c \le 7$. It is straightforward to obtain a bound for $c$ from $b_S(G)$ for a maximal solvable subgroup $S$.
 Indeed, let  $G$ act on the set $\Omega$ of right $S$-cosets via right multiplication, so $|\Omega|=|G:S|.$  If $(\beta_1 , \ldots , \beta_k )\in \Omega^k$ is a base for the natural action of $G/S_G$ on $\Omega$, then $$|(\beta_1 , \ldots , \beta_k )^G| \le |\Omega| \cdot (|\Omega|-1) \ldots (|\Omega|-k+1)<|\Omega|^k=|G:S|^k.$$
Therefore,
$$|G:S_G|<|G:S|^k,$$
and if Problem \ref{prob} has a positive answer, then $c \le 5$ in Theorem \ref{indexBab} (we can assume that $S$ is maximal solvable, so $S_G$ is the solvable radical of $G$ and $G/S_G$ satisfies the condition of Problem \ref{prob}).

\bigskip

    The paper is organised as follows. In Section 2 we state necessary definitions and preliminary results. In Section 3 we obtain bounds on $b_S(S \cdot \SL_n(q))$ for irreducible maximal solvable subgroups of $\GL_n(q).$   Finally, we  use these results to prove Theorems \ref{theorem} and \ref{theoremGR} in Sections \ref{sec411} and \ref{sec412} respectively.  

\section{Definitions and preliminaries}\label{secnot}

All group actions are right actions. For example, the action of a linear transformation $g$ of a vector space $V$ on $v \in V$ is $(v)g \in V.$ For finite classical groups we follow the  notation of \cite{kleidlieb}. For algebraic groups our standard references are \cite[Chapter 1]{carter}, \cite[Chapter 1]{gorlyo} and \cite{hump}.

 Let $p$ be a prime and $q=p^f$,  $f \in \mathbb{N}$. Denote a finite field of size $q$ by $\mathbb{F}_q$, its algebraic closure by $K$ and the multiplicative group of a field $\mathbb{F}$ by $\mathbb{F}^*.$ 
Throughout, unless stated otherwise, $V=\mathbb{F}_{q}^n$ denotes a vector space of dimension $n$ over $\mathbb{F}_{q}$.

An $\mathbb{F}$-\emph{semilinear} transformation of a vector space $V$ over a field $\mathbb{F}$ is a map $g: V \to V$ such that there exists a field automorphism $\sigma(g) \in \Aut(\mathbb{F})$ with 
$$(v+u)g=(v)g+ (u)g \text{ and } (\lambda v)g= \lambda^{\sigma(g)}(v)g$$
for $v,u \in V$ and $\lambda \in \mathbb{F}.$
Such a map is \emph{non-singular} if $\{v \in V \mid (v)g=0\}=\{0\}.$ We define the group $\GaL(V,\mathbb{F})$ to be the set of all non-singular $\mathbb{F}$-semilinear transformations of $V$ with composition as the binary operation.   
We reserve the letter $\beta$ for a \emph{basis} of $V$, which we view as an ordered set. 
If $\alpha \in \Aut(\mathbb{F})$, then $\phi_{\beta}(\alpha)$ denotes the unique  $g \in \GaL(V, \mathbb{F})$ such that
\begin{equation}
\label{defphibet}
\left(\sum_{i=1}^n \lambda_i v_i \right) \phi_{\beta}(\alpha)= \sum_{i=1}^n\lambda_i^{\alpha} v_i
\end{equation}
 where $\beta =\{v_1, \ldots, v_n\}.$ If $\mathbb{F}= \mathbb{F}_{q}$ and $\alpha \in \Aut(\mathbb{F})$ is such that $\lambda^{\alpha}=\lambda^p$ for all $\lambda \in \mathbb{F}$,  then we denote $\phi_{\beta}(\alpha)$ by $\phi_{\beta}$ \index{$\phi_{\beta}$} or simply $\phi$ when $\beta$ is understood. It is routine to check that  $$\GaL(V, \mathbb{F}_q)= \GL(V, \mathbb{F}_q) \rtimes \langle \phi \rangle \cong \GL_n(q) \rtimes \langle \phi \rangle.$$

\medskip

\begin{Not}\label{not} We fix the following notation.
\begin{longtable}{p{1em} p{7em} p{27em} }
1. & $\delta_{i,j}$ & Kronecker delta\\
2. & $p'$ & set of all primes except $p$\\
3. & $(a,b)$ & greatest common divisor of integers $a$ and $b$\\
4. & $g^G$ & conjugacy class of $g \in G$\\
5. & $Z(G)$ &  center of a group $G$\\
6. & ${\bf F}(G)$ &  Fitting subgroup \index{Fitting subgroup} of a finite group $G$\\
7. & $O_{\pi}(G)$ & unique maximal normal $\pi$-subgroup for a set of primes $\pi$\\
8. & $A \rtimes B$ & semidirect product of groups $A$ and $B$ with $A$ normal\\
9. &  $\Sym(n)$ &  symmetric group of degree $n$ \\
10. & $\sgn(\pi)$ &   sign of a permutation $\pi$\\
11. & $M_n(\mathbb{F})$& algebra of all $n \times n$ matrices over $\mathbb{F}$\\
12. & $D(G)$ & subgroup of all diagonal matrices of a matrix group $G$\\
13. & $RT(G)$ & subgroup of all upper-triangular 
 matrices of a matrix group $G$\\
14. & $\Det(H)$ & $\{\det(h) \mid h \in H \}$ for $H \le \GL(V)$\\ 
15. & $\diag(\alpha_1, \ldots, \alpha_n)$ &   diagonal matrix with entries $\alpha_1, \ldots, \alpha_n$ on its diagonal  \\  
16. &  $\diag[g_1, \ldots, g_k]$ &   block-diagonal matrix with blocks $g_1, \ldots, g_k$ on its diagonal  \\
17. & $\per(\sigma)$ &  permutation matrix corresponding to $\sigma \in \Sym(n)$\\
18. & $g^{\top}$ &  transpose of a matrix $g$\\
19. & $g \otimes h$ &   Kronecker product \index{Kronecker product}   
 $\begin{pmatrix}
g \cdot h_{1,1}     &  \ldots & g \cdot h_{1,m}  \\
\ldots           &   \ldots   & \ldots   \\
g \cdot h_{m,1}     & \ldots     & g \cdot h_{m,m}      
\end{pmatrix} \in \GL_{nm}(q)$ \\ & & for $g \in \GL_n(q)$ and $h \in \GL_m(q)$ \\
\end{longtable}
\end{Not}

 An irreducible subgroup $G$ of $\GL(V)$ is \emph{imprimitive} if there exists a decomposition 
 $$V=V_1 \oplus \ldots \oplus V_k; \text{ } k>1$$  stabilised by $G$. So for each $i \in \{1, \ldots, k\}$ and $g \in G$ there exist   $j \in \{1, \ldots, k\}$ such that $(V_i)g=V_j.$ If such a decomposition does not exist, then $G$ is \emph{primitive}.

It is convenient to view  the symmetric group 
as a group of permutation matrices. We define the \emph{wreath product} \index{wreath product}  of  $X \le \GL_n(q)$ and a  group of permutation matrices $Y \le \GL_m(q)$ as the matrix group $X \wr Y \le \GL_{nm}(q)$ obtained
by replacing the entries 1 and 0  in every matrix in $Y$ by
arbitrary matrices in $X$ and by zero $(n \times n)$ matrices respectively. 

Let $A$ be an $(nm\times nm)$ matrix. We can view $A$ as the matrix
$$\begin{pmatrix}
A_{11}      &  \ldots & A_{1m}  \\
\ldots           &   \ldots   & \ldots   \\
A_{m1}     & \ldots     & A_{mm}      
\end{pmatrix}$$
where the $A_{ij}$ are $(n \times n)$ matrices. The vector $(A_{i1}, \ldots, A_{im})$ is the $i$-th $(n \times n)$-\emph{row} \index{$(n \times n)$-row} of $A.$

\subsection{General results and solvable linear groups}
\label{missec}

We begin by proving some general results which were mentioned in Section \ref{intro}.

\begin{Lem}\label{logd}
If $G\le \Sym(\Omega)$ and $d=|\Omega|,$ then $b(G) \ge \log_d|G|$.
\end{Lem}
\begin{proof}
Let $B \in \Omega^{b(G)}$ be a base. Every element of $G$ is uniquely determined by its action on $B$. Indeed, if $Bx=By$ for $x,y \in G$, then $Bxy^{-1}=B$ and $x=y$ since $B$ is a regular point. Hence $|G|\le d^{b(G)}.$ 
\end{proof}

\begin{Lem}\label{SdotG0}
Let $G_0$ be a finite simple nonabelian group. Let $G_0 \le G \le \Aut (G_0)$. If  $\Reg_H(H \cdot G_0,5)\ge 5,$ for every maximal solvable subgroup $H$ of $\Aut(G_0),$ then $\Reg_S(G,5)\ge 5,$ for every solvable subgroup $S$ of $G$.
\end{Lem}
\begin{proof} Let $S\le G$ be solvable. Let $H$ be a maximal solvable subgroup of $  \Aut(G_0)$ containing $S$. Since 
$$\Reg_H(H \cdot G_0,5)\ge 5,$$
 there exist $x_{(i,1)},x_{(i,2)},x_{(i,3)},x_{(i,4)},x_{(i,5)} \in G_0$ for $i \in\{1, \ldots, 5\}$ such that
$$\omega_i=(Hx_{(i,1)},Hx_{(i,2)},Hx_{(i,3)},Hx_{(i,4)},Hx_{(i,5)})$$ 
are $H \cdot G_0$-regular points   in $\Omega^5=\{Hx \mid x \in H \cdot G_0 \}^5$, and the $\omega_i$ lie in distinct orbits. 
We claim that $$\omega_i'=(Sx_{(i,1)},Sx_{(i,2)},Sx_{(i,3)},Sx_{(i,4)},Sx_{(i,4)}) \text{ for } i \in \{1, \ldots, 5\}$$ lie in distinct $G$-regular orbits in $(\Omega')^5=\{Sx \mid x \in G\}^5.$ Indeed, $$S^{x_{(i,1)}} \cap S^{x_{(i,2)}} \cap S^{x_{(i,3)}} \cap S^{x_{(i,4)}} \cap S^{x_{(i,5)}} \le H^{x_{(i,1)}} \cap H^{x_{(i,2)}} \cap H^{x_{(i,3)}} \cap H^{x_{(i,4)}} \cap H^{x_{(i,5)}}=1,$$ so $\omega_i'$ are regular. Assume that $\omega_1' g = \omega_2'$ for some $g \in G.$ Therefore,
$$(Sx_{(1,i)})g=Sx_{(2,i)} \text{ for } i \in \{1, \ldots, 5\}$$
and 
$$g \in \bigcap_{i=1}^5 (x_{(1,i)}^{-1}S x_{(2,i)}) \subseteq \bigcap_{i=1}^5 (x_{(1,i)}^{-1}H x_{(2,i)})= \emptyset$$
where the last equality holds since $\omega_1$ and $\omega_2$ lie in distinct $H \cdot G_0$-orbits. Hence $\omega_1'$ and $\omega_2'$ lie in distinct $G$-orbits. The same argument shows that all of the $\omega_i'$ lie in distinct $G$-orbits, so $\Reg_S(G,5)\ge 5$.
\end{proof}

  Further we state two classical results due to Suprunenko.
\begin{Lem}[{\cite[\S 18, Theorem 5]{sup}}] \label{supirr}
An irreducible solvable subgroup of $\GL_n(q)$ is either primitive, or conjugate in $\GL_n(q)$ to a subgroup of the wreath product
$S \wr \Gamma$
where $S$ is a primitive solvable subgroup of $\GL_m(q)$ and $\Gamma$ is a transitive solvable
subgroup of the symmetric group $\Sym(k)$ and $km = n$. In particular, an irreducible maximal solvable  subgroup of 
$\GL_n(q)$ is either primitive, or conjugate in $\GL_n(q)$
to  $S \wr \Gamma$, where $S$ is a  primitive maximal solvable subgroup of $\GL_m(q)$
and $\Gamma$ is a  transitive maximal solvable subgroup of $\Sym(k)$.
\end{Lem}

\begin{Lem}[{\cite[\S 18, Theorem 3]{sup}}]
\label{supreduce}
Let  $H$ be a
subgroup of $\GL_n(q)$. In a suitable basis $\beta$ of $V$, the matrices $g\in H$ have the
shape
\begin{gather}\label{stup}
\begin{pmatrix}
g_k     & g_{k,(k-1)} & \ldots & g_{k,1}  \\
0          &    g_{k-1}  & \ldots & g_{(k-1),1}  \\
\ldots     & \ldots     & \ldots & \ldots     \\
      0    &   0        & \ldots & g_1 
\end{pmatrix}
\end{gather}
where the mapping $\gamma_i: H \to \GL_{n_i}(q)$, $g \mapsto g_i$ is an irreducible 
representation of $H$ of degree $n_i$, and $g_{i,j}$ is an $(n_i \times n_j)$ matrix over $\mathbb{F}_q$, and $n _1 + \ldots + n _k = n$.
 The group $H$  is solvable if and only if all the
groups
$$H_i ={\mathrm{Im} } (\gamma_i) \text{ for } i = 1, \ldots, k,$$
are solvable.
\end{Lem}

 Now we state three technical lemmas about solvable linear groups. We use them and ideas from their proofs many times throughout  this paper. 

\begin{Lem}\label{supSL}
Let $H \le G \le \GL_n(q).$ Assume $\Det(H)=\Det(G).$ Then the following hold:
\begin{enumerate}[font=\normalfont]
\item $H \cdot (\SL_n(q) \cap G)=G$.
\item If $g \in G$, then there exists $g_1 \in \SL_n(q) \cap G$ such that $H^g=H^{g_1}.$ 
\end{enumerate}
\end{Lem}
\begin{proof}
Let $g \in G$. Since $\Det(H)=\Det(G)$, there exists $h \in H$ such that $\det(h)=\det(g).$ Therefore, 
$h^{-1}g \in \SL_n(q) \cap G$, so $g=h \cdot (h^{-1}g)\in H \cdot (\SL_n(q) \cap G).$

Let $g_1=h^{-1}g,$ so $H^{g_1}=g^{-1}hHh^{-1}g=g^{-1}Hg=H^g.$
\end{proof}

\begin{Lem}\label{irrtog}
Let $H$ be a subgroup of $\GL_n(q)$ of shape \eqref{stup}. If for every $H_i$ there exists $x_i \in \GL_{n_i}(q)$ (respectively $\SL_{n_i}(q)$) such that the intersection $H_i \cap H_i^{x_i}$ consists of upper triangular matrices, then there exist $x,y \in \GL_n(q)$ (respectively $\SL_n(q)$) such that 
$$(H\cap H^x) \cap (H\cap H^x)^y \le D(\GL_n(q)).$$
\end{Lem}
\begin{proof}
Let $x=\diag(x_k, \ldots, x_1)$ and let  $$y=\diag(\sgn(\sigma),1 \ldots,1) \cdot \per(\sigma)$$ where  $$\sigma =(1,n)(2, n-1) \ldots ([n/2], [n/2+3/2]).$$ Here $[r]$ is the integer part of a positive number $r$. Note that $\det(y)=1$ since $\det(\per(\sigma))=\sgn(\sigma).$  Now
$H\cap H^x$ consists of upper triangular matrices and $(H\cap H^x)^y$ consists of lower triangular matrices, so
\begin{equation*}
(H\cap H^x) \cap (H\cap H^x)^y \le D(\GL_n(q)). 
\end{equation*} 
Finally, note that if $x_i \in \SL_{n_i}(q)$, then $x \in \SL_n(q).$ 
\end{proof}

\begin{Def}
Let $\{v_1, v_2, \ldots,  v_n \}$ be a basis of $V$. We refer to 
$$v=\alpha_1v_1+ \alpha_2v_2+  \ldots+ \alpha_nv_n $$ 
as the \emph{decomposition} \index{decomposition} of $v \in V$ with respect to this basis. And we say that the decomposition contains $v_i$ if $\alpha_i \ne 0.$
 \end{Def}

\begin{Lem}\label{diag}
Assume that matrices in $H\le \GL_n(q)$ have shape \eqref{stup} with respect to the  basis $$\{v_1, v_2, \ldots,  v_n \}$$ and $n_1 < n,$ so $H$ stabilises $$U=\langle v_{n-n_1+1}, \ldots, v_n \rangle.$$  Then there exists $z \in \SL_n(q)$ such that $D(\GL_n(q)) \cap H^z \le Z(\GL_n(q)).$
\end{Lem} 
\begin{proof}
Let $m=n_1.$ Define vectors 
\begin{equation*}
\begin{split}
u_1 & = v_1 + \ldots + v_{n-m} + v_{n-m+1}  \\
u_2 & = v_1 + \ldots + v_{n-m} + v_{n-m +2}\\
u_3 & = v_1 + \ldots + v_{n-m}  + v_{n-m +3}\\
\vdots \\
u_m & = v_1 + \ldots + v_{n-m} + v_{n}.
\end{split}
\end{equation*}
Let $z \in \GL_n(q)$ be such that 
\begin{equation}\label{ui}
(v_{n-m+i})z=u_i, \text{ } i=1, \ldots, m.
\end{equation}
 Such  $z$ exists since $u_1, \ldots, u_m$ are linearly independent and we can assume $z\in \SL_n(q)$ since $m<n.$ So \eqref{ui} implies that 
$$Uz = \langle u_1, \ldots, u_m \rangle $$ is $H^z$-invariant.

Let $g \in D(\GL_n(q)) \cap H^z,$ say $g$ is $\text{diag}(\alpha_1, \ldots , \alpha_n)$ with respect to the basis $\{v_1, v_2, \ldots, v_{n}\}$. Thus,
$$(u_j)g = \alpha_1 v_1 + \ldots + \alpha_{n-m} v_{n-m} + \alpha_{n-m+j} v_{n-m+j}. $$
Also, $(u_j)g$ must lie in $Uz,$ since $g \in H^z,$ so 
$$(u_j)g = \beta_1 u_1 + \beta_2 u_2 + \ldots + \beta_m u_m.$$
But the decomposition of $(u_j)g$ with respect to    $\{v_1, v_2, \ldots, v_{n}\}$ does not contain $v_{n-m +i}$ for $i \ne j$, so 
$$\beta_1 = \beta_2 = \ldots =\beta_{j-1} =\beta_{j+1}= \ldots = \beta_m =0$$ and $(u_j)g= \beta_j u_j$. Thus, 
$$\alpha_1 = \alpha_2 = \ldots =  \alpha_{n-m}=\alpha_{n-m+j}$$
for $0<j\le m.$
Therefore, $g$ is scalar and lies in $Z(\GL_n(q)).$
\end{proof}

The following lemma plays an important role in our proof of Theorems \ref{theorem} and \ref{theoremGR}. 

\begin{Lem}
\label{scfield}
Let  $\phi=\phi_{\beta}$ where $\beta$ is a basis of $V,$ and assume $n\ge 2$ with $(n,q) \ne (2,2),(2,3).$ Let $H \le \GaL_n(q) $ and assume  $H \cap \GL_n(q) \le Z(\GL_n(q))$. Then there exists $b \in \GL_n(q)$ such that every element of $H^b$ has the form $\phi^i g$ for some $i \in \{1, \ldots, f\}$ and $g \in Z(\GL_n(q)).$   
\end{Lem}
\begin{proof}
 Let $Z=Z(\GL_n(q))$ and $\Gamma=\GaL_n(q).$ Notice that $\Gamma/Z$ is almost simple. Let $G_0$ and ${G}$ be the socle of $\Gamma/Z$ and the group of inner-diagonal automorphisms of $G_0$ respectively, so $G_0=\PSL_n(q)$ and ${G}= \PGL_n(q).$ Without loss of generality, we may assume $Z \le H.$ Observe $H\cap \GL_n(q) = Z,$ so $H/Z$ is cyclic and consists of field automorphisms of $G_0.$ Let $\varphi \in H$ be such that $\langle Z \varphi \rangle=H/Z.$   Since 
$$\Gamma=\GL_n(q) \rtimes \langle \phi \rangle,$$
we obtain that $\varphi \in \phi^i  \GL_n(q)$ for some  $i \in \{1, \ldots, f\}$ and $Z\varphi \in (Z\phi^i) {G}$.

 Since $Z \varphi$ and $Z \phi^i$ can be viewed as field automorphisms of $G_0$ of the same order, they are conjugate in ${G}$ by \cite[(7-2)]{conjaut}. So   there exists  $Zb \in \PGL_n(q)$ such that $(Z\varphi)^{Zb}=Z\phi^i$ for some $i \in \{1, \ldots,  f\}.$  Therefore, $H^b = Z\langle \varphi^b \rangle = Z\langle \phi^i \rangle = \langle\phi^i \rangle Z.$
\end{proof}

\subsection{Singer cycles}\label{singsec}

\begin{Def}\label{sindef} A  \emph{Singer cycle} \index{Singer cycle} of $\GL_n(q)$ is a cyclic subgroup of order $q^n - 1$.
\end{Def}

 A Singer cycle always exists and all Singer cycles are conjugate in $\GL_n(q)$ (see Lemma \ref{singconj}). Sometimes, it is convenient to identify a Singer cycle with the multiplicative group of a field of order $q^n$ as follows. A field $\mathbb{F}_{q^n}$ can be considered as an $n$-dimensional vector space $V=\mathbb{F}_q^n$ over $\mathbb{F}_q.$  Right multiplication by a generator of  $\mathbb{F}_{q^n}^*$ determines a bijective $\mathbb{F}_q$-linear map from $V$ to itself of order $q^n-1$. So  $\mathbb{F}_{q^n}^*$  is isomorphic to a cyclic subgroup of $\GL_n(q)$ of order $q^n-1$. It is well known, and is easy to deduce from the observation above, that such a subgroup $T$ is irreducible, and $T \cup \{0\}$ is a field under the usual matrix addition and multiplication.

\begin{Lem}\label{singconj} 
All Singer cycles are conjugate in $\GL_n(q).$ 
\end{Lem}
\begin{proof}
 By \cite[Chapter II, Satz 3.10]{hupp}, an irreducible cyclic subgroup of $\GL_n(q)$ is conjugate to a subgroup of the Singer cycle constructed below above. It remains to show that a Singer cycle is irreducible.

Let $T$ be a Singer cycle. By Maschke's Theorem, $T$ is completely reducible, so 
$V = V_1 \oplus \ldots \oplus V_k$ where $T$ stabilises the $V_i$ and each $V_i$ is $\mathbb{F}_q[T]$-irreducible.  Notice that $q^n-1$ is the largest possible order of an element of $\GL_n(q).$ Indeed,  let $\chi_x(t)$ be the characteristic polynomial of $x \in \GL_n(q)$. By the Cayley--Hamilton Theorem $\chi_x(x)=0$ and the dimension of the subalgebra $\mathbb{F}_q[\langle x \rangle]$ of $M_n(\mathbb{F}_q)$, generated by $x$, is at most 
$\deg \chi_x(t) \le n.$ So $|\mathbb{F}_q[\langle x \rangle]|\le|\mathbb{F}_q|^n=q^n$ and $|x|\le q^n-1.$ Therefore,  $|T| \le \prod_{i=1}^k (q^{n_i}-1)$ which is  possible only if $k=1$, so $T$ is irreducible. 
\end{proof} 

\begin{Th}[{\cite[Chapter II, \S 7]{hupp}}]\label{sin1} 
If $T$ is a Singer cycle of $\GL_n(q)$, then $$C_{\GL_n(q)}(T) = T$$
and $N_{\GL_n(q)} (T) /T$ is cyclic of order $n$. Moreover, $N_{\GL_n(q)} (T)=T \rtimes \langle \varphi \rangle$, where $t^{\varphi}=t^q$ for $t \in T.$ 
\end{Th}

Let $t \in \GL_n(q)$ be a generator of a Singer cycle. By the identification of $\mathbb{F}_q^n$ and $\mathbb{F}_{q^n}$ above, the characteristic polynomial of $t$ is the characteristic polynomial $\prod_{i=0}^{n-1}(x-\lambda^{q^i})$ of a primitive element $\lambda \in \mathbb{F}_{q^n}$ over $\mathbb{F}_q.$ So $t$ is conjugate to $\diag(\lambda, \lambda^q, \ldots, \lambda^{q^{n-1}})$ in $\GL_n(K)$ where $K$ is the algebraic closure of $\mathbb{F}_q.$ In particular, $\det(t)$ generates $\mathbb{F}_q^*,$ so $\langle t \rangle \SL_n(q)=\GL_n(q)$ by Lemma \ref{supSL}.

\subsection{Fixed point ratios and elements of prime order}
\label{fprsec}

\begin{Def}
If a group $G$ acts on a set $\Omega$, then   $C_{\Omega}(x)$ is the set of points in $\Omega$  fixed
by   $x \in G$. If $G$ and $\Omega$ are finite, then  the \emph{fixed point ratio} \index{fixed point ratio} of $x$, 
denoted  by $\fpr(x)$, is the proportion of points in $\Omega$ fixed by $x$, i.e. $\fpr(x) = |C_{\Omega}(x)|/|\Omega|.$
\end{Def}

If  $G$ acts transitively on a set $\Omega$  and $H$ is a point stabiliser,  then it is easy to see that
\begin{equation}\label{fpr}
\fpr(x)= \frac{|x^{G} \cap H|}{|x^G|}
\end{equation}

 In \cite{fpr, fpr2, fpr3, fpr4} Burness studies   fixed point ratios for primitive actions of almost simple classical groups and we recall some  observations from \cite{burness}.
 
Let a group $G$ act faithfully   on the set $\Omega$ of right cosets of a subgroup $H$ of $G.$  
Let $Q(G, c)$ be the probability that
a randomly chosen $c$-tuple of points in $\Omega$ is not a base for $G$, so $G$ admits a base of size $c$ if and
only if $Q(G, c) < 1$. Of course, a $c$-tuple is not a  base if and only if it is fixed by some element  $x \in G$ of prime order, and the probability that a random $c$-tuple is fixed
by $x$ is equal to $\fpr(x)^c$. Let $\mathscr{P}$ be the set of elements of prime order in ${G}$, and let ${x}_1, \ldots, {x}_k$
be representatives for the ${G}$-classes of elements in $\mathscr{P}$. Since  fixed point
ratios are constant on conjugacy classes, 
\begin{equation}\label{ver}
Q(G,c) \le \sum_{{x} \in \mathscr{P}}\fpr({x})^c = \sum_{i=1}^{k}|{x_i}^{{G}}|\cdot\fpr({x}_i)^c=:\widehat{Q}(G,c).
\end{equation}

The following observation is useful for bounding $\widehat{Q}(G,c)$ in \eqref{ver}.
 
\begin{Lem}[{\cite[Lemma 2.1]{burness}}]
\label{fprAB}
Let $G$ act faithfully and transitively on $\Omega$ and let $H$ be a point stabiliser. If  $x_1,\ldots,x_k$ represent distinct $G$-classes such that $\sum_{i=1}^k |x_i^G \cap  H| \le A$ and $|x_i^G| \ge B$ for all $i \in \{1, \ldots, k\},$ then
$$\sum_{i=1}^k |x_i^G| \cdot \fpr (x_i)^c \le B \cdot (A/B)^c.$$ 
for all $c \in \mathbb{N}.$
\end{Lem}

Notice that if there exists $\xi \in \mathbb{R}$ such that $\fpr({x})\le |{x}^{{G}}|^{-\xi}$ for every  ${x} \in \mathscr{P}$,  then  
$$\widehat{Q}(G,c)\le \sum_{i=1}^{k}|{x}_i^{{G}}|^{1-c\xi}.$$

\begin{Def}
Let $\mathscr{C}$ be the set of conjugacy classes of prime order elements in ${G}$.
For $t \in \mathbb{R},$ 
$$\eta_G(t):= \sum_{C \in \mathscr{C}}|C|^{-t}.$$
If $Z(G)=1,$ then there exists $T_G \in \mathbb{R}$ such that $\eta_G(T_G) = 1.$
\end{Def}

\begin{Lem}\label{11}
	If $G$ acts faithfully and transitively on $\Omega$ and  $\fpr({x})\le|{x}^{{G}}|^{-\xi}$ for all  ${x} \in \mathscr{P}$ and $T_G < c\xi - 1$, then $b(G) \le c$.
\end{Lem}
\begin{proof}
We follow the proof of \cite[Proposition 2.1]{burness}. Let ${x}_1, \ldots, {x}_k$ be representatives of the ${G}$-classes
of prime order elements in ${G}$. 
By \eqref{ver}, 
$$Q(G,c) \le  \sum_{i=1}^{k}|{x_i}^{{G}}|\cdot \fpr({x}_i)^c \le \eta_G(c\xi - 1).$$
The result follows since $\eta_G(t) < 1$ for all $t>T_G.$
\end{proof}

We fix the following notation for the rest of the section. Let $\overline{G}$ be an adjoint simple algebraic group of classical type over the algebraic closure $K$ of ${\mathbb{F}_p}$. Let $\overline{G}_{\sigma}=\{g \in \overline{G} \mid g^{\sigma}=g\}$ where $\sigma$ is a Frobenius morphism of $\overline{G}.$ Let $\overline{G}$ be such that $G_0=O^{p'}(\overline{G}_{\sigma})'$ is a finite simple classical group over $\mathbb{F}_q$ for some $p$-power $q$. Here  $O^{p'}(G)$ is the subgroup of a finite group $G$ generated by all $p$-elements of $G$.
 In particular, let $\overline{G}=\PSL_n(K)$ and let $\sigma$ be such that  $\overline{G}_{\sigma}=\PGL_n(q)$, so  $G_0=\PSL_n(q).$  Let  $G$ be a finite almost simple  group with  socle $G_0.$  

As proved in \cite[Proposition 2.2]{burness}, we have $T_G<1/3$ if  $n \ge 6$.
Thus, if for such $G$  
\begin{equation}\label{0} 
\fpr(x)<|x^G|^{-\frac{4}{3c}}
\end{equation}
 for all $x \in \mathscr{P} $, then $\xi \ge {4}/{(3c)}$ and $c\xi -1 \ge 1/3 >T_G$ and $G$ has a base of size $c$.

Therefore, Lemma \ref{11} allows us to estimate the base size by calculating bounds for $|x^G|$ and $|x^G \cap H|$ for elements $x$ of prime order.

\begin{Def}
\label{nudef}
Let $x \in \PGL(V)=\PGL_n(q)$. Let $K$ be the algebraic closure of $\mathbb{F}_q$, and let $\overline{V}=K \otimes V.$ Let $\hat{x}$ be a preimage of $x$ in $\GL_n(q).$ Define
$$ \nu_{V, K}(x):= \min\{\dim [\overline{V}, \lambda \hat{x}] : \lambda \in K^* \}. $$ Here $[V,g]$ for a vector space $V$ and $g \in \GL(V)$ is the commutator in $V \rtimes \GL(V)$.  Therefore, $\nu_{V, K}(x)$ is  the minimal codimension of an eigenspace of $\hat{x}$  on $\overline{V}$. In particular, $\nu_{V, K}(x)$ does not depend on the choice of the preimage $\hat{x}.$  Sometimes we denote this number by $\nu(x)$ and $\nu_{V, K}(\hat{x})$. 
\end{Def}

\begin{Lem}[{\cite[Lemma 3.11]{fpr2}}]\label{prost}  
Let $x \in \PGL_n(q)$ have  prime order $r.$ Then either 
\begin{enumerate}[font=\normalfont]
\item $x$ lifts to  $\hat{x} \in \GL_n(q)$ of order $r$ such that $|x^{\PGL_n(q)}|=|\hat{x}^{\GL_n(q)}|;$ or \label{prost1}
\item $r$ divides both $q-1$ and $n$, and $x$ is $\PGL_n(K)$-conjugate to the image of $$\diag[I_{n/r}, \omega I_{n/r}, \ldots , \omega^{r-1} I_{n/r} ],$$ where $\omega \in K$ is a primitive $r$-th root of unity. \label{prost2}
\end{enumerate} 
\end{Lem}

\begin{Rem}
Lemma 3.11 from \cite{fpr2} is formulated for   all classical groups, but only for $r \ne 2$. It is easy to see  from its proof  that the condition $r\ne 2$ is necessary only for orthogonal and symplectic cases; if $x \in \PGL_n(q)$ then the statement is true for all primes $r.$  
\end{Rem}

\begin{Lem}\label{xGoGs}
Let $x \in \overline{G}_{\sigma}$ have prime order.
\begin{enumerate}[font=\normalfont]
\item If $x$ is semisimple, then $x^{\overline{G}_{\sigma}}=x^{G_0}.$
\item If $x$ is unipotent and $G_0=\PSL_n(q)$, then $|x^{\overline{G}_{\sigma}}| \le n|x^{G_0}|.$
\end{enumerate} 
\end{Lem}
\begin{proof}
See \cite[4.2.2(j)]{gorlyo} for the proof of $(1)$ and \cite[Lemma 3.20]{fpr2} for $(2)$.
\end{proof}

\begin{Lem}\label{lemxGbound}
Let  $x \in G \cap\PGL_n(q)$ have prime order  and $s:=\nu(x).$ Then 
\begin{equation}\label{5uni}
 |x^G| > 
\begin{cases}
  \frac{1}{2n}  q^{ns} \ge \frac{1}{2n} q^{n^2/2}  & \text{ for } s \ge n/2;\\
  \frac{1}{2n} q^{2s(n-s)} \ge  
\frac{1}{2n}   q^{3n^2/8}  & \text{ for }  n/4 \le s < n/2.
\end{cases}
\end{equation}
\end{Lem}
\begin{proof}
The statement follows by combining Lemma \ref{xGoGs} with \cite[Propositions 3.22 and 3.36, Lemmas 3.34 and 3.38]{fpr2}.
\end{proof}

\subsection{Computations using {\sf GAP} and {\sc Magma}}
\label{gapsec}

We use the  computer algebra systems {\sf GAP} \cite{GAP4} and {\sc Magma} \cite{magma}  to check the inequality $b_S(S \cdot \SL_n(q))\le c$ for  particular cases where $S$ is a maximal solvable subgroup of $\GL_n(q)$ and $q$ and $n$ are small. In this section we discuss how the statement $b_S(G)\le c$ can be checked for given  $S \le \GL_n(q)$ and integer $c$.

In general, one can use the following algorithm to compute $b_S(G)$ in {\sc Magma} precisely: 
\begin{enumerate}
\item use \texttt{CosetAction} function to obtain the permutation representation of $G$ induced by the action of $G$ on the set of cosets of $S$;
\item compute the orbits of the point stabiliser $\overline{S}$ (the image of $S$ under the coset action homomorphism) using \texttt{Orbits} function;
\item take orbit representatives and their corresponding stabilisers (2-point stabilisers) in $\overline{S}$ (using \texttt{Representative} and \texttt{Stabiliser});
\item obtain the orbits for these 2-point stabilisers and repeat the above procedure until a $(c-1)$-point stabiliser has a regular orbit.
\end{enumerate}
The minimal $c$ such that a $(c-1)$-point stabiliser has a regular orbit is the base size $b_S(G).$  This procedure is effective for $|G:S|$ up to $5 \times 10^6$ when this code executes in {\sc Magma} V2.28-4 on a machine with a 2.6 GHz processor.
 In practice, we typically use this method to derive a lower bound on $b_S(G)$ and we only need to use it in a handful of cases with $|G:S|$ very small. And to obtain an upper bound on $b_S(G)$, we often use random search. That is to say, we seek elements $a_i \in G$  such that 
\begin{equation}
\label{gapint}
S \cap S^{a_1} \cap \ldots \cap S^{a_{c-1}} =S_G=\cap_{g \in G} S^g.
\end{equation}
Randomised search allows us to find such $a_i$ in all cases we consider.

A more difficult  task is to construct a set of representatives of the conjugacy classes of maximal solvable subgroups of $G$.  Generating sets for primitive maximal solvable subgroups of $\GL_n(q)$ for small $n$ (and for subgroups containing primitive maximal solvable subgroups of $\GL_n(q)$ in the general case) can be found in \cite{short} and \cite[\S 21]{sup}. The function \texttt{IrreducibleSolvableGroupMS}$(n,p,i)$ in the {\sf GAP} package \texttt{PrimGrp} \cite{PrimGrp} realises the results of \cite{short}. It returns a representative of the $i$-th conjugacy class of irreducible solvable subgroups of $\GL_n(p)$, where $n>1$, $p$ is a prime, and $p^n < 256$.   While constructing a specific subgroup can be difficult, it is usually not necessary. Indeed, \eqref{gapint} clearly holds if  
\begin{equation}
\label{gapintM}
M \cap M^{a_1} \cap \ldots \cap M^{a_{c-1}} =S_G,
\end{equation}
for any overgroup $M>S$, so it is enough to construct an overgroup $M$ of $S$ such that \eqref{gapintM} holds.   The function \texttt{ClassicalMaximals} in  {\sc Magma} realises the results of \cite{maxlow}. It  constructs all maximal subgroups of a classical group ($\GL_n(q)$ for example) up to conjugation for $n \le 17$. Since this function works with explicit generators of these subgroups, there are no practical restrictions on the value of $q$.  If a maximal subgroup $M$ is too big and \eqref{gapintM} does not hold, then we use the function \texttt{MaximalSubgroups} to obtain all maximal subgroups of $M$ up to conjugation and check if \eqref{gapintM} holds for them. In practice, at most three iterations are needed to obtain \eqref{gapintM} for some overgroup of the solvable subgroup under investigation.

When $S\le \GaL_n(q)$ and $S \not\le \GL_n(q)$, we consider $S$ as a subgroup of $\GL_{nf}(p)$ where $q=p^f$ and the methods above apply. We only use computations for $S \not\le \GaL_n(q)$ when $n \in \{3,4\}$ and $q \in \{3,4,5\}$ (see {Case 2} of the proof of Theorem \ref{theoremGR} in Section \ref{sec412}). Here we construct $\Aut(\PSL_n(q))$ in  {\sc Magma} using \texttt{AutomorphismGroupSimpleGroup} and then apply the method above using   \texttt{MaximalSubgroups}.

We often write ``$b_S(G) \le c$ is verified by computation''\index{computation} to imply that the statement is verified using one of the procedures described above.   
Sometimes we carry out similar calculations to show  that the intersection of a specific number of conjugates of $S$ satisfies particular order bounds.  We often write ``computation shows'' to summarise such routine calculations.

\section{Irreducible solvable subgroups}
\label{ch2}

Recall from the introduction that $b_S(G)$ is the minimal number such that there exist $$x_1, \ldots, x_{b_S(G)} \in G  \text{ with } S^{x_1} \cap \ldots \cap S^{x_{b_S(G)}}=S_G$$ where $S_G= \cap _{g \in G}S^g$. Let $S$ be an irreducible maximal solvable subgroup of $\GL_n(q)$.  The goal of this section is to obtain upper bounds for $b_S(S \cdot \SL_n(q))$. These bounds play an important role in the proof of Theorems \ref{theorem}, \ref{theoremGR}.  Although, with some exceptions, the bound  $b_S(S \cdot \SL_n(q)) \le 4$ follows from Theorem \ref{bernclass}, it is not sufficient for our purposes. In this section we prove that  $b_S(S \cdot \SL_n(q)) \le 2$ with a short list of exceptions.

\subsection{Primitive subgroups} \label{sec31}

    We start our study with a special case:  $S$ is a primitive maximal solvable subgroup  of $\GL_n(q)$ (recall the definition of a primitive linear group after Notation \ref{not}). In Section \ref{impsec} we use these results to obtain bounds for $b_S(S \cdot \SL_n(q))$ where $S$ is irreducible.  

To prove results about $b_S(S \cdot \SL_n(q))$, we need  information about  primitive solvable groups $S$ including an upper bounds on $|S|$ and lower bound for $\nu(x)$ (the codimension of a largest eigenspace of $x$, see Definition \ref{nudef}), where $x$ is a prime order element of the image of $S$ in $\PGL_n(q).$  This information is needed to apply the probabilistic method described in Section \ref{fprsec}.

The following theorem collects properties of primitive maximal solvable subgroups from \cite[\S \S 19 -- 20]{sup} and \cite[\S 2.5]{short}.  For $A \le \GL_m(q)$ and $B \le \GL_n(q)$, $A \otimes B = \{a \otimes b \mid a \in A, b \in B\}$ where $a \otimes b$ is the Kronecker product defined  in Notation \ref{not}.  

\begin{Th}\label{suplem}
Let $S \le \GL_n(q)$ be a primitive maximal solvable subgroup. Then $S$
admits a unique chain of subgroups 
\begin{equation}\label{ryad}
S \trianglerighteq C \trianglerighteq F \trianglerighteq A
\end{equation}
 where $A$ is the unique maximal abelian normal subgroup of $S$,  $C =C_S(A)$, and  $F$ is the full preimage in $C$ of the maximal  abelian normal subgroup of $S/A$ contained in $C/A.$ The following hold: 
\begin{enumerate}[font=\normalfont]
\item[$(a)$] $A$ is the multiplicative group of a field extension $K$ of the field of scalar matrices $ \Delta = \{\alpha I_n: \alpha \in \mathbb{F}_q\}$ and $m:=|K : \Delta|$ divides  $n$.
\item[$(b)$]  $S/C$ is isomorphic to a subgroup of the Galois group of the extension $K:\Delta,$ so 
$$S=C.\sigma,$$
where $\sigma$ is cyclic of order dividing $m$. 
\item[$(c)$] $C \le \GL_e(K)$ where $e=n/m$.
\item[$(d)$] If
\begin{equation}\label{r}
 e=p_1^{l_1} \cdot \ldots \cdot p_t^{l_t},
\end{equation}
where $t \in \mathbb{N},$ and the $p_i$ are distinct primes, then each $p_i$ divides $|A|=q^m-1.$
\item[$(e)$] $|F/A|=e^2$ and $F/A$ is the direct product of elementary abelian groups. 
\item[$(f)$] If $Q_i$ is the full preimage in $F$ of the Sylow $p_i$-subgroup of $F/A,$ then
$$Q_i= \langle u_1 \rangle \langle v_1 \rangle \ldots \langle u_{l_i} \rangle \langle v_{l_i} \rangle A,$$
where 
\begin{equation}\label{quv}
[u_j,v_j]=\eta_j, \text{ } \eta_j^{p_i}=1, \text{ } \eta_j \ne 1, \text{ } \eta_j \in A; \text{ } u_j^{p_i},v_j^{p_i} \in A
\end{equation}
and elements from distinct pairs $(u_j,v_j)$ commute.
\item[$(g)$] In a suitable $K$-basis of $K^e$, 
$$F=\tilde{Q_1} \otimes \ldots \otimes \tilde{Q_t},$$
where $\tilde{Q_i}$ is an absolutely irreducible subgroup of $\GL_{p_i^{l_i}}(K)$ isomorphic to $Q_i.$
\item[$(h)$] If $N=N_{\GL_e(K)}(F)$ and $N_i=N_{G_i}(\tilde{Q_i}),$ then
 $$N=N_1 \otimes \ldots \otimes N_t.$$
\item[$(i)$] $N_i/\tilde{Q_i}$ is isomorphic to a completely reducible subgroup of $\Sp_{2l_i}(p_i).$
\item[$(j)$] $F/A=C_{C/A}(F/A)$.

\end{enumerate}
\end{Th}

The following result is due to Gluck and Manz {\cite[Theorem 3.5]{manz}}.

\begin{Th}  \label{lowboundsol}
If $S$ is a completely reducible solvable  subgroup of $\GL_n(q)$, then $|S| < q^{9n/4}/2.8.$
\end{Th}

Let $S$ be a primitive maximal solvable  subgroup of $\GL_n(q)$. If $A,$ $F$ and $C$ are as in  \eqref{ryad}, then by Theorem \ref{suplem}
\begin{equation*}
\begin{split}
|A|&=q^m-1;\\
|F:A|&=(n/m)^2=e^2;\\
|S:C|&\le m.\\
\end{split}
\end{equation*}
By  Theorem \ref{suplem} (i) and Theorem \ref{lowboundsol}
\begin{equation}
\label{CFbound}
|C:F| \le 
\begin{cases}
\prod_{i=1}^t |\Sp_{2l_i}(p_i)|;\\
\prod_{i=1}^t ((p_i^{2l_i})^{9/4}/2.8)<e^{9/2},
\end{cases}
\end{equation}
where $e=n/m$, $p_i$ and $l_i$ are as in \eqref{r}. Notice that 
$$|\Sp_{2l_i}(p_i)|=p_i^{l_i^2} \prod_{j=1}^{l_i}(p_i^{2j}-1)\le p_i^{l_i^2} \prod_{j=1}^{l_i}p_i^{2j}=p_i^{l_i(2l_i+1)}. $$
Denote $\log_2(e)$ by $l$, so $l_i \le l$ for $i=1, \ldots, t.$ Therefore,
$$|C:F| \le \prod_{i=1}^t p_i^{l_i(2l+1)}=e^{2l+1} $$ 
and
\begin{equation*}
|S|\le (q^m-1)\frac{n^2}{m^2}m \cdot \min \{ e^{2l+1}, e^{9/2}\}=(q^m-1)m \cdot \min \{e^{2l+1}, e^{13/2}\},
\end{equation*}
so
\begin{equation}\label{H}
|S/Z(\GL_n(q))|\le \left( \frac{q^m-1}{q-1} \right) m\cdot \min \{e^{2l+1}, e^{13/2}\}.
\end{equation}

\begin{Lem}\label{6}
 Let $S$  be a primitive maximal  solvable subgroup of $\GL_n(q)$, $Z=Z(\GL_n(q))$ and let $H=S/Z \le \PGL_n(q).$ Let $x \in H$ and let $\hat{x}$ be a preimage of $x$ in $S$. If $x$ has prime order $r$, so  $\hat{x}$ lies in $S \backslash Z$, then 
the following hold:
\begin{enumerate}[font=\normalfont]
\item if $\hat{x}  \in F $, then $\nu(x) \ge n/2; $ \label{6it2}
\item if $\hat{x}  \in C \backslash F$, then $\nu(x) \ge n/4; $ \label{6it3}
\item if $\hat{x}  \in S \backslash C$, then $\nu(x)=n-n/k \ge n/2. $ \label{6it4}
\end{enumerate}
\end{Lem}
\begin{proof}
The proof follows the beginning of the proof of \cite[Proposition 4]{gluck}.  Let ${\mathbb{F}}$ be the algebraic closure of $\mathbb{F}_q.$

If $\hat{x} \in A$, then, since $A$ is cyclic and $(|A|,q)=1$, $\hat{x}$ is conjugate in $\GL_n({\mathbb{F}})$ to 
$$\diag(\lambda, \lambda^q, \ldots, \lambda^{q^{m_1-1}}, \ldots, \lambda, \lambda^q, \ldots, \lambda^{q^{m_1-1}}); \text{ } \lambda \in {\mathbb{F}^*}$$
by \cite[Lemma 1.3]{buturl}. Here $m_1$ is the smallest possible integer such that  $\lambda^{q^{m_1}}=\lambda.$ Therefore, $\nu(\hat{x})=n-n/m_1\ge n/2.$ Moreover, if $z \in A$ is nontrivial, then $C_V(z)=\{0\}.$

Let $\lambda \in \mathbb{F}^*.$ If $\hat{x} \in S \backslash C,$ then 
$$[\lambda \hat{x}, z]=[\hat{x},z] \in A \backslash \{1\}$$ for some $z \in A.$ Notice 
$$C_V((\lambda \hat{x})^{-1}) \cap C_V(z^{-1} \lambda \hat{x} z) \subseteq C_V([\lambda \hat{x}, z])=\{0\}$$
and $\dim C_V((\lambda \hat{x})^{-1})= \dim C_V((\lambda \hat{x})^{z}),$ so $ \dim  C_V((\lambda \hat{x})) \le n/2$ for every $\lambda \in \mathbb{F}^*$. Hence $\nu(\hat{x})\ge n/2.$
Moreover, if  \ref{prost2} holds in Lemma \ref{prost}, then $\nu(x)=n-n/r.$ 
Suppose \ref{prost1} holds in Lemma \ref{prost}, so $\hat{x} \in S$ has order $r.$
 If $\hat{x}  \in S \backslash C$ then, by $(b)$ of Theorem \ref{suplem} and \cite[7-2]{conjaut}, $\hat{x}$ is conjugate to a field automorphism $\sigma_1 \le \langle \sigma \rangle$ of $\GL_e(K)$ of order $r$ where $K$ is as in Theorem \ref{suplem}$(a)$.
Such an element acts as a permutation with $(n/r)$ $r$-cycles on a suitable basis of $V$, so  $\nu(x)=n-n/r$.

If $\hat{x} \in F \backslash A,$ then, by $(f)$ of Theorem \ref{suplem}, there exists $h \in F \backslash A$ such that $$[ \lambda \hat{x}, h]= [\hat{x},h] \in Z \backslash \{1\}$$ and $\nu(\hat{x}) \ge n/2$ as above.

If $\hat{x} \in C \backslash F$, then $[\lambda \hat{x}, h] \in F \backslash A$ for some $h \in F$ by $(j)$ of Theorem \ref{suplem}. Therefore, 
$$\dim C_V([\lambda \hat{x}, h]) \le n/2 \text{ and } \dim C_V(\lambda \hat{x}) \le 3n/4$$
for every $\lambda \in K^*$,
so $\nu(\hat{x}) \ge n/4.$
\end{proof}

\begin{Lem}\label{sinbase}
If $S$ is the normaliser of a Singer cycle of $\GL_n(q)$ with $n\ge 3$, then $b_{S}(S \cdot \SL_n(q)) \le 3$ with equality if and only if $(n,q)=(3,2).$
\end{Lem}
\begin{proof}
We follow \cite[Lemma 6.4]{burPS} which proves the statement of the lemma in the case of prime $n$. Let $\hat{G}$ be $S \cdot \SL_n(q),$ let  $G$ be $\hat{G}/Z(\hat{G})\le \PGL_n(q)$ and let $H$ be $S/Z(\hat{G})\le G.$   Obviously, 
$$b_S(\hat{G})=b_H(G).$$

Let $x\in H$ be an element of prime order. Since $S$ is the normaliser of a Singer cycle $A$, $A=F=C$ in the notation of Theorem \ref{suplem}, so $\nu(x) \ge n/2$ for all such $x \in H$ by Lemma \ref{6}.  Hence
$$|x^G| \ge \frac{1}{2n}q^{n^2/2}=b$$ 
for all $x \in H$ of prime order  by Lemma \ref{lemxGbound}. Let $$a:=|H|=n(q^n-1)/(q-1).$$
Recall $\widehat{Q}(G,c)$ from \eqref{ver} and that $b_H(G)\le c$ if  $\widehat{Q}(G,c)<1.$ By Lemma \ref{fprAB}, 
$$\widehat{Q}(G,2) \le  a^2/b.$$
 It is routine to check that, if $n \ge 4,$ then this upper bound is less than $1$ unless  either $(n,q)=(6,2),(5,3), (5,2)$, or $n=4$ and  $q<13.$ For these cases the lemma is verified by computations. If $n=3$, then the lemma follows directly from \cite[Lemma 6.4]{burPS}.
\end{proof}

\begin{Lem}\label{sch} 
 Let $n \ge 6.$ If $S$ is a primitive maximal solvable  subgroup of $\GL_n(q)$, then $$b_S(S \cdot \SL_n(q)) = 2.$$ 
\end{Lem}
\begin{proof}

Let   $G$ be $(S \cdot \SL_n(q))/Z(\GL_n(q))\le \PGL_n(q)$ and let $H$ be $S/Z(\GL_n(q))\le G$, so 
$$b_S(S \cdot \SL_n(q))=b_H(G).$$

If $n \ge 6$ and for all $x \in G$ of prime order   
$$|x^G \cap H| < |x^G|^{(3c-4)/(3c)},$$
then $b_H(G) \le c$ by  \eqref{fpr}  and \eqref{0}. Therefore,  it suffices to show this inequality for $c=2$.

Let $s:=\nu(x).$ We use the bounds in \eqref{5uni} for $|x^G|.$
In most cases the trivial bound $|x^G \cap H| \le |H|$ is sufficient. 

For $n \ge 16$, the inequality $|x^G \cap H| < |x^G|^{1/3}$ follows for all $x \in H$ or prime order by combining  Lemma \ref{6} and the bounds   in \eqref{5uni} for $|x^G|$ and \eqref{H} for $|x^G \cap H|$ respectively. And there is a finite list of cases with $6 <n \le 15$ where these bounds  are insufficient. By  Theorem \ref{suplem} $(d)$,   $(q^m-1)$ must be divisible by $p_i$ for all $p_i,$ $i=1, \ldots, t.$ Using this statement and the bound for $|C:F|$ obtained in \eqref{CFbound} by using the precise orders of the $\Sp_{2l_i}(p_i)$, we reduce this list to  cases 1--6 in Table~\ref{tab}. Here $e$ is as in Theorem \ref{suplem}. 
The lemma follows by Lemma \ref{sinbase} in case $1$ and it is verified by computation  in cases 2--6. 
\begin{table}[h]
\centering
\caption{Exceptional cases in proof of Lemma \ref{sch}}
\label{tab}
\begin{tabular}{|l|l|l|l|}
\hline
\textbf{Case} & $n$ & $e$ & $q$   \\ \hline
\textbf{1}  & 6   & 1   & any   \\ \hline
\textbf{2}  & 6   & 2   & 3     \\ \hline
\textbf{3}  & 6   & 3   & 2,4 \\ \hline
\textbf{4}  & 7   & 1   & 2,3,4 \\ \hline
\textbf{5}  & 8   & 1   & 2     \\ \hline
\textbf{6}  & 8   & 8   & 3,5    \\ \hline
\end{tabular}
\end{table} 
\end{proof}

Before the next lemma, we remind the reader of the following definitions.
A subgroup of $\GL_n(q)$ is \emph{absolutely irreducible} if it is irreducible as a subgroup of $\GL_n(\mathbb{F})$, where $\mathbb{F}$ is the algebraic closure of $\mathbb{F}_q.$ For a prime $r$, a finite $r$-group $R$ is of \emph{symplectic-type} if every characteristic abelian subgroup of $R$ is cyclic. Symplectic-type groups are closely related to extra-special groups, see \cite[\S 4.6]{kleidlieb} and \cite[\S 2.4]{short} for a summary and details.

\begin{Lem}
\label{c6small}
 Let $S$ be a primitive maximal solvable subgroup of $\GL_n(q).$  Recall that $q=p^f.$ If $e=n=r^l$ for some integer $l$ and prime $r$ in Theorem \ref{suplem}, then  $S=S_1 \cdot Z(\GL_n(q))$ where $S_1=S \cap \GL_n(p^t)$,   $t$ divides $f$, and $S_1$ lies in the normaliser $M$ in $\GL_n(p^t)$ of an absolutely irreducible symplectic-type subgroup $F_1$ of $\GL_n(p^t)$. Moreover, $t$ may be chosen so that $\mathbb{F}_{p^t}$ is the smallest field over which such a representation of $F_1$ can be realised.
\end{Lem}
\begin{proof}
Let $F$  be as in Theorem \ref{suplem}. Let $W \le V$ be an irreducible $\mathbb{F}_{q}[F]$-submodule. By  Theorem \ref{suplem} $(f)$, $F=A \cdot F_1$ where $F_1$ is extra-special of order $r^{2l+1},$ so $W$ is a faithful irreducible $\mathbb{F}_{q}[F_1]$-module. Therefore, by \cite[Propositions 4.6.2 and 4.6.3]{kleidlieb}, $\dim W=r^{l}$, and $F_1$ is an absolutely irreducible symplectic-type subgroup of $ \GL_n(p^t)$  where $\mathbb{F}_{p^t}$ is the smallest field (of characteristic $p$) over which such a representation of $F_1$ can be realised. In particular, $t$ is the smallest positive integer for which $p^t \equiv 1 \mod |Z(F_1)|.$ Moreover, by \cite[Corollary 2.4.12]{short}, $|Z(F_1)|$ is either $r$ or $4$.  By  \cite[Theorem 2.4.12]{short}, $S=S_1 \cdot Z(\GL_n(q))$ where $S_1 \le N_{\GL_n(p^t)}(F_1)=M.$      
\end{proof}

Let $S$ be a  primitive maximal solvable subgroup of $\GL_n(q)$ with $2 \le n \le 5.$ The equation  $b_S(S \cdot \SL_n(q))=2$ does not always hold for such $n$. However, in view of Lemmas 
\ref{irrtog}, \ref{diag}   and \ref{prtoirr} (see below), if for every primitive maximal  solvable subgroup $S$ there exists $x \in \SL_n(q)$  such that
\begin{equation*}
 S \cap S^x \le RT(\GL_n(q)),
\end{equation*}
 then $b_M(M \cdot \SL_n(q)) \le 5$ for every maximal solvable subgroup $M$.  Recall that $D(\GL_n(q))$ and $RT(\GL_n(q))$ denote the subgroups of all diagonal and all upper-triangular matrices in $\GL_n(q)$ respectively. Therefore, if $b_S(S \cdot \SL_n(q))>2$, then we decide if there exists $x \in \SL_n(q)$ such that  $S \cap S^x$ lies in $D(\GL_n(q))$ or $RT(\GL_n(q)).$ In particular, we prove in this section that if $q>7$, then such an $x$ always exists, so  $$b_M(M \cdot \SL_n(q)) \le 5$$
for every maximal solvable subgroup $M$ of $\GL_n(q)$, for every $n \ge 2$; see Corollary \ref{corir} and Theorem \ref{irred}. We also use this observation in the proof of Theorem \ref{theorem} in Section \ref{sec411}.

\medskip

The following theorem is the main result of this section.

\begin{Th}
\label{sec311lem}
 Let $S$ be a primitive maximal solvable subgroup of $\GL_n(q)$ where $n \ge 2$ and  $(n,q) \ne (2,2), (2,3).$
Then either $b_S(S \cdot \SL_n(q))=2$, or one of the following holds:
\begin{enumerate}[font=\normalfont]
\item $n =2$,  $q>3$ is odd,  $S$ is the normaliser of a Singer cycle and  $b_S(S \cdot \SL_2(q))=3$. If $q>5$, then there exists $x\in \SL_2(q)$ such that  $S \cap S^x \le D(\GL_2(q));$ \label{thirr2odd}
\item $n =2$,  $q\ge 4$ is even,  $S$ is the normaliser of a Singer cycle and  $b_S(S \cdot \SL_2(q))=3$. In this case there exists $x\in \SL_2(q)$ such that  $S \cap S^x \le RT(\GL_2(q));$ \label{thirr2even}
\item $n=2,$ $q=9,$ $S$ is generated by all matrices with entries in $\mathbb{F}_3$ and scalar matrices, so $S=\GL_2(3) \cdot Z(\GL_2(9)),$ and  $b_S(S \cdot \SL_2(9))=3$.  In this case there exists $x\in \SL_2(9)$ such that  $S \cap S^x \le RT(\GL_2(9));$ \label{thirr29}
\item $n=2$, $q \in \{5,7\}$,  $S$  is an absolutely irreducible subgroup such that $S/Z(\GL_2(q))$ is isomorphic to $2^2.\Sp_2(2)$, and  $b_S(S \cdot \SL_2(q))=3+\delta_{5,q}$;
\item $n=3$, $q=2$,   $S$ is the normaliser of a Singer cycle and  $b_S(S \cdot \SL_3(2))=3$. \label{thirr32} 
\end{enumerate}  
\end{Th}
\begin{proof}
For $n \ge 6$ the statement holds by Lemma \ref{sch}. Assume $2 \le n \le 5$. First, we prove the base size results, and then we establish the additional assertions in \ref{thirr2odd} and \ref{thirr2even}. Parts \ref{thirr29}--\ref{thirr32} can be checked computationally (see Section \ref{gapsec}).

Let $m$ and $e$ be as in Theorem \ref{suplem}. Notice that, for $n=2,3,5$, either $m=n$ or $m=1,$ since $n$ is prime. If $n=4,$ then there is the additional  case $m=e=2.$

\medskip

{\bf Case 1:} $m=n.$

 If $m=n$, then $S$ is the normaliser of a Singer cycle and the statement of the theorem on $b_S(S \cdot \SL_n(q))$ follows by Lemma \ref{sinbase} for $n \ge 3$ and by \cite[Proposition 4.1]{burPS} for $n=2$.

\medskip

{\bf Case 2:} $m=1.$ 

Here  Lemma \ref{c6small} holds, since $n$ is either prime or $4=2^2$. Hence  $S= S_1 \cdot Z(\GL_2(q))$ where $S_1$ lies in the normaliser  $M$ of an absolutely irreducible symplectic-type subgroup of $\GL_n(p^t)$ where $t$ is minimal such that $p^t \equiv 1 \mod |Z(F_1)|$ and $F_1$ is as in the proof of Lemma \ref{c6small}. If  $t=1$, then $M$ (up to scalars) is a maximal subgroup of either $\GL_n(p)$(or $\SL_n(p)$) or $\Sp_{n}(p)$ (if $n$ is even) from the Aschbacher class $\mathcal{C}_6$ by \cite[Table 4.6.B]{kleidlieb}. Bounds on base sizes for primitive actions of  classical groups with $n \le 5$ are listed in \cite[Tables 2 and 3]{burness}. In particular, it follows that there exists $g \in \SL_n(p)$ such that $M \cap M^g \le Z(\GL_n(q))$ unless $n=2,4$ and $p \le 7.$ Hence  $S \cap S^g \le Z(\GL_n(q))$ and  $b_S(S \cdot \SL_n(q))=2$.  The results for $n=2,4$ and $q \le 7$ are verified by computation.

For $t >1$ we prove the statement directly using the probabilistic method. Let   $G$  be $(S \cdot \SL_n(q))/Z(\GL_n(q))\le \PGL_n(q)$ and let $H$ be $S/Z(\GL_n(q))\le G$, so 
$$b_S(S \cdot \SL_n(q))=b_H(G).$$ 
 By Theorem \ref{suplem}, $H$ lies in a subgroup of $\PGL_n(q)$ isomorphic to $e^{2}.\Sp_{2l}(p_1^l)$, where $e=p_1^l$ for a prime $p_1$. The number of elements of order equal to each  prime dividing this subgroup can be computed directly. Further,
  using bounds on $|x^G|$ for $x$ of prime order from Lemma \ref{lemxGbound} (where we use bounds on $\nu(x)$ from Lemma \ref{6}), we obtain an upper bound for $\widehat{Q}(G,2)$ via Lemma \ref{fprAB}. For example, if $n=5$, then $H \le 5^2:\Sp_2(5)=\mathrm{ASL}_2(5)$ and 
  there are   $25$ elements of order $2$, $500$ elements of order $3$ and $624$ elements of order $5$, so 
$$\widehat{Q}(G,2)\le \frac{(25^2+500^2+624^2)}{(1/10)q^{12}}.$$ 
Here we use the fact that $\nu(x)\ge 2$ for all $x \in H$ of prime order by Lemma \ref{6}. This upper bound is less than $1$ for $q$ greater than $17,$ $17,$ $19$ and $7$ for $n=2,3,4$ and $5$ respectively. For smaller $q$, the value of $b_H(G)$ is verified using computations. In particular, if $(n,q)=(2,9)$ then \ref{thirr29} holds.   
 
\medskip 
 
{\bf Case 3:} $n=4,m=2.$

 Since $e=n/m=2$, $q$ is odd by Theorem \ref{suplem}$(d)$. Recall from Theorem \ref{suplem} that $S$ has a chain of normal subgroups:
$$S \ge C \ge F \ge A,$$
where $C \le \GL_2(q^2),$ so $S \le \GL_2(q^2).2,$ and $C/A$ is isomorphic to a subgroup of $2^2.\Sp_2(2).$ 
Therefore, $|S| $ divides $ 48(q^2-1)$ and $|H|$ divides $a=48(q+1).$ 
 
 For all elements of prime order in $H$,
 $$|x^G| \ge (1/8)q^6=b$$ by Lemma \ref{lemxGbound}. Hence, by Lemma \ref{fprAB},
 $\widehat{Q}(G,2) \le  a^2/b$.  This upper bound is less than $1$ if  $q \ge 13.$
If $q \le 11$, then  $b_S(S \cdot \SL_4(q))=2$  is verified  by computation.

This completes the base size computations.

\medskip



 Let us prove the remaining claims in parts $(1)$ and $(2)$.
 First assume $q$ is odd and fix a non-square $a \in \mathbb{F}_q$ (there are $(q-1)/2$ such elements). Consider 
\begin{equation}\label{thesin}
 S_a= \left\{
\alpha \begin{pmatrix}
1    & 0  \\
      0    &   1       
\end{pmatrix} + 
\beta \begin{pmatrix}
0    & 1  \\
a    &  0       
\end{pmatrix}
\mid \alpha, \beta \in \mathbb{F}_q \right\} \backslash \{0\}.
\end{equation}
 Notice that for an element of $S_a$
$$\det \begin{pmatrix}
\alpha    & \beta  \\
    a  \beta    & \alpha       
\end{pmatrix} =0$$ if, and only if, $a=(\alpha/ \beta)^2,$ so all matrices in $S_a$ are invertible. 
 Calculations show that $S_a$ is an abelian subgroup of $\GL_2(q)$ of order $q^2-1$ and $S_a \cup \{0\}$ under usual matrix addition and multiplication is a field, so $S_a$ is a Singer cycle. Notice that 
$$\varphi = \begin{pmatrix}
-1    & 0  \\
0    &  1       
\end{pmatrix}$$
normalises $S_a.$ Therefore, we can view $S$ as   $N_{\GL_2(q)}(S_a)=S_a \rtimes \langle \varphi \rangle$. 

 It is easy to see that if there exist $a, b \in \mathbb{F}_q$ such that $a \ne b$
and neither $a$ nor $b$ has square roots in $\mathbb{F}_q$, then $$S_a \cap S_b \le Z(\GL_2(q)).$$ 
If, in addition, $a \ne -b$, then the direct comparison of the matrices shows that 
\begin{equation}
\label{2sindiagodd}
 N_{\GL_2(q)}(S_a) \cap N_{\GL_2(q)}(S_b) =\langle \varphi \rangle Z(\GL_2(q)) \le D(\GL_2(q)).
\end{equation} 
It is possible to find such $a$ and $b$ if $q>5,$ so in this case there exists $x \in \SL_2(q)$ such that $$S \cap S^x \le D(\GL_2(q)),$$ since all Singer cycles are conjugate in $\GL_2(q)$ and $\Det(S_a)=\Det(\GL_2(q)).$

Finally, suppose $q$ is even and let $a \in \mathbb{F}_q $ be such that there are no roots of $x^2+x+a$ in $\mathbb{F}_q$ (there are $q/2$ such elements in $\mathbb{F}_q$). 
 Consider  
\begin{equation} \label{thesineven}
S_a= \left\{
\alpha \begin{pmatrix}
1    & 0  \\
      0    &   1       
\end{pmatrix} + 
\beta \begin{pmatrix}
0    & 1  \\
a    &  1       
\end{pmatrix} \mid \alpha, \beta \in \mathbb{F}_q
\right\} \backslash \{0\}.
\end{equation} Notice that  for an element of $S_a$
$$\det \begin{pmatrix}
\alpha    & \beta  \\
    a  \beta    & \alpha +\beta      
\end{pmatrix} =0$$ if, and only if, $a=(\alpha/ \beta)^2 +(\alpha/\beta),$ so all matrices in $S_a$ are invertible. 
 Calculations show that $S_a$ is an abelian subgroup of $\GL_2(q)$ of order $q^2-1$ and $S_a \cup \{0\}$ under usual matrix addition and multiplication is a field, so $S_a$ is a Singer cycle. Notice that 
$$\varphi = \begin{pmatrix}
1    & 0  \\
1    &  1       
\end{pmatrix}$$
normalises $S_a.$  Therefore, we can view $S$ as   $N_{\GL_2(q)}(S_a)=S_a \rtimes \langle \varphi \rangle$.

 It is easy to see that if there exist $a, b \in \mathbb{F}_q$ such that $a \ne b$
and neither $x^2+x+a$ nor $x^2+x+b$ has roots in $\mathbb{F}_q$, then $$S_a \cap S_b \le Z(\GL_2(q)).$$
Calculations show that  
\begin{equation}
\label{2sindiageven}
N_{\GL_2(q)}(S_a) \cap N_{\GL_2(q)}(S_b)=\langle \varphi \rangle Z(\GL_2(q))
\end{equation} consists of  lower triangular matrices. It is possible to find such $a$ and $b$ if $q \ge 4,$ so in this case there exists $x \in \SL_2(q)$ such that all matrices in $S \cap S^x$ are  lower triangular,   since all Singer cycles are conjugate in $\GL_2(q)$ and $\Det(S_a)=\Det(\GL_2(q)).$
\end{proof}

\subsection{Imprimitive irreducible subgroups}\label{impsec}

We commence by obtaining  results about the  subgroups of wreath products in $\GL_n(q).$ We combine these results with those of the previous section to obtain 
an upper bound on $b_S(S \cdot \SL_n(q))$ for imprimitive maximal solvable subgroups $S$ of $\GL_n(q).$   

\medskip

Let $A(n)$  be the following $n \times n$ matrix:
\begin{equation}\label{igrek}
\begin{aligned}
A(n) & = \begin{pmatrix}
1      & -1     & 1      & -1     & \ldots & (-1)^{n+1}         \\
0      &   1    & -1     & 1      & \ldots &  (-1)^{n}         \\
0      &    0   & 1      & -1     & \ldots &  (-1)^{n-1}        \\
       &        &        & \ddots &\ddots  &          \\
       &        &        &        & 1      & -1        \\
       &        &        &        & 0      &  1      
\end{pmatrix}, \text{ so } \\
A(n)^{-1} & =
\begin{pmatrix}
1      & 1      & 0      & 0      &        &          \\
0      &   1    & 1      & 0      &        &           \\
0      &    0   & 1      & 1      &        &          \\
       &        &        & \ddots &\ddots  &          \\
       &        &        &        & 1      & 1        \\
       &        &        &        & 0      &  1      
\end{pmatrix}.
\end{aligned}
\end{equation}

\begin{Lem}\label{igrekl}
Let $H \le X \wr Y,$ where $X \le \GL_m(q),$ $Y \le \Sym(k).$  
Let $A(k)=(y_{ij})$ be as in \eqref{igrek} and let $x_i$ for $i =1, \ldots, k$ be arbitrary elements of $X$.
Define $x \in \GL_{mk}(q)$ to be
\begin{gather*}
\diag(x_1, \ldots, x_k) (I_m \otimes A(k))=
\begin{pmatrix}
y_{11}x_1      & y_{12}x_1      & \ldots     & y_{1k}x_1            \\
y_{21}x_2      &  y_{22}x_2     & \ldots     & y_{2k}x_2             \\
\vdots      &             &            & \vdots               \\
y_{k 1}x_{k}   & y_{k 2}x_{k}   &  \ldots    & y_{k k}x_{k}               
\end{pmatrix}.
\end{gather*} 
 Let $h=\diag[D_1, \ldots, D_k] \cdot s  \in H$, where $D_i \in X$ and $s \in Y,$ so $h$ is obtained from the permutation matrix $s$ by replacing $1$ in the $j$-th row by the $(m \times m)$ matrix $D_j$ for $j =1, \ldots, k$ and replacing each zero by an $(m\times m)$ zero matrix. 
If $h^x \in H$, then $s$ is trivial and $D_j^{x_j}=D_{j+1}^{x_{j+1}}$ for $j=1, \ldots, k-1.$ 
\end{Lem}
\begin{proof}
   If $s$ does not stabilise the point $k$, then there is more than one non-zero $(m \times m)$-block in the last $(m \times m)$-row of $h^{x}$ and, thus, $h^{x}$ does not lie in   $ X \wr Y.$
Assume that $s$ stabilises the last ${k-j}$ points and $(j)s=i \le j.$ The  $j$-th $(m \times m)$-row of $h^{x}$ is 
\begin{equation*}	
\begin{split}
\bigl( \overbrace{0, \ldots, 0}^{i-1},  \overbrace{{x_j}^{-1}(D_{j}){x_i},-{x_j}^{-1}(D_{j}){x_i},  \ldots,(-1)^{j-i} {x_j}^{-1}(D_{j}){x_i}}^{j+1-i}, \ldots \bigr).
\end{split}
\end{equation*}
Therefore, $h^x$ does not lie in $X \wr Y$ if $i \ne j,$ since the $j$-th  $(m \times m)$-row contains more than one non-zero $(m \times m)$-block in that case. So $i=j$ and the  $j$-th $(m \times m)$-row of $h^{x}$ is
\begin{equation*}
\begin{split}
\bigl( \overbrace{0, \ldots, 0}^{i-1},  (D_{j})^{x_j},   
( ( D_{j+1})^{x_{j+1}} -(D_{j})^{x_j}  ) ,  \ldots, (-1)^{k-j+1}((D_{j+1})^{x_{j+1}} - (D_{j})^{x_j}  ) \bigr).
\end{split}
\end{equation*}
 So, if $h^x \in H$, then $s$ stabilises $j$ and $D_{j}^{{x_j}}=D_{j+1}^{x_{j+1}}.$
\end{proof}


\begin{Lem}\label{prtoirr}
Let $T(\GL_n(q))$ be one of the following subgroups: $Z(\GL_n(q))$, $D(\GL_n(q))$ or $RT(\GL_n(q))$. Let $H$ be a subgroup of $\GL_n(q)$ such that
$$H \le H_1 \wr \Sym(k), $$
where $n=mk$, $H_1 \le \GL_m(q)$. If there exist $g_1, \ldots, g_b \in \GL_m(q)$ ({respectively } $\SL_m(q))$ such that 
$$H_1 \cap H_1^{g_1} \cap \ldots \cap H_1^{g_b} \le T(\GL_m(q)),$$
then there exist   $x_1, \ldots, x_b \in \GL_n(q)$ ({respectively} $\SL_n(q))$ such that 
\begin{equation}\label{irreq}
H \cap H^{x_1} \cap \ldots \cap H^{x_b} \le T(\GL_n(q)).
\end{equation}
\end{Lem}
\begin{proof}
Define $x_i$ to be 
\begin{gather*}
\begin{pmatrix}
y_{11}g_{i}      & y_{1 2}g_{i }      & \ldots     & y_{1 k}g_{i }            \\
y_{21}g_{i}      &  y_{22}g_{i}     & \ldots     & y_{2k}g_{i}             \\
\vdots      &             &            & \vdots               \\
y_{k 1}g_{i }   & y_{k 2}g_{i}   &  \ldots    & y_{kk}g_{i}               
\end{pmatrix},
\end{gather*} 
 where $y=(y_{ij})=A(n).$ Let us show that \eqref{irreq} holds for such $x_i.$  Let  $h=\diag[D_1, \ldots, D_k] \cdot s  \in H$, where $D_i \in H_1$ and $s \in \Sym(k).$
If $$h \in H \cap H^{x_1} \cap \ldots \cap H^{x_b}$$  then, by  Lemma \ref{igrekl}, $s$ is trivial and $D_i=D_j$ for all $1 \le i, j \le {k}$. Thus, $$D_i \in H_1 \cap H_1^{g_1} \cap \ldots \cap H_1^{g_b} \le T(\GL_m(q))$$
and $h \in  T(\GL_n(q)).$ 

Notice that $\det(y)=1$ and if $\det(g_i)=1,$ then 
\begin{equation*}
\det(x_i)=\det(g_i \otimes y)=\det(g_i)^k \cdot \det({y})^m=1. \qedhere
\end{equation*} 
\end{proof}

\begin{Cor}\label{corir}
Let $S$ be a maximal solvable subgroup of $\GL_n(q)$ and assume that matrices in $S$ have shape \eqref{stup}; so $S_i =\gamma_i(S)= P_{i} \wr \Gamma_i$,
where $P_{i}$ is a primitive solvable subgroup of $\GL_{m_i}(q)$, $\Gamma_i$ is a transitive solvable
subgroup of the symmetric group $\Sym(k_i)$, and $k_im_i = n_i$. If, for every $P_{i}$, there exist $x_i \in \SL_{m_i}(q)$ such that
\begin{equation*}
P_{i} \cap P_{i}^{x_i}\le RT(\GL_{m_i}(q)),
\end{equation*}
then $b_S(S \cdot \SL_n(q))\le5.$
\end{Cor}
\begin{proof}
The statement follows  from Lemmas \ref{irrtog},  \ref{prtoirr} and  \ref{diag}.
\end{proof}


\begin{Th}\label{irred}
Let $S$ be an irreducible maximal solvable  subgroup of $\GL_n(q)$ with $n \ge 2$, and $(n,q)$ is neither  $(2,2)$ nor $(2,3)$. Then either $b_S(S \cdot \SL_n(q))=2$ or one of the following holds: 
\begin{enumerate}[font=\normalfont]
\item $n =2$,  $q>3$ is odd, $S$ is the normaliser of a Singer cycle and $b_S(S \cdot \SL_2(q))=3$. If $q>5$, then there exists $x\in \SL_2(q)$ such that  $S \cap S^x \le D(\GL_2(q));$ \label{irred11}
\item $n =2$,  $q\ge 4$ is even,  $S$ is the normaliser of a Singer cycle and $b_S(S \cdot \SL_2(q))=3$. In this case there exists $x\in \SL_2(q)$ such that  $S \cap S^x \le RT(\GL_2(q));$ \label{irred12}
\item $n=2,$ $q=9,$ $S$ (up to conjugacy) is generated by all matrices with entries in $\mathbb{F}_3$ and scalar matrices, so $S=\GL_2(3) \cdot Z(\GL_2(9)),$ and $b_S(S \cdot \SL_2(9))=3$. In this case there exists $x \in \SL_2(9)$ such that $S \cap S^x \le RT(\GL_2(9));$ \label{irred135}
\item $n=2$, $q \in \{5,7\}$,  $S$  is an absolutely irreducible subgroup such that $S/Z(\GL_2(q))$ is isomorphic to $2^2.\Sp_2(2)$ and    $b_S(S \cdot \SL_2(q))=3+\delta_{5,q}$;\label{irred13}
\item $n=3$, $q=2$,  $S$ is the normaliser of a Singer cycle and $b_S(S \cdot \SL_3(2))=3$; \label{irred14} 
\item $n=4$, $q=3$, $S=\GL_2(3) \wr \Sym(2)$ and $b_S(S \cdot \SL_4(3))=3$. In this case there exists $x \in \SL_4(3)$ such that $S \cap S^x \le RT(\GL_4(3)).$ \label{irred15} 
\end{enumerate} 
\end{Th}
\begin{proof}
Let $S$ be an  irreducible maximal solvable subgroup of $\GL_n(q).$ The statement  follows by Lemmas \ref{supirr} and \ref{prtoirr} and Theorem \ref{sec311lem}
 for all cases except the case when $S$ is  conjugate to $S_1 \wr \Gamma$, where 
$n=kl$, $\Gamma$ is a transitive maximal  solvable subgroup of $\Sym(l)$, and $S_1$ is  one of the following groups:
\begin{enumerate}[label=\alph*)]
\item $k=3, q=2$ and $S_1$ is the normaliser of a Singer cycle of $\GL_3(2);$ \label{irred1}
\item $k=2$, $q=2,3$ and $S_1=\GL_2(q);$ 
\item $k=2$, $q \in \{5,7\}$ and $S_1/Z(\GL_2(q))$ is an absolutely irreducible subgroup such that $S_1/Z(\GL_2(q))$ is isomorphic to $2^2.\Sp_2(2)$;
\item $k=2$, $q=9$ and $S_1$ is $\GL_2(3) \cdot Z(\GL_2(9));$ \label{irred4}
\item $k=2$, $q>3$  and $S_1$ is the normaliser of a Singer cycle of $\GL_2(q);$ \label{irred6}

\end{enumerate}
If $l \ge 3$, then $b_S(S \cdot \SL_n(q))=2$ by \cite[Theorem 3.1]{james}. 

We verify by computation that $b_S(S \cdot \SL_n(q))=2$ for $l=2$ for all cases \ref{irred1} -- \ref{irred4} except  $S_1=\GL_2(3)$. If $S_1=\GL_2(3)$, then there exists $x \in \SL_4(3)$ such that $S \cap S^x \le RT(\GL_4(3))$. 

Consider  case \ref{irred6} with $q$ odd: so we assume $S=S_1 \wr \Sym(2)\ \le \GL_4(q)$, where $S_1$ is the normaliser of a Singer cycle $S_a$ in $\GL_2(q)$  as in \eqref{thesin}. Hence $S \cdot \SL_4(q)=\GL_4(q)$ by Lemma \ref{supSL} since the determinant of a generator of a Singer cycle generates $\mathbb{F}_q^*$. 

Let $$s=
\begin{pmatrix}
0 & 0& 1 & 0\\
0 & 0& 0 & 1\\
1 & 0& 0 & 0\\
0 & 1& 0 & 0
\end{pmatrix},
$$  
so  $g \in S$ has shape 
$$
\begin{pmatrix}
\alpha_1' & \beta_1& 0 & 0\\
a \beta_1' &  \alpha_1 & 0 & 0 \\
0 & 0& \alpha_2' & \beta_2\\
0 & 0& a \beta_2' & \alpha_2
\end{pmatrix} \cdot s^i,
$$
where $i =0,1,$ $\alpha_j' = \pm \alpha_j$, $\beta_j' = \pm \beta_j$ and $\alpha_j'\beta_j' = \alpha_j \beta_j.$ Consider $g^y$ where $y=A(n)$ is as in \eqref{igrek}. First let $i=0,$ so 
$$g^y = \Scale[0.95]{
\begin{pmatrix}
\alpha_1' +a \beta_1' &\alpha_1+ \beta_1 - \alpha_1' -a\beta_1' & \alpha_1' +a\beta_1'- \alpha_1 - \beta_1 & \alpha_1 + \beta_1 -\alpha_1' -a\beta_1'\\
a \beta_1' &  \alpha_1 -a\beta_1' & a\beta_1' - \alpha_1 + \alpha_2' & \alpha_1 + \beta_2 -a\beta_1' - \alpha_2' \\
0 & 0& \alpha_2' +a\beta_2' & \alpha_2+\beta_2-\alpha_2' -a\beta_2'\\
0 & 0& a \beta_2' & \alpha_2 - a\beta_2'
\end{pmatrix}}.
$$ 
Assume that $g^y \in S,$ so ${g^y}_{1,3}={g^y}_{1,4}={g^y}_{2,3}={g^y}_{2,4}=0.$ Thus, $0={g^y}_{2,3}+{g^y}_{2,4}=\beta_2.$ Also, ${g^y}_{1,2}= - {g^y}_{1,3} = 0$, but, since the left upper $(2 \times 2)$ block must lie in $N_{\GL_2(q)}(S_a)$, 
$$({g^y}_{1,2})a = \pm {g^y}_{2,1},$$
so $0 ={g^y}_{2,1} = a\beta_1'$ and $\beta_1=0.$ Therefore, ${g^y}_{1,2}= \alpha_1 -\alpha_1',$ so
$$\alpha_1 =\alpha_1'.$$ Now, since $\beta_2=0,$ ${g^y}_{3,4} = \pm ({g^y}_{4,3}) a^{-1}=0,$  
$$\alpha_2=\alpha_2'$$
and since ${g^y}_{2,4}=0$ we obtain $\alpha_1=\alpha_2,$ so $g^y$ is a scalar.  

Now let $i=1$, so 
$$g^y = \Scale[0.95]{
\begin{pmatrix}
0 & 0 & \alpha_1' + a\beta_1'& \alpha_1 +\beta_1 - \alpha_1' - a\beta_1'\\
\alpha_2' &  \beta_2'-\alpha_2' & a\beta_1' - \beta_2' + \alpha_2' & \alpha_1 - a\beta_1' +a\beta_2 - \alpha_2' \\
a \beta_2' + \alpha_2' & \alpha_2 + \beta_2 -a\beta_2' -\alpha_2'& a\beta_2' + \alpha_2 - \alpha_2 -\beta_2 & \alpha_2+ \beta_2 - a\beta_2' \alpha_2'\\
a\beta_2' & \alpha_2 -a \beta_2'& a\beta_2' - \alpha_2 & \alpha_2 - a\beta_2'
\end{pmatrix}}.
$$ 
If $g^y \in  S$, then ${g^y}_{2,1}={g^y}_{2,2}=0.$ So $\alpha_2 = \beta_2= 0$ which  contradicts the invertibility of $g.$
Therefore, $S \cap S^y  \le Z(\GL_4(q)).$ 

Consider  case \ref{irred6} with $q$ even:  so we  assume $S=S_1 \wr \Sym(2)\ \le \GL_4(q)$, where $S_1$ is the normaliser of a Singer cycle $S_a$ in $\GL_2(q)$ as in \eqref{thesineven}. Since $q>2$, we can choose $a\ne 1$ and arguments similar to those in case \ref{irred6} show that \[S \cap S^y \le Z(\GL_n(q)). \qedhere\] 
\end{proof}

We conclude the section with a technical lemma about the normalisers in $\GaL_n(q)$ of some specific solvable subgroups of $\GL_n(q)$. This will be useful later.

\begin{Lem} \label{lemRTdiagfield} \phantom{sdfs}
\begin{enumerate}[font=\normalfont]
\item Let $q=p^f$ with $f>1$ and let $S$ be the normaliser in $\GaL_2(q)$ of the Singer cycle $S_a$ defined in \eqref{thesin} and \eqref{thesineven} for odd and even $q$ respectively. Then there exists $x \in \SL_2(q)$ such that $$S \cap S^x \le RT(\GL_2(q)).$$ \label{lrtdf1}
\item Let $(n,q)=(2,9)$ and let $S$ be $\GL_2(3) \cdot Z(\GL_2(9)) \rtimes \langle \phi \rangle$ where $\phi$ is the standard field automorphism as in \eqref{defphibet}. Then there exist $x, y \in \SL_2(9)$ such that 
$$S \cap S^x \le RT(\GL_2(9))\rtimes \langle \phi \rangle \text{ and } (S \cap S^x) \cap (S \cap S^x)^y \le Z(\GL_2(9)) \rtimes \langle \phi \rangle.$$ \label{lrtdf2}
\item Let $S$ be the normaliser of the Singer cycle of $\GL_3(2)$ generated by the matrix 
$$ \begin{pmatrix}
0    & 0 & 1  \\
1 &0    & 0  \\
  0&    1    &   1      
\end{pmatrix}.$$ Then there exists $x \in \SL_3(2)$ such that 
$$ S\cap S^x = \left\langle
 \begin{pmatrix}
1    & 1 & 1  \\
0 &1    & 0  \\
  1&    0    &   0      
\end{pmatrix} 
\right\rangle. $$ \label{lrtdf3}
\end{enumerate}
\end{Lem}
\begin{proof}
\ref{lrtdf1} Let $q$ be odd. It is routine to check that 
$$N_{\GaL_2(q)}(S_a) = S_a \rtimes \left\langle \phi \begin{pmatrix}
a^{-(p-1)/2}    & 0  \\
0    &  1       
\end{pmatrix} \right\rangle$$ 
where $\phi : \lambda v_i \mapsto \lambda^p v_i$ for $\lambda \in \mathbb{F}_q$ and $\{v_1, v_2\}$ is the basis of $V$ with respect to which matrices from $S_a$ have shape \eqref{thesin}. Since $f>1$, we have $q>5$ and there exist non-squares $a, b \in \mathbb{F}_q$ such that $a \ne \pm b$. For such $a,b$, the direct comparison of the elements shows that  
$$N_{\GaL_2(q)}(S_a) \cap N_{\GaL_2(q)}(S_b)=\langle \varphi \rangle Z(\GL_2(q)) \le D(\GL_2(q)).$$

\medskip

Now assume $q$ is even.  It is routine to check that  $N_{\GaL_2(q)}(S_a)$ is solvable and 
$$N_{\GaL_2(q)}(S_a)= (S_a) \rtimes \left\langle \phi \begin{pmatrix}
a^{-1}    & 0  \\
1    &  a^1       
\end{pmatrix} \right\rangle$$ 
where $\phi : \lambda v_i \mapsto \lambda^p v_i$ for $\lambda \in \mathbb{F}_q$ and $\{v_1, v_2\}$ is the basis of $V$ with respect  to which matrices from $S_a$ have shape \eqref{thesineven}. Since $q \ge 4,$ there exist $a, b \in \mathbb{F}_q$ such that $a \ne b$
and neither $x^2+x+a$ nor $x^2+x+b$ has roots in $\mathbb{F}_q$. 
For such $a,b$, it is easy to check that   $$N_{\GaL_2(q)}(S_a) \cap N_{\GaL_2(q)}(S_b)=\langle \varphi \rangle Z(\GL_2(q))\le RT(\GL_2(q)).$$

\medskip

\ref{lrtdf2}
Let $\omega$ be a primitive element of $\mathbb{F}_9$ such that $\omega^2-\omega-1=0.$ Let
$$x = \begin{pmatrix}
\omega^{-1} & \omega^2 \\
0 & \omega
\end{pmatrix} \text{ and } y=\begin{pmatrix}
0 & -1 \\
1 & 0
\end{pmatrix}.$$ It is routine to check the statement of the lemma directly, by hand or computationally. 

\medskip

\ref{lrtdf3} 
Computation shows that 
$$ S\cap S^x = \left\langle
 \begin{pmatrix}
1    & 1 & 1  \\
0 &1    & 0  \\
  1&    0    &   0      
\end{pmatrix} 
\right\rangle $$
has order 3 where $$
x= \begin{pmatrix}
1    & 0 & 1  \\
1 &1    & 1  \\
  0&    0    &   1      
\end{pmatrix}. $$ 
\end{proof}

\section{Solvable subgroups of $\GaL_n(q)$}
\label{sec411}

In this section, $S$ is a maximal solvable subgroup of $\GaL_n(q),$ $G=S \cdot \SL_n(q)$, $H=S/Z(\GL_n(q))$ and $\overline{G}=G/Z(\GL_n(q)).$ Our goal in this section is to prove the following theorem.

\begin{T1}
Let  $n \ge 2$ and $(n,q)$ is neither $(2,2)$ nor $(2,3).$ If $S$ is a maximal solvable subgroup of $\GaL_n(q)$, 
 then $\Reg_S(S \cdot \SL_n(q),5)\ge 5$. In particular $b_S(S \cdot \SL_n(q)) \le 5.$
\end{T1}

Before we start the proof, let us discuss the structure of a maximal solvable subgroup $S \le \GaL_n(q)$ and fix some notation. 

Let $0< V_1 \le V$ be such that $(V_1)S=V_1$ and $V_1$ has no non-zero proper $S$-invariant subspace. It is easy to see that $S$ acts semilinearly on $V_1$. Set  $n_1=\dim V_1$ and let
 $$\gamma_1:S \to \GaL_{n_1}(q)$$  be the homomorphism defined by $\gamma_1:g \mapsto g|_{_{V_1}}.$  Since $V_1$ is $S$-invariant, $S$ acts (semilinearly) on $V/V_1.$ Let $V_1 < V_2 \le V$ be such that $(V_2/V_1)S=V_2/V_1$ and $V_2/V_1$ has no non-zero proper $S$-invariant subspace.  Observe that $S$ acts semilinearly on $V_2/V_1$. Setting $n_2=\dim (V_2/V_1),$ let $\gamma_2:S \to \GaL_{n_2}(q)$  be the map $\gamma_2:g \mapsto g|_{_{(V_2/V_1)}}.$ 
Continuing this procedure we obtain the chain of subspaces 
\begin{equation}
\label{Gchain}
0=V_0<V_1< \ldots < V_k=V
\end{equation}
 and a sequence of homomorphisms $\gamma_1, \ldots, \gamma_k$ such that $V_{i}/V_{i-1}$ is $S$-invariant and has no  non-zero proper $S$-invariant  subspaces and $\gamma_i(S) \le \GaL_{n_i}(q)$ is the restriction of $S$ to $V_{i}/V_{i-1}$ for $i \in \{1, \ldots, k\}.$

 \begin{Lem}
 \label{GammairGL}
 If $k=1$, then either $S \cap \GL_n(q)$   is an irreducible solvable subgroup of $\GL_n(q)$, or there exists $x \in \SL_n(q)$ such that $S \cap S^x \cap \GL_n(q) \le Z(\GL_n(q)).$
 \end{Lem}
 \begin{proof}
 Since $k=1,$  $V$ has no  non-zero proper $S$-invariant subspace. Let $G= S \cdot \SL_n(q)$ and $M=S \cap \GL_n(q).$ Assume that $M$ is reducible, so there exists $0<U_1 <V$ of dimension $m$ such that $(U_1)M=U_1$ and $U_1$ is $\mathbb{F}_q[M]$-irreducible.
 
 Since $M \trianglelefteq S,$ $U_1 \varphi$ is an irreducible $\mathbb{F}_q[M]$- submodule of dimension $m$ for all $\varphi \in S.$    Let $\varphi \in S$ be such that $M \varphi$ is a generator of $S/M$, so $|\varphi|=|S:M|=r.$ Let $U_i=U_1 \varphi^{i-1}$ for $i \in \{1, \ldots, r\},$ so $V=\sum_{i=1}^r U_i$ since $k=1$ and there are no non-zero proper $S$-invariant subspaces in $V$.

We claim that $$V=U_1 \oplus \ldots \oplus U_{n/m},$$
so $M$ is completely reducible. The argument  is similar to the proof of Clifford's Theorem. Let $W_i=\sum_{j=1}^i U_i.$ If  $W_{i+1}=W_i$ for some $i$, then $W_{i}$ is $G$-invariant. Indeed,
$$(W_i)\varphi=(U_1 \varphi + \ldots + U_i \varphi)=(U_2 + \ldots + U_{i+1})\subseteq W_{i+1}=W_i$$
 and $(W_i)\varphi=W_i$ since $\dim (W_i)\varphi = \dim W_i.$ On the other hand, if $W_{i+1}>W_i$, then  $W_i \cap U_{i+1}$ is a proper $\mathbb{F}_q[M]$-submodule of $U_{i+1},$ so it must be zero. Hence  $W_{i+1}=W_i\oplus U_{i+1}$. Since $k=1,$ $W_{i+1}>W_i$ for $1 \le i \le n/m -1.$ So, by induction, 
 $$V=W_{n/m}=U_1 \oplus \ldots \oplus U_{n/m}.$$
 
  Let $M_i$ be the restriction of $M$ on $U_i$, so $M_i$ is an irreducible solvable subgroup of $\GL_m(q)$ for all $i \in \{1, \ldots, n/m\}.$ Notice that $q$ is not a prime since $k=1$ and $m<n$. Therefore, by Theorem \ref{irred}, one of the following holds:
  \begin{enumerate}
  \item  there exists $x_i \in \SL_m(q)$ such that $M_i \cap M_i^{x_i} \le Z(\GL_m(q))$; \label{clif1}
  \item  $m=2$ and $M_i$ lies in the normaliser $N_i$ of a Singer cycle of $\GL_2(q);$ \label{clif2}
  \item $m=2$, $q=9$ and $M_i$ lies in a subgroup $N_i$ of $\GL_2(9)$ that is conjugate in $\GL_2(9)$ to  $\GL_2(3) \cdot Z(\GL_2(9)).$ \label{clif3} 
\end{enumerate}  
In  case \ref{clif1}, define $x \in \SL_n(q)$ to be $\diag(x_1, \ldots, x_{n/m}) (I_m \otimes A(n/m))$ as in Lemma  \ref{igrekl}. It follows directly from Lemma \ref{igrekl} that $S \cap S^x \cap \GL_n(q) \le Z(\GL_n(q)).$

 In cases \ref{clif2} and \ref{clif3}, all $N_i$ are conjugate in $\GL_2(q)$ by  Lemma \ref{singconj} and computations respectively. Hence, in a suitable basis, $M$ is a subgroup of the group
$$\{\diag[g_1, \ldots, g_{n/m}] \mid g_i \in N_1\}.$$ Hence $M \le N_1 \wr P$ for any solvable transitive subgroup $P$ of $\Sym(n/m).$ By \cite[Lemma 4,\S15]{sup}, $N_1 \wr P$ is irreducible. Now there exists  $x \in \SL_n(q)$ such that $S \cap S^x \cap \GL_n(q)=M \cap M^x \le Z(\GL_n(q))$ by Theorem \ref{irred} since $n\ge 4$ and $q$ is not prime.  
 \end{proof}
 
 We start the proof of Theorem \ref{theorem} with the case $k=1.$

\begin{Th}
Theorem {\rm \ref{theorem}} holds for $k=1.$ 
\end{Th}
\begin{proof}
Let $M = S \cap \GL_n(q)$ and recall that $H=S/Z(\GL_n(q))$. By Lemma \ref{GammairGL}, if $M$ is not a subgroup of one of the groups listed in  \ref{irred11} -- \ref{irred15} of Theorem \ref{irred}, then there exists $x \in \SL_n(q)$ such that $M \cap M^x \le Z(\GL_n(q)).$ So $H \cap H^{\overline{x}}$ is a cyclic subgroup of $\overline{G}$ and by Theorem \ref{zenab} there exists $\overline{y} \in \overline{G}$ such that 
$$(H \cap H^{\overline{x}}) \cap (H \cap H^{\overline{x}})^{\overline{y}}=1.$$ Hence $b_S(G) \le 4$ and $\Reg_S(G,5) \ge 5$ by Lemma \ref{base4}.

And if  $M$ is a subgroup of one of the groups in \ref{irred13} -- \ref{irred15} of Theorem \ref{irred}, then $S=M$ and $b_S(G) \le 4.$ If  $M$ is a subgroup of one of the groups in \ref{irred11} -- \ref{irred135}, then $n=2$ and $b_S(G) \le 3$ by \cite[Table 3]{burness}  for $q>4$ and by computation for $q=4$.  So $\Reg_S(G,5) \ge 5$ by Lemma \ref{base4}.
\end{proof}

 For the rest of the section we assume that $k>1.$ Our proof for $q=2$ naturally splits into two cases.  If $q=2$ and $S_0$ is the normaliser of a Singer cycle of $\GL_3(2)$, then there is no $x, y \in \GL_3(2)$ such that $S_0^{x} \cap S_0^{y}$ is contained in $RT(\GL_3(2))$ (see Theorem \ref{irred}). So, in Theorem \ref{lem41}, we assume that if $q=2$, then there is no  $i \in \{1, \ldots, k\}$ such that $\gamma_i(S)$ is the normaliser of a Singer cycle of $\GL_3(2).$ In Theorem \ref{lem42} we address the case where there exists such an $i$.

\begin{Th}
\label{lem41}
Let $k>1.$ Theorem {\rm \ref{theorem}} holds if
 $q \ge 3$, or
if  $q=2$ and there is no  $i \in \{1, \ldots, k\}$ such that $\gamma_i(S)$ is the normaliser of a Singer cycle of $\GL_3(2).$

\end{Th} 
\begin{proof}

Since $k>1,$ there exists a nontrivial $S$-invariant subspace $U<V$ of dimension $m<n$. Let $M$ be $S \cap \GL_n(q).$ We fix  
\begin{equation}
\label{basGLG}
\beta=\{v_1, v_2, \ldots, v_{n-m+1}, v_{n-m+2}, \ldots, v_n \} 
\end{equation}
such that the last $ \left( \sum_{j=1}^i n_j \right)$  
 vectors in $\beta$ form a basis of $V_i$. So  $g \in M$ has  shape 
\begin{equation}
\label{stupG}
\begin{pmatrix}
\gamma_k(g)     & * & \ldots & * & * \\
0          &    \gamma_{k-1}(g)  & \ldots & * & * \\
     &     & \ddots & \ddots  &    \\
   0    &    \ldots& 0 & \gamma_{2}(g) & * \\
      0    &   \ldots        & \ldots& 0 & \gamma_{1}(g) 
\end{pmatrix}
\end{equation}
 with respect to $\beta.$ Recall that $q=p^f$, so if $f=1,$ then $\GaL_n(q)=\GL_n(q)$ and $S=M.$ Let  $\phi \in \GaL_n(q)$ be such that \begin{equation}
\label{phidef}
\phi : \lambda v_i \mapsto \lambda^p v_i \text{ for all } i \in \{1, \ldots, n\} \text{ and } \lambda \in \mathbb{F}_q.
\end{equation}

Assume that $f>1.$  Let $x= \diag[x_k, \ldots, x_1]$, where $x_i\in \SL_{n_i}(q)$ is such that $\gamma_i(M) \cap (\gamma_i(M))^{x_i} \le RT(\GL_{n_i}(q))$. Such an $x_i$ exists by Lemma \ref{GammairGL} and Theorem \ref{irred}. Let $y$ be as in the proof of Lemma \ref{irrtog}. Recall that  we defined $\gamma_i$ only on $S$, but it is easy to see that $\gamma_i$ can be extended to $\Stab_{\GaL_n(q)}(V_i, V_{i-1})$ via $\gamma_i:g \mapsto g|_{_{(V_i/V_{i-1})}}$ for $g \in \Stab_{\GaL_n(q)}(V_i, V_{i-1})$.

  We present the following piece of the proof as a proposition for easy reference. 
\begin{Prop} 
\label{clm44}
Let $f>1.$ There exists $\beta$ as in \eqref{basGLG} such that for every $\varphi \in (S \cap S^x) \cap (S \cap S^x)^y$, 
$$\varphi=\phi^j g \text{ for } j \in \{0,1, \ldots,f-1\}$$
where $g=\diag[g_k, \ldots, g_1] \in \GL_n(q)$ is diagonal with $g_i=\gamma_i(g) \in D(\GL_{n_i}(q))$.
 Moreover, either $g_i \in Z(\GL_{n_i}(q))$ or $n_i=2$,  $\gamma_i(S)$ normalises a Singer cycle of $\GL_2(q)$ and $(S \cap S^x) \cap (S \cap S^x)^y \le D(\GL_n(q)).$
\end{Prop}
\begin{proof}
Consider $\varphi \in S \cap S^x.$
Notice that, by Theorem \ref{irred}, one of the following holds:
\begin{enumerate}
\item[$(a)$] $\gamma_i(M) \cap \gamma_i(M)^{x_i} \le Z(\GL_{n_i}(q))$; 
\item[$(b)$]  $n_i=2$ and $\gamma_i(M)$ lies  in the normaliser $N$ of a Singer cycle $T$ in $\GL_2(q)$;
\item[$(c)$] $n_i=2$, $q=9$ and $\gamma_i(M)$ lies in $\GL_2(3) \cdot Z(\GL_2(9)).$
\end{enumerate}
  Let us show that we can assume that $S$ lies in the normaliser of $T$ in $\GaL_2(q)$ in $(b)$, and in the normaliser of $\GL_2(3)$ in $\GaL_2(9)$ in $(c)$. In $(b)$, if $\gamma_i(M)$ is imprimitive, so lies in a conjugate of $\GL_1(q)\wr \Sym(2),$ then  $\gamma_i(M) \cap \gamma_i(M)^{x_i} \le Z(\GL_{n_i}(q))$ by Theorem \ref{irred} and $(a)$ holds. If $\gamma_i(M)$ is primitive, then $R=\gamma_i(M) \cap T$ is the unique maximal normal irreducible abelian subgroup, so $S$ normalises $C_{\GL_2(q)}(R)=T$ and $S \le N_{\GaL_2(q)}(T)$ as claimed. In $(c)$, the claim is verified by computation.    

 If $n_i=2$ and $\gamma_i(S)$ lies in the normaliser of a Singer cycle of $\GL_2(q)$ in $\GaL_2(q)$, then, by Lemma \ref{lemRTdiagfield}, we can choose $x_i \in \SL_2(q)$ such that $\gamma_i(S) \cap \gamma_i(S)^{x_i}\le RT(\GL_2(q))$.  If $\varphi$ acts linearly on $V_i/V_{i-1},$ then $\varphi$ acts linearly on $V$, so $S \cap S^x \le \GL_n(q)$.
 
  If $n_i=2,$ $q=9$ and $\gamma_i(S)$ lies in the normaliser of $\GL_2(3)$ in $\GaL_2(9),$ then
  $$\gamma_i((S \cap S^x) \cap (S \cap S^x)^y) \le Z(\GL_2(9))\rtimes \langle \phi \rangle$$  by  Lemma \ref{lemRTdiagfield}.

     Otherwise $\gamma_i(M) \cap \gamma_i(M)^{x_i} \le Z(\GL_{n_i}(q))$ and, by Lemma \ref{scfield}, there exists a basis of $V_i/V_{i-1}$ such that $\gamma_i(\varphi)= \phi^{j_i} g_i$ with scalar $g_i$ and $j_i \in \{0,1, \ldots , f-1 \}.$ Hence we can choose $\beta$ in \eqref{basGLG} such that if $\varphi \in S \cap S^x,$ then $\gamma_i(\varphi)= \phi^{j_i} g_i$ with scalar $g_i.$ Since $(\phi^{j_1})^{-1} \varphi$ acts on $V_1\ne 0$ linearly,   $(\phi^{j_1})^{-1} \varphi \in \GL_n(q),$ so $j_i=j_l$ for all $i,l \in \{1, \ldots, k\}$
and 
\begin{equation}
\label{phitri}
\varphi = \phi^j 
\begin{pmatrix}
g_k     & * & \ldots & * & * \\
0          &    g_{k-1}  & \ldots & * & * \\
     &     & \ddots & \ddots  &    \\
   0    &    \ldots& 0 & g_{2} & * \\
      0    &   \ldots        & \ldots& 0 & g_{1}
\end{pmatrix}
\end{equation}
 with $g_i \in Z(\GL_{n_i}(q))$ unless $\gamma_i(S)$ normalises a Singer cycle of $\GL_2(q)$  where $g_i \in RT(\GL_2(q))$ and $j=0$.


Since $\varphi \in (S \cap S^x)^y,$
$$
\varphi = \phi^j 
\begin{pmatrix}
g_1'     & 0 & \ldots & 0 & 0 \\
*          &    g_{k-1}'  & \ldots & 0 & 0 \\
     &     & \ddots & \ddots  &    \\
   *    &    \ldots& * & g_{2}' & 0 \\
      *    &   \ldots        & \ldots& * & g_{1}'
\end{pmatrix}
$$ with $g_i$ either scalar or lower-triangular. So 
\begin{equation}
\label{phidiagG}
\varphi =\phi^j g \text{ where } g=\diag[g_k, \ldots, g_1]
\end{equation}
 with $g_i$ scalar unless $\gamma_i(S)$ normalises a Singer cycle of $\GL_2(q)$  where $g_i \in D(\GL_2(q))$ and $j=0$.  
\end{proof}

\medskip

We now resume our proof of Theorem \ref{lem41}.
Let $\Omega$ be the set of right $S$-cosets in $G$ and let $\overline{\Omega}$ be the set of right $H$-cosets in $\overline{G}$, where the action is given by  right multiplication.  Since $S$ is maximal solvable and $$S_{G}:=\cap_{g \in G} S^g=Z(\GL_n(q)),$$ $\Reg_S(G,5)$ is the number of $\overline{G}$-regular orbits on  $\overline{\Omega}^5$.   Therefore, $$\overline{\omega}=(H\overline{g_1},H\overline{g_2},H\overline{g_3},H\overline{g_4},H\overline{g_5}) \in \overline{\Omega}^5$$ is regular under the action of $\overline{G}$ if and only if the stabiliser of $$ \omega =(Sg_1, Sg_2, Sg_3,Sg_4, Sg_5) \in \Omega^5$$ under the action of $G$ is equal to $Z(\GL_n(q))$, where $g_i$ is a preimage in $\GaL_n(q)$ of $\overline{g_i}.$  
In our proof we say that   $\omega \in \Omega^5$ is regular if it is stabilised only by elements from $Z(\GL_n(q))$, so $\overline{\omega}$ is regular in terms of the standard definition; that is, the stabiliser of $\overline{\omega}$ in $\overline{G}$ is trivial.  

The proof of Theorem \ref{lem41} splits into five cases. To show $\Reg_S(G,5) \ge 5$, in each case we find five regular orbits in $\Omega^5$ or show that $b_S(G)\le 4,$ so $\Reg_S(G,5) \ge 5$ by Lemma \ref{base4}. Recall that $n_i = \dim V_i/V_{i-1}$ for $i \in \{1, \ldots, k\}.$  Different cases arise according to the number of $n_i$ for $i \in \{1, \ldots, k\}$ which equal 2  and in what rows of  $g \in M=S\cap \GL_n(q)$ the submatrix $\gamma_i(g)$ is located for $n_i=2.$ We now specify the cases.
\begin{description}[before={\renewcommand\makelabel[1]{\bfseries ##1}}]
\item[{\bf Case 1.}] Either $f>1$, or $f=1$ and the number of $i \in \{1, \ldots, k\}$ with $n_i=2$ is not one;
\item[{\bf Case 2.}] $f=1, $ $n$ is even, there exist exactly one $i \in \{1, \ldots, k\}$ with $n_i=2$ and $\gamma_i(g)$ appears in rows  $\{n/2,n/2+1\}$ of $g \in M$ for such $i$;
\item[{\bf Case 3.}] $f=1, $ $n$ is odd, there exist exactly one $i \in \{1, \ldots, k\}$ with $n_i=2$ and $\gamma_i(g)$ appears in rows  $\{(n+1)/2,(n+1)/2+1\}$ of $g \in M$ for such $i$;
\item[{\bf Case 4.}] $f=1, $ $n$ is odd, there exist exactly one $i \in \{1, \ldots, k\}$ with $n_i=2$ and $\gamma_i(g)$ appears in rows  $\{(n-1)/2,(n-1)/2+1\}$ of $g \in M$ for such $i$;
\item[{\bf Case 5.}] $f=1,$ there exist exactly one $i \in \{1, \ldots, k\}$ with $n_i=2$ and none of {\bf Cases 2 -- 4} holds. 
\end{description}


Our proof of Theorem \ref{lem41} in {\bf Cases 2 -- 5} is analogous and repetitive, so we only demonstrate a proof for {\bf Cases 1} and {\bf 5}. A detailed proof for {\bf Cases 2 -- 4} can be found in \cite{thesis}.

\medskip

Before we proceed with the proof of Theorem \ref{lem41}, let us resolve two situations which both arise often in our analysis of these cases and can be readily settled.

\begin{Prop}
\label{ni1}
Let $S$ be a maximal solvable subgroup of $\GaL_n(q),$ let $\beta$ be as in \eqref{basGLG},  $n \ge 2$ and $(n,q)$ is neither $(2,2)$ nor $(2,3).$ If  all $V_i$ in \eqref{Gchain} are such that $n_i=1,$ then there exist $x,y,z \in \SL_n(q)$ such that $S \cap S^x \cap S^y \cap S^z \le Z(\GL_n(q)).$ 
\end{Prop}
\begin{proof}
Let $\sigma =(1, n)(2, n - 1) \ldots ([n/2], [n/2 + 3/2]) \in \Sym(n).$ We define $x,y,z \in \SL_n(q)$ as follows:
\begin{itemize}
 \item $x = \diag(\sgn(\sigma), 1\ldots, 1) \cdot \per (\sigma);$ 
\item $(v_i)y=v_i$ for $i \in \{1, \ldots, n-1\}$ and $(v_n)y= \sum_{i=1}^n v_i$;
\item $(v_i)z=v_i$ for $i \in \{1, \ldots, n-1\}$ and $(v_n)z= \theta v_1 + \sum_{i=2}^n v_i$, 
\end{itemize}
where $\theta$ is a generator of $\mathbb{F}_q^*.$ Let $\varphi \in S \cap S^x \cap S^y \cap S^z.$ Since $\varphi \in S \cap S^x,$ it stabilises $\langle v_i \rangle $ for all $i \in \{1, \ldots, n\}.$ So $\varphi =\phi^j \diag(\alpha_1, \ldots, \alpha_n)$ for $j \in \{0,1, \ldots, f-1\}$ and $\alpha_i \in \mathbb{F}_q^*.$ Since $\varphi \in S^y,$ it stabilises $(V_1)y =\langle v_n \rangle y= \langle v_1 + \ldots +v_n \rangle.$ So $\alpha_1=\alpha_i$ for $i \in \{2, \ldots, n\}.$  Since $\varphi \in S^z,$ it stabilises $(V_1)z =\langle v_n \rangle z= \langle \theta v_1 + v_2 + \ldots +v_n \rangle.$ So $\theta^{p^j}=\theta$ and $j=0.$ Hence $\varphi \in Z(\GL_n(q)).$ 
\end{proof}

\begin{Prop}
\label{n3GL}
Fix a basis $\beta$ as in \eqref{basGLG}. Let $n \in \{2,3\}$ and $(n,q)$ is neither $(2,2)$ nor $(2,3)$. Let $S$ be a maximal solvable subgroup of $\GaL_n(q).$ If $S$ stabilises a  non-zero proper subspace $V_1$ of $V$, then there exist $x,y,z \in \SL_n(q)$ such that 
 $$S \cap S^x \cap S^y \cap S^z \le Z(\GL_n(q)).$$
\end{Prop}
\begin{proof}
We can assume that $V_1$ has no non-zero proper $S$-invariant subspaces. Let $V_i$ be as in \eqref{Gchain}.

If $n=2$, then $\dim V_1=1$ and $n_1=n_2=1$, so the statement follows by Proposition \ref{ni1}.

Assume $n=3$.  If $k=3,$ so $n_i=1$ for all $i$, then the statement follows by Proposition \ref{ni1}, so we assume $k=2.$

Let $n_1=1,$ so $n_2=2.$ Let $f>1.$ If $q=9$, the proposition is verified using computations, so we assume $q\ne 9.$ Assume that $\gamma_2(S)$ is a subgroup of the normaliser of a Singer cycle of $\GL_2(q)$ in $\GaL_2(q).$ Therefore, by \eqref{2sindiagodd} and \eqref{2sindiageven}, there exists $x_1 \in \SL_2(q)$ such that $\gamma_2(S) \cap (\gamma_2(S))^{x_1} \le \langle \varphi \rangle Z(\GL_2(q))$ where $\varphi = \diag(-1,1)$ if $q$ is odd and $\varphi = \begin{pmatrix}
1 & 0\\
1 & 1
\end{pmatrix}$ if $q$ is even. Hence, if $x=\diag[x_1,1],$ then $S \cap S^x \le \GL_3(q)$ and matrices in $S \cap S^x$ have shape 
$$
 \begin{pmatrix}
\alpha_1 \varphi & *\\
0 & \alpha_2
\end{pmatrix}
$$ with $\alpha_1, \alpha_2 \in \mathbb{F}_q^*.$ Let $y=\per((1,2,3))$ and let $z=z_1$ if q is odd and $z=z_2$ if $q$ is even where
$$
z_1=
\begin{pmatrix}
1&0&0\\
0&1&0\\
1&1&1
\end{pmatrix}, \text{ }
z_2=\begin{pmatrix}
0&1&0\\
0&0&1\\
1&0&0
\end{pmatrix}.
$$
Calculations show that $S \cap S^x \cap S^y \cap S^z \le Z(\GL_n(q)).$

If $\gamma_2(S)$ does not normalise a Singer cycle of $\GL_2(q),$ then, by Theorem \ref{irred}, 
   there exists $x_1 \in \SL_2(q)$ such that $\gamma_2(S) \cap (\gamma_2(S))^{x_1} \cap \GL_2(q) \le Z(\GL_2(q)).$ Let $x=\diag[x_1,1]$ so $\varphi \in S \cap S^x$ has shape \eqref{phitri} with $g_2 \in Z(\GL_2(q))$ and $g_1 \in \mathbb{F}_q^*.$ In particular, let 
$$
\varphi =\phi^j
\begin{pmatrix}
\alpha_1 & 0 & \delta_1\\
 0&     \alpha_1 & \delta_2\\
0 & 0& \alpha_2  
\end{pmatrix}.
$$
 Let $y=
\begin{pmatrix}
0&1&0\\
1&0&0\\
\theta&1&1
\end{pmatrix}$ and consider $\varphi \in (S \cap S^x) \cap (S \cap S^x)^y.$ Since $\varphi \in S^y,$ it stabilises $$(V_1)y=\langle v_3 \rangle y= \langle \theta v_1 +v_2 +v_3 \rangle.$$ Therefore,
$$(\theta v_1 +v_2 +v_3)\varphi = \theta^{p^j} \alpha_1 v_1 + \alpha_1 v_2 + (\alpha_2 +\delta_1 +\delta_2)v_3 \in \langle \theta v_1 +v_2 +v_3 \rangle,$$
so $\theta^{p^j} \alpha_1= \theta \alpha_1$ and $j=0$. Thus, $(S \cap S^x) \cap (S \cap S^x)^y \le \GL_3(q).$ Now calculations  show that $(S \cap S^x) \cap (S \cap S^x)^y \le Z(\GL_3(q)).$

Now assume $f=1$ (we  continue to assume that $n_1=1$). Let $\sigma=(1,3) \in \Sym(3)$ and $x=\diag(\sgn(\sigma),1,1) \cdot {\rm perm}(\sigma) \in \SL_3(q)$. Matrices from $S \cap S^x$ have shape  
  \begin{equation*} 
  \begin{pmatrix}
*    & 0 & 0  \\
* &*    & *  \\
  0&    0    &   *      
\end{pmatrix}.
\end{equation*}
Let   
  \begin{equation*} 
y=  \begin{pmatrix}
0    & -1 & 0  \\
1 &0    & 0  \\
  1&    1    &   1      
\end{pmatrix} \in \SL_3(q).
\end{equation*}
It is easy to check that $$(S\cap S^x) \cap (S\cap S^x)^y \le Z(\GL_3(q)).$$

To complete the proof, we assume  $n_1=2,$ so $n_2=1.$ Let $\iota$ be the inverse-transpose map on $\GL_n(q)$. As mentioned in the introduction, $\iota$ induces an automorphism on $\SL_3(q)$ and we can consider the group $\GaL_3(q) \rtimes \langle \iota \rangle$ (here $\iota$ and $\phi$ commute). It is easy to see that $S^{\iota}$ is a maximal solvable subgroup of $\GaL_3(q)$ stabilising $1$-dimensional subspace of $V$. Hence, by taking $x,y,z \in \SL_3(q)$ for $S^{\iota}$  as in the case $n_1=1,$ we obtain
\[(S\cap S^{x^{\iota}}) \cap (S\cap S^{x^{\iota}})^{y^{\iota}} \le Z(\GL_3(q)). \qedhere\]
\end{proof} 

We now consider {\bf Cases 1} and {\bf 5} in the following two propositions.

\begin{Prop}
\label{case1prop}
Theorem $\ref{lem41}$ holds in {\bf Case 1}.
\end{Prop} 
\begin{proof}
In this case, there are four conjugates of $M=S \cap \GL_n(q)$ whose intersection lies in $D(\GL_n(q)).$ Indeed, if $f>1,$ or $f=1$ and there is no $i \in \{1, \ldots, k\}$ (here $k$ is as in \eqref{Gchain}) such that $n_i=2$, then 
 there exist $x,y \in \SL_n(q)$ such that 
$$(M \cap M^x) \cap (M \cap M^x)^y \le D(\GL_n(q))$$
by Proposition \ref{clm44},  or by Theorem \ref{irred}  and Lemma \ref{irrtog}.

Assume $f=1$ (so $S=M$) and $t=|\{i \in \{1, \ldots, k\} \mid n_i=2\}|\ge 2.$ Let $j_1, \ldots, j_t \in \{1, \ldots, n\}$ be such that    $(2 \times 2)$ blocks (corresponding to $V_i/V_{i-1}$ of dimension $2$) on the diagonal in matrices of $M$ occur in the rows $$(j_1, j_1+1), (j_2, j_2 +1), \ldots, (j_t, j_t+1).$$ 
Let $\tilde{x}=\diag(\sgn(\sigma), 1, \ldots, 1) \cdot {\rm perm}(\sigma) \in \SL_n(q),$ where  $$\sigma=(j_1, j_1+1, j_2, j_2 +1, \ldots, j_t, j_t+1)(j_1, j_1+1).$$  Let $x= \diag[x_k, \ldots, x_1]\tilde{x}$, where $x_i\in \SL_{n_i}(q)$ is such that $\gamma_i(M) \cap (\gamma_i(M))^{x_i} \le RT(\GL_{n_i}(q))$ if $n_i \ne 2$ and $x_i$ is the identity matrix if $n_i=2.$ Such $x_i$ exist by Theorem \ref{irred}. If $y$ is as in Lemma \ref{irrtog}, then calculations show that  
\begin{equation}\label{m4diag}
(M \cap M^x) \cap (M \cap M^x)^y \le D(\GL_n(q)).
\end{equation}


Let $0<U <V$ be an $S$-invariant subspace and $\dim U=m.$ In general, we take $U=V_i$ for some $i<k$ (in most cases it is sufficient to take $U=V_1$). We fix $\beta$ to be as in \eqref{basGLG} such that Proposition \ref{clm44} holds for $f>1.$ So 
$$U = \langle v_{n-m+1}, v_{n-m+2}, \ldots, v_n \rangle.$$ Since in the proof we reorder $\beta,$ let us fix a second notation for  $$v_{n-m+1}, v_{n-m+2}, \ldots, v_n,$$ namely 
$w_1, \ldots, w_m$
respectively. So $U = \langle w_1, \ldots, w_m \rangle.$

  Our proof splits into the following four  subcases:
\begin{description}[before={\renewcommand\makelabel[1]{\bfseries ##1}}]
\item[{\bf Case (1.1)}] $n-m\ge 2 $ and $m \ge 2$;
\item[{\bf Case (1.2)}] $n-m=m=1$;
\item[{\bf Case (1.3)}] $n-m \ge 2$ and $m=1$;
\item[{\bf Case (1.4)}] $n-m=1$  and $m \ge 2$ (dual to {\bf Case (1.3)} via the inverse-transpose automorphism).
\end{description}

{\bf Case (1.1).} Let $n-m\ge 2 $ and $m \ge 2$, so $n \ge 4$. We claim that there exist $z_i \in \SL_n(q)$ for $i=1, \ldots, 5$ such that the points $\omega_i=(S,Sx,Sy,Sxy,Sz_i)$ lie in distinct regular orbits. 

 Recall from \eqref{m4diag} and Proposition \ref{clm44} that, if either $f=1$ or  $f>1$ and  $\gamma_i(S)$ normalises a Singer cycle of $\GL_2(q)$ for some $i \in \{1, \ldots, k\}$, then $$(S \cap S^x) \cap (S \cap S^x)^y \le D(\GL_n(q)).$$ Also, we can assume that $n_l \ge 2$ for some $l \in \{1, \ldots, k\}$ in \eqref{Gchain}: otherwise Theorem~\ref{lem41} follows by Proposition \ref{ni1}. Let $s \in \{1, \ldots, n\}$ be such that $$V_l = \langle v_s, v_{s+1}, \ldots, v_{s+n_l-1}, V_{l-1} \rangle.$$

 If $(S \cap S^x) \cap (S \cap S^x)^y$ does not lie in $\GL_n(q)$, then   $\varphi \in (S \cap S^x) \cap (S \cap S^x)^y$ stabilises $W=\langle v_s, v_{s+1}, \ldots, v_{s+n_l-1} \rangle$ and the restriction $\varphi|_{_{W}}$ is $\gamma_l(\varphi)=\phi^j g_l$ where $g_l =\lambda I_{n_l}$ for some $\lambda \in \mathbb{F}_q^*$ by \eqref{phidiagG}. We now relabel vectors in $\beta$ as follows: $v_s, v_{s+1}, \ldots, v_{s+n_l-1}$ becomes $v_1, v_{2}, \ldots, v_{n_l}$ respectively; $v_1, v_{2}, \ldots, v_{n_l}$  becomes $v_s, v_{s+1}, \ldots, v_{s+n_l-1}$  respectively;  the remaining labels are unchanged.

  Let $z_1, \ldots, z_5 \in \SL_n(q)$ be such that $$(w_{j})z_i=u_{(i,j)}$$ for $i=1, \ldots, 5$ and $j=1, \ldots, m,$ where

\begingroup
\allowdisplaybreaks
\begin{align*}
& \left\{ 
\begin{aligned}
u_{(1,1)} & =\theta v_1+v_2;\\
u_{(1,2)} & =\theta v_1+v_3+v_4;\\
u_{(1,2+r)} & =\theta v_1+v_{4+r}\phantom{;} \text{ for } r \in \{1, \ldots, m-3\};\\
u_{(1,m)} & =\theta v_1 +\sum_{r=2+m}^{n}v_r;
\end{aligned}
\right. \\
& \left\{ 
\begin{aligned}
u_{(2,1)} & =\theta v_1+v_3;\\
u_{(2,2)} & =\theta v_1+v_2+v_4;\\
u_{(2,2+r)} & =\theta v_1+v_{4+r}\phantom{;} \text{ for } r \in \{1, \ldots, m-3\};\\
u_{(2,m)} & =\theta v_1 +\sum_{r=2+m}^{n}v_r;\\
\end{aligned}
\right. \\
& \left\{ 
\begin{aligned}
u_{(3,1)} & =\theta v_1+v_4;\\
u_{(3,2)} & =\theta v_1+v_2+v_3;\\
u_{(3,2+r)} & =\theta v_1+v_{4+r}\phantom{;} \text{ for } r \in \{1, \ldots, m-3\}; \\
u_{(3,m)} & =\theta v_1 +\sum_{r=2+m}^{n}v_r;
\end{aligned}
\right. \stepcounter{equation}\tag{\theequation} \label{orb}  \\
& \left\{ 
\begin{aligned}
u_{(4,1)} & =v_2+v_3;\\
u_{(4,2)} & =\theta v_1+v_2+v_4;\\
u_{(4,2+r)} & =v_2+v_{4+r}\phantom{;} \text{ for } r \in \{1, \ldots, m-3\};\\
u_{(4,m)} & =v_2 +\sum_{r=2+m}^{n}v_r;
\end{aligned}
\right. \\
& \left\{ 
\begin{aligned}
u_{(5,1)} & =v_2+v_4;\\
u_{(5,2)} & =\theta v_1+v_2+v_3;\\
u_{(5,2+r)} & =v_2+v_{4+r}\phantom{;} \text{ for } r \in \{1, \ldots, m-3\};\\
u_{(5,m)} & =v_2 +\sum_{r=2+m}^{n}v_r.
\end{aligned}
\right.  
\end{align*}
\endgroup

Recall that $\theta$ is a generator of $\mathbb{F}_q^*.$ Such $z_i$ always exist since $m<n$ and $u_{(i,1)}, \ldots, u_{(i,m)}$ are linearly independent  for every $i =1 ,\ldots, 5$. 
Consider $\varphi \in (S \cap S^x) \cap (S \cap S^x)^y \cap S^{z_1}$, so $\varphi=\phi^jg$ as in \eqref{phidiagG} with $g=\diag({\alpha_1, \ldots, \alpha_n})$ for $\alpha_i \in \mathbb{F}_q^*.$ Notice that $\varphi$ stabilises the subspace $Uz_1.$ 
Therefore, 
$$(u_{(1,1)})\varphi =
 \theta^{p^j} \alpha_1 v_1 + \alpha_2 v_2=
 \mu_1 u_{(1,1)} + \ldots + \mu_m u_{(1,m)}
$$ for $\mu_i \in \mathbb{F}_q^*.$ The first expression does not contain $v_i$ for $i>2$, so $\mu_i=0$ for $i>1$ and $(u_{(1,1)})\varphi=\mu_1 u_{(1,1)}= \mu_1(\theta v_1 +v_2).$ Therefore, $\mu_1= \alpha_2.$  Assume  that $\varphi \notin \GL_n(q)$, so, by the arguments before \eqref{orb}, $\alpha_1 =\alpha_2$,  $\theta^{p^j-1}=1$ and $j=0$ which is a contradiction. Thus, we can assume $\varphi=g \in \GL_n(q).$
 The same arguments as in the proof of Lemma \ref{diag} show that if $j\in \{1,\ldots, m\}$ then $$(u_{(1,j)} )g = \delta_j u_{(1,j)}$$ for some $\delta_j \in \mathbb{F}_q$. Since all $u_{(1,j)}$ have $v_1$ in the decomposition \eqref{orb}, all $\delta_j$ are equal and $g$ is a scalar matrix, so $\omega_1$ is a regular point in $\Omega^5.$ It is routine to check that every $\omega_i \in \Omega^5$ is regular. 

Assume that $\omega_1$ and $\omega_2$ lie in the same orbit in $\Omega^5$, so there exists $\varphi \in \GaL_n(q)$ such that $\omega_1 \varphi =\omega_2$. This implies $$\varphi \in (S \cap S^x) \cap (S \cap S^x)^y \cap z_1^{-1}Sz_2.$$
In particular, $\varphi= \phi^j g$ with $g =\diag(\alpha_1, \ldots, \alpha_n) \in D(\GL_n(q))$ and $g=z_1^{-1}\psi z_2$, $\psi \in S$. Consider $(u_{(1,1)})\varphi$.  Firstly, 
\begin{equation}\label{orr}
(u_{(1,1)})\varphi=(\theta v_1+v_2)\varphi= \theta^{p^j} \alpha_1 v_1 + \alpha_2 v_2;
\end{equation}
on the other hand,
$$(u_{(1,1)})\varphi =(u_{(1,1)})z_1^{-1}\psi z_2=(w_{1})\psi z_2=(\eta_{1} w_{1} + \ldots +\eta_{m} w_{m})z_2 \in Uz_2,$$
for some $\eta_i \in \mathbb{F}_q$. So $(u_{(1,1)})\varphi= \delta u_{(2,2)}$ for some $\delta \in \mathbb{F}_q$ since there are no $v_i$ for $i>2$ in decomposition \eqref{orr}. However, 
$$\theta^{p^j} \alpha_1 v_1 + \alpha_2 v_2=\delta u_{(2,2)}$$
if and only if $(u_{(1,1)})\varphi=0$, which contradicts the fact that $\varphi$ is invertible. Hence  there is no such $\varphi$, so points $\omega_1$ and $\omega_2$ are in  distinct regular orbits on $\Omega^5.$ The same arguments show that all $\omega_i$ lie in  distinct regular orbits on $\Omega^5.$   

\medskip

{\bf Case (1.2).} If $n-m=m=1$, then $n=2$ and  Theorem \ref{lem41} follows by Proposition \ref{n3GL}.

\medskip

{\bf Case (1.3).} Let $m=1$ and $n-m \ge 2$, so $n \ge 3.$ 
If $n=3$, then Theorem \ref{lem41} follows by Proposition \ref{n3GL}, so we assume that $m=1$ and $n \ge 4.$ If $k\ge 4$ in \eqref{stup}, then $U = \langle v_{n_1+n_2+1}, \ldots, v_n \rangle$ is an $S$-invariant subspace with $m=n_3 + \ldots + n_k \ge 2$ and $n-m=n_1 +n_2 \ge 2$, so the proof as in {\bf Case (1.1)} using \eqref{orb} works.

 Assume that $k=3.$ If $n_3 \ge 2$, then    $U = \langle v_{n_3+1}, \ldots, v_n \rangle$ is an $S$-invariant subspace with $m=n_1 + n_2 \ge 2$ and $n-m=n_3 \ge 2$, so the proof as in {\bf Case (1.1)} using \eqref{orb} also works.

Let $n_3=1.$ Since $m=1$, $n_1=1.$ If $n_2=2$ and $q\in \{2,3,5\}$, then Theorem~\ref{lem41} is verified by computation. Otherwise, by Theorem \ref{irred} and Lemma \ref{scfield}, there exist $\tilde{x}_1,\tilde{x}_2 \in \SL_{n_2} (q)$  such that $$\gamma_2(S) \cap (\gamma_2(S))^{\tilde{x}_1} \cap (\gamma_2(S))^{\tilde{x}_2} \le \langle \phi^i \rangle Z(\GL_{n_2}(q))$$ for some $i \in \{0, 1, \ldots, f-1\}$. Let $x_1, x_2 \in \GL_n(q)$ be the matrices $\diag[1,\tilde{x}_1,1]$ and $\diag[1,\tilde{x}_2,1]$ respectively.  Let $\tilde{x}= \diag(\sgn(\sigma),1 \ldots, 1) \cdot \per(\sigma) \in \SL_n(q)$ with $\sigma=(1,n)$. Calculations show that $(S \cap S^{x_1} \cap S^{x_2 \tilde{x}}) \cap \GL_n(q) \le D(\GL_n(q)).$ So, with respect to a basis $\beta$ as in \eqref{basGLG},  $\varphi \in S \cap S^{x_1} \cap S^{x_2 \tilde{x}}$ has shape
$$\phi^j \diag(\alpha_1, \alpha_2, \ldots, \alpha_2, \alpha_3)$$ with $j \in \{0,1, \ldots, f-1\}$ and $\alpha_1, \alpha_2, \alpha_3 \in \mathbb{F}_q^*.$ Let 
$$y=
\begin{pmatrix}
1& 0      & 0\ldots  & 0 &0 \\
0& 1      & 0 \ldots  & 0 &0\\
&       &\ddots &  &  &\\
&       & & \ddots &  &\\
0& 0      & 0& \ldots & 1 &0 \\
1& \theta & 1& \ldots & 1 &1
\end{pmatrix}.$$
Calculations show that $S \cap S^{x_1} \cap S^{x_2 \tilde{x}} \cap S^y \le Z(\GL_n(q)).$ 

Now let $k=2$, so $n_2=n-m=n-1$ and $n_2>2$. If $(n,q)$ is $(4,2)$ or $(5,3)$, then Theorem \ref{lem41} is verified by  computation.  Otherwise,  by Theorem \ref{irred} and Lemma \ref{scfield},
there exists $x \in \SL_n(q)$ such that $S \cap S^x$ consists of elements of  shape (with respect to a basis $\beta$ as in \eqref{basGLG})
\begin{equation*} 
\phi^j
  \begin{pmatrix}
\alpha_1 &0 & \dots    & 0 & \delta_1  \\
0 & \alpha_1 & \dots  & 0 & \delta_2 \\
 &  & \ddots  &  &  \\
 0& \dots & 0  & \alpha_{1} & \delta_{n-1} \\
0 & 0 & \dots  & 0 & \alpha_2 \\
\end{pmatrix}
\end{equation*}
where $\alpha_i, \delta_i \in \mathbb{F}_q$ and $j \in \{0,1, \ldots, f-1\}.$
Let $\sigma =(1, n)(2, n - 1) \ldots ([n/2], [n/2 + 3/2]) \in \Sym(n).$ Let $$y = \per(\sigma) \cdot \diag(\sgn(\sigma), 1, \ldots, 1, \theta, \theta^{-1}).$$ Calculations show that $(S \cap S^x) \cap (S \cap S^x)^y \le Z(\GL_n(q)).$
Thus, $b_S(G)\le 4$ and $\Reg_{S}(G,5) \ge 5.$

\medskip

{\bf Case (1.4).} Now let $n-m=1$  and $m \ge 2$,  so $n\ge 3.$   Let $\iota$ be the inverse-transpose map on $\GL_n(q)$ and  consider the group $\GaL_n(q) \rtimes \langle \iota \rangle$. It is easy to see that {\bf Case (1.3)} holds for $S^{\iota}$. Hence, by taking $x,y,z \in \SL_n(q)$ for $S^{\iota}$  as in {\bf Case (1.3)}, we obtain
$$(S\cap S^{x^{\iota}}) \cap (S\cap S^{x^{\iota}})^{y^{\iota}} \le Z(\GL_n(q)).$$ Thus, $b_S(G)\le 4$ and $\Reg_{S}(G,5) \ge 5.$ This concludes the proof of Proposition \ref{case1prop}.
\end{proof}

\bigskip

Recall that $f=1$ for {\bf Cases 2 -- 5}, so $\GaL_n(q)=\GL_n(q)$ and $S=M.$ Therefore, $\beta$ and $\gamma_i$ are as in Lemma \ref{supreduce} and matrices from $S$ have shape \eqref{stup}. Denote $\gamma_i(S)$ by $S_i$.

\begin{Prop}
\label{case5prop}
Theorem $\ref{lem41}$ holds in {\bf Cases 2} -- {\bf 5}.
\end{Prop}
\begin{proof}
As mentioned before Proposition \ref{ni1}, we only provide a detailed proof for {\bf Case 5}.

 Let $i \in \{1, \ldots, k\}$  be  such that $n_i=2$. Our proof  splits into three subcases:
\begin{description}[before={\renewcommand\makelabel[1]{\bfseries ##1}}]
\item[{\bf Case (5.1)}] $n\ge 9$ and $\gamma_i(g)$ appears in rows  $\{l,l+1\}$ of $g \in S$ where $l>n/2$ if $n$ is even and $l>(n+1)/2$ if $n$ is odd;
\item[{\bf Case (5.2)}] $n \ge 9$ and $\gamma_i(g)$ appears in rows $\{l-1,l\},$ where $l \le n/2$ if $n$ is even and $l \le (n-1)/2$ if $n$ is odd;
\item[{\bf Case (5.3)}] $n<9$.
\end{description} 

{\bf Case (5.1).} Consider the case when  the only $(2\times 2)$ block is in  rows $\{l,l+1\},$ where $l>n/2$ if $n$ is even and $l>(n+1)/2$ if $n$ is odd. Denote $s:=n-l$. Let $x= \diag[x_k, \ldots, x_1]$, where $x_i\in \SL_{n_i}(q)$ is such that $S_i \cap S_i^{x_i} \le RT(\GL_{n_i}(q))$ if $n_i \ne 2$ and $x_i$ is the identity matrix if $n_i=2.$ If $y$ is as in Lemma \ref{irrtog}, then  calculations show that  
$$(S \cap S^x) \cap (S \cap S^x)^y$$ consists of matrices of shape 
\begin{equation}\label{ugol} 
\diag[\alpha_1, \ldots, \alpha_{s-1}, \begin{pmatrix} \alpha_s & \beta_s \\ 0 & \alpha_{s+1} \end{pmatrix}, \alpha_{s+1}, \ldots, \alpha_{l-1},\begin{pmatrix} \alpha_l & 0 \\ \beta_{l+1} & \alpha_{l+1} \end{pmatrix}, \alpha_{l+2}, \ldots, \alpha_n ]
\end{equation}

Assume that $n \ge 9$. We take as the $S$-invariant subspace $U$ the subspace with  basis $\{v_{l}, \ldots, v_n\}$, so
$m\ge 2$ and $n-m\ge [n/2] \ge 4.$  Let us rename some basis vectors for convenience, so denote vectors 
$$v_1, \ldots, v_{s-1}, v_{s+2}, \ldots, v_{l-1}, v_{l+2}, \ldots, v_{n}$$ by $$w_1, w_2, \ldots, w_{n-4}$$ respectively.
Let $z_1, \ldots, z_5 \in \SL_n(q)$ be such that $$(v_{n-m+j})z_i=u_{(i,j)}$$ for $i=1, \ldots, 5$ and $j=1, \ldots, m,$ where
\begin{gather}\label{orb2t1}
\left\{ 
\begin{aligned}
u_{(1,1)} & = v_{s} + v_{l+1};\\
u_{(1,2)} & = v_{s}+ v_{s+1}+v_{l}+w_1;\\
u_{(1,2+r)} & =v_{s+1}+w_{1+r}\phantom{;} &\text{ for } r \in \{1, \ldots, m-3\};\\
u_{(1,m)} & =v_{s+1}+\sum_{r=m-1}^{n-4}w_r;
\end{aligned}
\right.
\end{gather}
\begin{gather}
\left\{ 
\begin{aligned}
u_{(2,1)} & = v_{s} + v_{l+1};\\
u_{(2,2)} & = v_{s}+v_{s+1}+v_{l}+w_2;\\
u_{(2,2+r)} & =v_{s+1}+w_{(2,r)}\phantom{;} &\text{ for } r \in \{1, \ldots, m-3\};\\
u_{(2,m)} & =v_{s+1}+\sum_{r=m-2}^{n-5}w_{(2,r)}.
\end{aligned}
\right. 
\end{gather}
Here $w_{(2,1)}, \ldots, w_{(2,n-5)}$ are equal to $w_1, w_3, \ldots, w_{n-4}$ respectively.  
\begin{gather}
\left\{ 
\begin{aligned}
u_{(3,1)} & = v_{s} + v_{l+1};\\
u_{(3,2)} & = v_{s}+ v_{s+1}+v_{l}+w_3;\\
u_{(3,2+r)} & =v_{s+1}+w_{(3,r)}\phantom{;} &\text{ for } r \in \{1, \ldots, m-3\};\\
u_{(3,m)} & =v_{s+1}+\sum_{r=m-2}^{n-5}w_{(3,r)}.
\end{aligned}
\right.
\end{gather} 
Here $w_{(3,1)}, \ldots, w_{(3,n-5)}$ are equal to $w_1, w_2, w_4 \ldots, w_{n-4}$ respectively.  
\begin{gather} 
\left\{ 
\begin{aligned}
u_{(4,1)} & = v_{s} + v_{l+1};\\
u_{(4,2)} & = v_{s}+ v_{s+1}+v_{l}+w_4;\\
u_{(4,2+r)} & =v_{s+1}+w_{(4,r)}\phantom{;} &\text{ for } r \in \{1, \ldots, m-3\};\\
u_{(4,m)} & =v_{s+1}+\sum_{r=m-2}^{n-5}w_{(4,r)}.
\end{aligned}
\right. 
\end{gather}
Here $w_{(4,1)}, \ldots, w_{(4,n-5)}$ are equal to $w_1, w_2, w_3, w_5 \ldots, w_{n-4}$ respectively.  
\begin{gather} 
\left\{ 
\begin{aligned}\label{orb2t3}
u_{(5,1)} & = v_{s} + v_{l+1};\\
u_{(5,2)} & = v_{s}+ v_{s+1}+v_{l}+w_5;\\
u_{(5,2+r)} & =v_{s+1}+w_{(5,r)}\phantom{;} &\text{ for } r \in \{1, \ldots, m-3\};\\
u_{(5,m)} & =v_{s+1}+\sum_{r=m-2}^{n-5}w_{(5,r)}.
\end{aligned}
\right. 
\end{gather}
Here $w_{(5,1)}, \ldots, w_{(5,n-5)}$ are equal to $w_1, w_2, w_3, w_4, w_6, \ldots, w_{n-4}$ respectively.

 Notice that $n \ge 9$ and $n-m \ge 3$ is necessary for such a definition of $u_{(i,j)}$. 

Let $$\omega_i=(S,Sx,Sy,Sxy,Sz_i).$$ We first show that the $\omega_i$ are regular points of $\Omega^5$. Consider $\omega_1$.  The regularity of the remaining points can be shown using the same arguments. Let $t \in (S \cap S^x) \cap (S \cap S^x)^y \cap S^{z_1}$, so it takes shape \eqref{ugol} for some $\alpha_i, \beta_i \in \mathbb{F}_q$ and stabilises the subspace $Uz_1.$ Thus
\begin{equation}\label{b2t1}
(u_{(1,1)})t=\alpha_{s}v_s +\beta_s v_{s+1} + \alpha_{l+1} v_{l+1} +\beta_{l+1}v_l,
\end{equation}
since $t \in (S \cap S^x) \cap (S \cap S^x)^y$ and 
\begin{equation}\label{b2t2}
(u_{(1,1)})t=\eta_1 u_{(1,1)} +\eta_2 u_{(1,2)} + \ldots + \eta_m u_{(1,m)},
\end{equation}
since $t$ stabilises  $Uz_1.$ There is no $w_i$ for $i \ge 1$ in decomposition \eqref{b2t1}, so in \eqref{b2t2} we obtain
$$(u_{(1,1)})t=\eta_1 u_{(1,1)}=\alpha_s u_{(1,1)},$$
in particular $\beta_{s}=\beta_{l+1}=0$, so $t$ is diagonal. Therefore,
\begin{equation}\label{2b2t2}
(u_{(1,2)})t=\alpha_{s}v_s +\alpha_{s+1}v_{s+1}+ \alpha_{l} v_{l} +\alpha_{1}w_1,
\end{equation}
must also lie in $\langle u_{(1,1)}, \ldots, u_{(1,m)} \rangle$. There are no $v_{l+1}$ and $w_i$ for $i \ge 2$ in \eqref{2b2t2}, so 
$$(u_{(1,2)})t= \alpha_s u_{(1,1)}.$$ If $i>2$ then
 the same arguments show that 
$$(u_{(1,i)})t=\alpha_{s+1} u_{(1,i)}=\alpha_s u_{(1,i)},$$
for the remaining $i \in \{1, \ldots , m\},$ so $t$ is scalar. Therefore, $\omega_1$ is regular.

Now we claim that  the $\omega_i$ lie in distinct orbits of $\Omega^5.$ Here we prove that $\omega_1$ and $\omega_2$  lie in distinct orbits;  the remaining cases are similar.  Assume that $\omega_1 g =\omega_2$ for $g \in \GL_n(q),$ so 
$$g \in (S \cap S^x) \cap (S \cap S^x)^y \cap z_1^{-1}Sz_2.$$
Therefore, $g$ has shape \eqref{ugol} and $g=z_1^{-1}hz_2$ for $h \in S,$ so 
\begin{equation}\label{b2t3}
(u_{(1,2)})g=\alpha_{s}v_s +\beta_s v_{s+1} + \alpha_{s+1} v_{s+1} +\alpha_{l}v_l+\alpha_1 w_1,
\end{equation}
and 
\begin{equation*}
(u_{(1,2)})g=(u_{(1,2)})z_1^{-1}hz_2=(v_{n-m+2})hz_2 = \delta_1 u_{(2,1)} + \ldots + \delta_m u_{(2,m)}.
\end{equation*}
Since there is no $w_2$ in \eqref{b2t3}, $\delta_2$ is zero. Therefore, there must be no $v_l$ in \eqref{b2t3}, so $\alpha_l$ must be zero, which contradicts the existence of such an invertible $g.$ Thus, $\omega_1$ and $\omega_2$ lie in  distinct orbits. 

\medskip

{\bf Case (5.2).} Consider the case when $n \ge 9$ and the only $(2\times 2)$ block is in  rows $\{l-1,l\},$ where $l \le n/2$ if $n$ is even and $l \le (n-1)/2$ if $n$ is odd. We take as the $S$-invariant subspace $U$ the subspace with  basis $\{v_{l+1}, \ldots, v_n\}$, so
$m\ge [n/2]$ and $n-m \ge 2.$  If $n-m \ge 3$, then the proof is  as in {\bf Case (5.1).} 

 Consider the case  $n-m=2,$ so $n_k=2$. Let $x= \diag[x_k, \ldots, x_1]$, where $x_i\in \SL_{n_i}(q)$ is such that $S_i \cap S_i^{x_i} \le RT(\GL_{n_i}(q))$ if $n_i \ne 2$ and $x_i$ is the identity matrix if $n_i=2.$ Let $y$ be as in Lemma \ref{irrtog}. If $k=2$ in \eqref{stup}, then there exists $x_1 \in \SL_{n_1}(q)$ such that $S_1 \cap S_1^{x_1} \le D(\GL_{n_1}(q))$, by 
Theorem \ref{irred}, since $n_1=n-2\ge 7.$ Thus, $(S\cap S^x) \cap (S\cap S^x)^y \le D(\GL_n(q))$ and the rest of the proof is as in {\bf Case 1}.

Let $k>2$ and $n_1 \ne 1$. So $n_1 >2$, since the only $(2\times 2)$ block is in  rows $\{l-1,l\},$ where $l \le n/2$. 
We can take $U$ to be the subspace with  basis $\{v_{n -n_1+1}, \ldots, v_n\}$. Therefore, $n-m\ge 3$  and the proof is the same as the one using \eqref{orb2t1} -- \eqref{orb2t3}. 

Let $k>2$ and $n_1=1.$ Denote $\diag(-1,1\ldots,1) \cdot \per((2,n))$ by $\tilde{x}$. Calculations show that  $S \cap S^{x \tilde{x}} \le RT(\GL_n(q))$, so $(S \cap S^{x \tilde{x}}) \cap (S \cap S^{x \tilde{x}})^y\le D(\GL_n(q))$. The rest of the proof is  as in {\bf Case 1}, since we can take $U=V_{k-1}=\langle v_3, \ldots, v_n \rangle,$ so $n-m=2,$ $m\ge 2.$

\medskip

{\bf Case (5.3).}  Now let $n<9.$ If $n=3,$ then Theorem~\ref{lem41} follows by Proposition \ref{n3GL}, so $n \in \{4,5,6,7,8\}$.

Assume that $n_i=1$ for some $i\le k$. So there is a $(1\times 1)$ block in the row $j \le n$ of $g \in S$. Let $\tilde{x}$ be $\diag(\sgn(\sigma), 1, \ldots, 1) \cdot \per(\sigma)$ with $\sigma=(s+1,j)$ if $j>s+1$ and  $\sigma=(j,s,s+1)$ if $j<s.$ It is easy to see that 
$$S \cap S^{x \tilde x}\le RT(\GL_n(q)),$$
so the rest of proof is  as in {\bf Case 1}.

If we exclude the cases previously resolved, then $k \in \{2,3\}$ and it remains to handle the following list of possibilities:
\begin{enumerate}
\item[$(a)$] $k=2$ and $(n,n_1,n_2)$ is one of the following:
$$(5,2,3), (5,3,2), (6,2,4), (6,4,2), (7,2,5), (7,5,2), (8,2,6), (8,6,2)$$
\item[$(b)$] $k=3$ and $(n, n_1, n_2, n_3)$ is $(8,2,3,3)$ or $(8,3,3,2).$
\end{enumerate}
If $(n,q) \ne (6,3),$ then, by Theorem \ref{irred}, for all $i$ such that $n_i\ne 2$  there exist $x_i \in \SL_{n_i}(q)$ such that $S_i \cap S_i^{x_i} \le D(\GL_{n_i}(q))$. Let $x=\diag[x_k, \ldots, x_1]$, where $x_i$ is the identity matrix  if $n_i=2$. It is easy to check directly that $$(S \cap S^x) \cap (S \cap S^x)^y \le D(\GL_n(q))$$
for all cases above, so the rest of the proof is  as in {\bf Case 1}.

  For $(n,q)=(6,3)$ the result is established by computation. This concludes the proof of Proposition \ref{case5prop}.
\end{proof}
  
 Theorem \ref{lem41}  now follows from     Propositions \ref{case1prop} and \ref{case5prop}.
\end{proof}

\begin{Th}
\label{lem42}
If $k>1$,  $q=2$ and  there exists  $i \in \{1, \ldots, k\}$ such that $\gamma_i(S)$ is the normaliser of a Singer cycle of $\GL_3(2),$ then Theorem {\rm \ref{theorem}} holds.
\end{Th} 
\begin{proof}
Notice that $\GL_n(2)=\SL_n(2),$ so $S \cdot \SL_n(2)=\GL_n(2).$  

Since $k>1,$ $S$ stabilises a nontrivial invariant subspace $U$ of dimension $m<n$. Assume that matrices in $S$ take shape \eqref{stup} in the  basis $$\{v_1, v_2, \ldots, v_{n-m+1}, v_{n-m+2}, \ldots, v_n \}.$$

The main difference from the case $q>2$ is  that if $S_0$ is the normaliser of a Singer cycle of $\GL_3(2)$, then there is no $x, y \in \GL_3(2)$ such that $S_0^{x} \cap S_0^{y}$ is contained in $RT(\GL_3(2))$ (see Theorem \ref{irred}). Since all Singer cycles are conjugate in $\GL_3(2),$ we assume that $S_0$ is the normaliser of $$
\left\langle
\begin{pmatrix}
0&0&1\\
1&0&0\\
0&1&1
\end{pmatrix}
 \right\rangle.
$$ By Lemma \ref{lemRTdiagfield}, if $$x_0=\begin{pmatrix}
0&1&0\\
1&0&1\\
0&0&1
\end{pmatrix},
$$ then 
\begin{equation*}
S_0 \cap S_0^{x_0} = P =  \left\{
\begin{pmatrix}
1 & 1& 1\\
0& 1 & 0\\
1 & 0 & 0
\end{pmatrix},
\begin{pmatrix}
0 & 0& 1\\
0& 1 & 0\\
1 & 1 & 1
\end{pmatrix},
\begin{pmatrix}
1 & 0& 0\\
0& 1 & 0\\
0 & 0 & 1
\end{pmatrix} \right\}. 
\end{equation*}

By Theorem \ref{irred}, for $n_i \ge 3$   there exist $x_i\in \GL_{n_i}(2)$ such that $S_i \cap S_i^{x_i} \le Z(\GL_{n_i}(2))$ if $S_i$ is not conjugate to $S_0$ and $S_i \cap S_i^{x_0} \le P$ if $S_i$ is conjugate to $S_0$. Notice that $Z(\GL_{n_i}(2))=1.$  Let $x=\diag[x_k, \ldots, x_1]$, where $x_i$ is an identity matrix if $n_i = 2.$  Therefore, matrices in $S \cap S^x$ are upper triangular except for $(2 \times 2)$ and  $(3 \times 3)$ blocks on the diagonal. Let $y$ be the permutation matrix corresponding to the permutation $$(1,n)(2, n-1) \ldots (n/2, n/2+1)$$ if $n$ is even, and $$(1,n)(2, n-1) \ldots ((n-1)/2, (n+1)/2)$$ if $n$ is odd.

 Fix some $(3\times 3)$ block on the diagonal of matrices in $S$ such that $S_i=S_0$ is in this block, so $P$ is in this block in $S \cap S^x$.
 This block intersects  one, two or three blocks in $(S \cap S^x)^y.$  
If it intersects  at least two blocks, then the  matrices in $(S \cap S^x) \cap (S \cap S^x)^y$ have the following restriction  to the chosen $(3 \times 3)$ block:
$$
\begin{pmatrix}
*&*&0\\
*&*&*\\
*&*&*
\end{pmatrix}. 
$$ 
Since  the only such matrix in $P$ is the identity, every matrix in $(S \cap S^x) \cap (S \cap S^x)^y$ has the identity submatrix in this block.  

If the chosen $(3 \times 3)$ block intersects  a bigger block in $(S \cap S^x)^y$, then it must lie in the block with  $S_i \cap S_i^{x_i}$ for some $i$ such that $n_i>3$ and  all such matrices are scalar.

Let the chosen $(3 \times 3)$ block intersect another $(3 \times 3)$ block in $(S \cap S^x)^y$ which consists of matrices in $P$. By Theorem \ref{irred}, $b_{S_0}(\GL_3(2))=3,$ so there exists $y_0$ such that $$S_0 \cap S_0^{x_0} \cap S_0^{y_0} =1.$$ Let $\tilde{y}=\diag[1, \ldots, 1, y_0, 1, \ldots, 1]$ where $y_0$ is in the chosen block. Therefore,
 $$(S \cap S^x) \cap (S \cap S^x)^{y\tilde{y}}$$
consists of matrices which are diagonal except, possibly, for $(2 \times 2)$ blocks.

The rest of the proof is  as in Theorem \ref{lem41}. 
\end{proof}

Theorem \ref{theorem} now follows by Theorems \ref{lem41} and \ref{lem42}.

\section{Solvable subgroups not contained in $\GaL_n(q)$}
\label{sec412}

Recall that $V=\mathbb{F}_q^n$ and let $\beta=\{v_1, \ldots, v_n\}$ be a basis of $V$. Let $\Gamma= \GaL _n(q)= \GL_n(q) \rtimes \langle  \phi_{\beta} \rangle$
and $A=A(V)=A(n,q):=\Gamma \rtimes \langle \iota_{\beta} \rangle$ where $\iota_{\beta}$ is the inverse-transpose  map of $\GL_n(q)$ with respect to $\beta.$  Our goal is to prove the following theorem.

\begin{T4}
Let $n\ge 3.$ If $S$ is a maximal solvable subgroup of $A(n,q)$ not contained in $\Gamma,$ then one of the following holds:
\begin{enumerate}
\item[$(1)$] $b_S(S \cdot \SL_n(q))\le 4$;
\item[$(2)$] $n=4,$ $q \in \{2,3\}$, $S$ is the normaliser in $A(n,q)$ of the stabiliser in $\GaL_n(q)$ of a $2$-dimensional subspace of $V$, $b_S(S \cdot \SL_4(q))=5$ and $\Reg_S(S \cdot \SL_4(q),5)\ge 5.$
\end{enumerate}
\end{T4} 

Before we start the proof, let us discuss the structure of a maximal solvable subgroup
$S$ of $A(n,q)$ and fix some notation. In this section, we assume  $S$ is not contained in $\Gamma.$

Consider the action of $\Gamma$ on the set $\Omega_1$ of subspaces of $V$ of dimension $m<n.$ This action is transitive and equivalent to the action of $\Gamma$ on the set
$$\Omega_1'=\{\mathrm{Stab}_{\Gamma}(U) \mid U<V, \dim U= m\}$$ by conjugation. Let $\Omega_2$ be the set of subspaces of $V$ of dimension $n-m$ and let  $$\Omega_2'=\{\mathrm{Stab}_{\Gamma}(W) \mid W<V, \dim W= n-m\}.$$ It is easy to see that $\Gamma$ acts on $\Omega=\Omega_1 \cup \Omega_2$ with orbits $\Omega_1$ and $\Omega_2$ (respectively, on $\Omega'=\Omega_1' \cup \Omega_2'$ with orbits $\Omega_1'$ and $\Omega_2'$). Notice that if $m=n/2$, then $\Omega_1=\Omega_2$ and there is only one orbit. We can extend this action to an action of $A$ on $\Omega'$ by conjugation which is equivalent to the following action of $A$ on $\Omega$:  for $U,W \in \Omega \text{ and } \varphi \in A$,
$$ U\varphi =W \text{ if and only if } 
(\mathrm{Stab}_{\Gamma}(U))^{\varphi}=\mathrm{Stab}_{\Gamma}(W).$$ 
In particular, if $U= \langle v_{n-m+1}, v_{n-m+2}, \ldots, v_n \rangle$, then $U\iota_{\beta}=\langle v_1, \ldots, v_{n-m}\rangle:=U'$ of dimension $n-m.$ Moreover, if $\varphi=\iota_{\beta}g$ with $g \in \Gamma$ and $h \in \Gamma,$ then 
$$(Uh)\varphi=U\iota_{\beta}h^{\iota_{\beta}}g=U'h^{\iota_{\beta}}g.$$  So elements of $A\backslash \Gamma$ interchange $\Omega_1$ and $\Omega_2.$

Let us now define the action of $A$ on the pairs of subspaces $(U,W)$ of $V$ with $\dim U=m \le n/2$ and $\dim W=n-m$ where either
\begin{itemize}
\item  $U<W$, or
\item $U\cap W= \{0\}.$
\end{itemize}
 Here we let $(U,W) \varphi =(U\varphi, W \varphi).$ The pair $(U, W)$ is not ordered, but for convenience we usually list first the subspace of smaller dimension.  Notice that this action is equivalent to the action of $A$ by conjugation on 
 \begin{itemize}
\item {$\{\mathrm{Stab}_{\Gamma}((U,W)) \mid U \le W < V, \dim U=m, \dim W=n-m\}$};
\item $\{\mathrm{Stab}_{\Gamma}((U,W)) \mid U , W < V, U \cap W = \{0\}, \dim U=m, \dim W=n-m\}$ 
 \end{itemize}
respectively
where $\mathrm{Stab}_{\Gamma}((U,W))=\mathrm{Stab}_{\Gamma}(U) \cap \mathrm{Stab}_{\Gamma}(W).$

\begin{Def}
\label{deftypepm}
 Let $M$ be the stabiliser in $A$ of a pair $(U,W)$ of subspaces of $V$ where $\dim U=m \le n/2$, $\dim W=n-m$. 
\begin{itemize}
\item If $U \le W$, then $M$ is a subgroup \emph{of type} $P_{m,n-m};$
\item If $U\cap W=\{0\}$, then $M$ is a subgroup \emph{of type} $\GL_m(q) \oplus \GL_{n-m}(q).$
\end{itemize}
We denote the class of subgroups of $A$ as above by $\mathcal{C}_1^*.$
\end{Def}
We follow \cite[\S 4.1]{kleidlieb} for the definition of ``type''. Let $S$  be a maximal solvable subgroup of $A$ not contained in $\Gamma$ such that $S \le M$ where $M$ is a maximal subgroup of $A$ contained in Aschbacher's class $\mathcal{C}_1$ (here we use the definition of $\mathcal{C}_1$ in \cite{kleidlieb}, while Aschbacher \cite{asch} uses the notation $\mathcal{C}_1'$ to include such subgroups). By \cite[\S 4.1]{kleidlieb}, $M$ is as in Definition \ref{deftypepm}.  Notice that  $\mathcal{C}_1$ does not contain the subgroups of type $\GL_{n/2}(q) \oplus \GL_{n/2}(q)$ when $n$ is even as such a subgroup lies in the maximal subgroup stabilising the decomposition $V=U \oplus W$ which is in Aschbacher's class $\mathcal{C}_2.$ We define $\mathcal{C}_1^*$ to include such subgroups as it suits our purposes better.
In \cite[\S 4.1]{kleidlieb}, the type $P_{m,n-m}$ is used only when $m<n/2;$ when $m=n/2$ such $M$ are labelled $P_m$ (since they are parabolic subgroups stabilising an $m$-subspace). We let $m \le n/2$ and use the label $P_{m,n-m}$ since it allows us to present more uniform statements.

Let  $M$ be as in Definition \ref{deftypepm} and let $\beta=\{v_1, \ldots, v_n\}$ be a basis of $V$ such that $U=\langle v_{n-m+1}, \ldots, v_n \rangle$ and 
$$W= \begin{cases}
\langle v_{m+1}, \ldots, v_n\rangle &\text{ if } M \text{ is of type } P_{m, n-m};\\
\langle v_1, \ldots, v_{n-m}\rangle &\text{ if } M \text{ is of type } \GL_{m}(q) \oplus \GL_{n-m}(q).
\end{cases} 
$$ 
For such $\beta$ we say that it is \emph{associated} with $(U,W).$
 By Definition \ref{deftypepm}, $M$ is the normaliser of $\mathrm{Stab}_{\Gamma}(U,W)$ in $A$. Therefore, with respect to $\beta$,
\begin{equation}
\label{Mstr}
M =\begin{cases}
\mathrm{Stab}_{\Gamma}(U,W) \rtimes \langle \iota_{\beta}a(n,m) \rangle & \text{ if } M \text{ is of type } P_{m, n-m};\\
\mathrm{Stab}_{\Gamma}(U,W) \rtimes \langle \iota_{\beta} \rangle & \text{ if } M \text{ is of type } \GL_{m}(q) \oplus \GL_{n-m}(q)
\end{cases}
\end{equation}
where $a(n,m) \in \GL_n(q)$ is 
\begin{equation}
\label{adef}
\begin{pmatrix}
0 & 0 & I_{m} \\
0 & I_{n-2m} & 0 \\
I_{m} & 0 & 0
\end{pmatrix}.
\end{equation}

\begin{Lem}
\label{GRmmim2}
Let $n\ge 3$ and let $S \le A$ be a maximal solvable subgroup that is not contained in $\Gamma$. If $m$ is the least integer such that $S$ lies in $M$ as in \eqref{Mstr}, then 
 $\tilde{S}=S \cap \Gamma$ acts  on $U$ irreducibly.
\end{Lem} 
\begin{proof}
Let $\beta$ and $M$ be as in \eqref{Mstr} and let $P$ be $\mathrm{Stab}_{\Gamma}(U,W),$ so 
$$M=P \rtimes \langle \iota_{\beta} a\rangle$$
where $a$ is $a(n,m)$ if $M$ is of type $P_{m,n-m}$ and $a=I_n$ otherwise.

Assume that $\tilde{S}$ acts  reducibly on $U$, so  $\tilde{S}$ stabilises $U_1 <U$ of dimension $s<m.$  Assume that $U_1$ is a minimal such subspace, so $\tilde{S}$ stabilises no proper non-zero subspace of $U_1.$ Let $\varphi \in S \backslash \tilde{S}$ be such that
$S=\langle \tilde{S}, \varphi \rangle$, so $\varphi=\iota_{\beta}a \cdot g$ with $g \in P.$  Hence $\tilde{S}$ stabilises $W_1=U_1 \varphi$ of dimension  $n-s$. Notice that $U_1<W_1$ if $M$ is of type $P_{m,n-m}$ and $W_1 \cap U_1=0$ otherwise.
Since $\varphi^2 \in \Gamma,$ we obtain $\varphi^2 \in \tilde{S}$ and   $$W_1\varphi=U_1\varphi^2=U_1.$$
Therefore, $S$ normalises the stabiliser of $(U_1, W_1)$ in $\Gamma,$ so $S$ lies in a  subgroup of $A$ of type $P_{s, n-s}$ or $\GL_{s}(q) \oplus \GL_{n-s}(q)$  which contradicts the assumption of the lemma. 
\end{proof}

Notice that if we take $U=W=V$, then the proof of Lemma \ref{GRmmim2} implies the following statement.
\begin{Lem}
\label{GRirred2}
If $S \le A$ is not contained  in a  subgroup of $A$ from the class $\mathcal{C}_1^*$, then $\tilde{S}=S \cap \Gamma$ acts irreducibly on $V$.
\end{Lem}

\begin{Th}
Theorem {\rm \ref{theoremGR}} holds if $S$ is not contained  in a subgroup of $A$ from the class $\mathcal{C}_1^*$.
\end{Th}
\begin{proof}
Let $\hat{S}:=S \cap \GL_n(q).$ Lemmas \ref{GammairGL}  and \ref{GRirred2} and Theorem \ref{irred} imply that either there exists $x \in \SL_n(q)$ such that $\hat{S} \cap \hat{S}^x \le Z(\GL_n(q))$, or $\hat{S}$ lies in one of the groups in \ref{irred14}, \ref{irred15} of Theorem \ref{irred}. In the latter case the statement is verified by computation. 

Otherwise, $\overline{S} \cap \overline{S}^{\overline{x}}$ is an abelian subgroup of $G=(S \cdot \SL_n(q))/Z(\GL_n(q))$ where $\overline{\phantom{G}}:S \cdot \SL_n(q) \to G$ is the natural homomorphism.  By Theorem \ref{zenab}, there exist $\overline{y} \in G$ such that $\overline{S} \cap \overline{S}^{\overline{x}} \cap (\overline{S} \cap \overline{S}^{\overline{x}})^{\overline{y}}=1.$ So $b_S(S \cdot \SL_n(q)) \le 4$ and the statement follows. 
\end{proof}

For the rest of the section we assume that $S$ lies in a subgroup of $A$ from the class $\mathcal{C}_1^*$.

 Let $m_1$ be minimal such that $S$ lies in a  subgroup $M_1<A$ of type $P_{m_1, n-m_1}$ or $\GL_{m_1}(q) \oplus \GL_{n-m_1}(q)$ stabilising subspaces $(U_1,W_1)$. Let $V_1$ be $V$;  let $V_2=W_1/U_1$ and $n_2=\dim V_2$. Let $\beta$ be a basis associated with $(U_1,W_1).$ Notice that elements from $\mathrm{Stab}_{\Gamma}(U_1,W_1)$ induce semilinear transformations on $V_2$, and $\iota_{\beta}$ induces the inverse-transpose map on $\GL(V_2).$  Hence there exists a homomorphism 
$$\psi: S \to A(V_2) $$ mapping $x \in S$ to the element it induces on $V_2.$ Denote $\psi(S)$ by $S|_{_{V_2}}.$ Let $m_2$ be minimal such that $S|_{_{V_2}}$ lies in a  subgroup $M_2<A(V_2)$ of type $P_{m_2, n_2-m_2}$, or $\GL_{m_2}(q) \oplus \GL_{n_2-m_2}(q)$ stabilising subspaces $(U_2/U_1,W_2/U_1)$ of $V_2$.

 Repeating the arguments above we obtain a chain of subspaces 
\begin{equation}
\label{GRUs}
 0=U_0<U_1< \ldots < U_k<V \text { with } \dim U_i/U_{i-1} =m_i \text{ for } i \in\{1 \ldots, k\},
\end{equation} 
 subspaces $W_i$  for $i \in \{0, 1, \ldots, k\}$ where $W_0=V$, and groups $S|_{V_i}$ where $V_i=W_{i-1}/U_{i-1}$ for $i \in \{1, \ldots, k+1\}.$ Here $S|_{V_{i}}$ stabilises $(U_{i}/U_{i-1}, W_{i}/U_{i-1})$ for $i\in \{1, \ldots, k\}$ and $S|_{V_{k+1}}$ stabilises no subspace of $V_{k+1}.$ If $U_k=W_k,$ then $V_{k+1}$ is a zero space and  $S|_{V_{k+1}}$ is trivial.  Let $\beta=\{v_1, \ldots, v_n\}$ be a basis of $V$ such that  $\beta_i=\{v_1 +U_{i-1}, \ldots, v_n +U_{i-1}\} \cap V_i$ is a basis associated with $(U_{i}/U_{i-1}, W_{i}/U_{i-1})$ for $i\in \{1, \ldots, k\}.$

Let $\varphi \in S$. Hence $\varphi|_{_{V_i}}$ lies in $M_i$ for $i=1, \ldots, k$, so $$\varphi|_{_{V_i}}= (\iota_{\beta_i} a_i)^l \cdot g_i \text{ with } l\in \{0,1\} \text{ and } g_i \in \mathrm{Stab}_{\GaL (V_i)}((U_{i}/U_{i-1}, W_{i}/U_{i-1})),$$
where $a_i$ is $a(n_i,m_i)$  as in \eqref{adef} if $M_i$ is of type $P_{m_i, n_i-m_i}$ and $a_i$ is $I_{n_i}$ otherwise. Let $a \in \GL_n(q)$ be such that  $a|_{_{V_k}}=a_k$ and 
\begin{equation}
\label{adefGR}
a|_{_{V_i}}= \begin{cases}
\begin{pmatrix}
a|_{_{V_{i+1}}} & 0\\
0 & I_{m_{i}}
\end{pmatrix} &\text{ if } M_i \text{ is of type } \GL_{m_i}(q) \oplus \GL_{n_i-m_i}(q)\\
\begin{pmatrix}
0 & 0 & I_{m_i}\\
0 & a|_{_{V_{i+1}}} & 0\\
I_{m_i} & 0 & 0
\end{pmatrix} &\text{ if } M_i \text{ is of type } P_{m_i, n_i -m_i}
\end{cases}
\end{equation}
 for $i \in \{1, \ldots, k-1\}.$  Therefore, 
\begin{equation}
\label{GRvarphishape}
\varphi = (\iota_{\beta}a)^l \cdot (\phi_{\beta})^j \cdot g,
\end{equation}
where $g \in \GL_n(q)$ and $g|_{_{V_i}}$ stabilises $(U_{i}/U_{i-1}, W_{i}/U_{i-1}).$ More specifically, let $s \le k$ be the number of $i \in \{1, \ldots, k\}$ such that $M_i$ is of type $P_{m_i, n_i-m_i}$ and let $i_1 < \ldots < i_s$ be the corresponding $i$-s. Therefore,
\begin{equation}
\label{GRguppdiag}
g= \begin{pmatrix}
g_{i_1} '&* &\ldots &\ldots &\ldots & &* \\
 & \ddots &* &\ldots &\ldots& \ldots &* \\
 & &g_{i_s}' &* &\ldots & &* \\
 & & &g_{k+1} &* & &* \\
 & & & & g_k &* & *\\
 & & & & & \ddots &*\\
& & & & & & g_1
\end{pmatrix}
\end{equation} 
where $g_i, g_i' \in \GL_{m_i}(q)$ and $g_{k+1} \in \GL_{n_{k+1}}(q).$

\begin{example}
Let $k=3,$ and let $M_1$ and $M_3$ be of types $P_{m_i, n_i-m_i}$ for $i=1$ and $i=3$ respectively. Let $M_2$ be of type $\GL_{m_2}(q) \oplus \GL_{n_2-m_2}(q).$ Then 
$$
a=\begin{pmatrix}
 & & & & &  I_{m_1} \\
 &{\cellcolor{gray!50}} &{\cellcolor{gray!50}} &{\cellcolor{gray!50}} I_{m_3} &{\cellcolor{gray!20}} & \\
 & {\cellcolor{gray!50}} &{\cellcolor{gray!50}} I_{n_3-2m_3} &{\cellcolor{gray!50}} &{\cellcolor{gray!20}} & \\
 &{\cellcolor{gray!50}} I_{m_3} &{\cellcolor{gray!50}} & {\cellcolor{gray!50}} &{\cellcolor{gray!20}} & \\
 & {\cellcolor{gray!20}} & {\cellcolor{gray!20}} &{\cellcolor{gray!20}} &{\cellcolor{gray!20}} I_{m_2} & \\
I_{m_1} & & & & & 
\end{pmatrix}
$$
and $$\varphi = (\iota_{\beta}a)^l \cdot (\phi_{\beta})^j \cdot \begin{pmatrix}
 g_1 '&* &* &* &* &* \\
 & g_3' &* &* &0 &* \\
 & &g_4 &* &0 &* \\
 & & &g_3 &0 &* \\
 & & & & g_2 &* \\
 & & & & & g_1
\end{pmatrix}, $$
where $g_1, g_1' \in \GL_{m_1}(q),$ $g_2 \in \GL_{m_2}(q),$ $g_3, g_3' \in \GL_{m_3}(q),$ $g_4 \in \GL_{n_3-2m_3}(q).$
\end{example}

\begin{Lem}
\label{GRdiag}
Let $n \ge 3$. Let  $S$ be a maximal solvable subgroup of $A$. Assume that $S$ is contained in a subgroup of $A$ of type $P_{m,n-m}$ or $\GL_m(q) \oplus \GL_{n-m}(q)$ for some $m \le n/2.$ Let $\beta$ be as described after \eqref{GRUs}.
\begin{enumerate}
\item[$(1)$] If none of $(m_i, q)$ and $(n_i,q)$ lies in $\{(2,2), (2,3), (2,5)\},$ then there exist $x, y \in \SL_n(q)$ such that if $\varphi \in (S \cap S^x \cap S^y) \cap \Gamma,$ then 
\begin{equation}
\label{GRvarphdiagc1}
\varphi=(\phi_{\beta})^j \cdot \diag[g_{i_1}, \ldots,g_{i_s}, g_{k+1}, g_k, \ldots,  g_1]=(\phi_{\beta})^j \cdot \diag(\alpha_1, \ldots, \alpha_n)
\end{equation}
 where $\alpha_i \in \mathbb{F}_q^*$ and $j \in \{0,1, \ldots, f-1\}.$ Moreover, $g_i, g_i' \in Z(\GL_{m_i}(q))$ for $i \in \{1, \ldots, k\}$ and $g_{k+1} \in Z(\GL_{n_{k+1}}(q))$.
\item[$(2)$] If $q \in \{2,3,5\}$ and at least one of $m_i$ or $n_i$ is $2$, then there exist $x,y \in \SL_n(q)$ such that if $g \in S \cap S^x \cap S^y \cap \GL_n(q)$, then $$g=\diag[g_{i_1}, \ldots,g_{i_s}, g_{k+1}, g_k, \ldots,  g_1]$$ with $g_i, g_i' \in \GL_{m_i}(q)$ for $i \in \{1, \ldots, k\},$ $g_{k+1} \in \GL_{n_{k+1}}(q)$, and one of the following holds:
\begin{enumerate} 
\item[$(2a)$]  for each $i \in \{1, \ldots, k+1\}$, either  $g_i$, $g_i'$ are  scalar matrices over $\mathbb{F}_q$ or $m_i=2$ and  $g_i$, $g_i'$ are upper-triangular matrices in $\GL_2(q);$
\item[$(2b)$] there exist exactly one $j \in \{1, \ldots, k+1\} \backslash \{i_1, \ldots, i_s\}$ such that $g_j \in \GL_2(q),$ and $g_i$, $g_i'$ are scalar for $i \in \{1, \ldots, t\} \backslash \{j\}.$
\end{enumerate}
\end{enumerate}
\end{Lem}
\begin{proof}
{\bf(1)} We start with the proof of $(1),$ so neither $(m_1,q)$ nor $(n_2,q)$ lies in $\{(2,2), (2,3),(2,5)\}.$ 
Let us fix $q$ and assume that $n$ is minimal such that there exists $S\le A$ for which the statement of the lemma does not hold. Let $K$ be $S|_{_{U_1}} \le A(U_1)$ and let $R$ be $S|_{_{V_2}} \le A(V_2).$

We claim that there exist $x_2, y_2 \in \SL_{n_2}(q)$ such that $(1)$ holds for $R \le A(V_2)$. Indeed, if $k>1$, then $R$ stabilises $(U_2,W_2)$ and $R$ is not a counterexample to the lemma since $n_2<n.$ If $k=1$, then $\tilde{R}=R \cap \Gamma$ acts irreducibly on $V_2$ by Lemma \ref{GRirred2}.  Hence, by  Lemma \ref{GammairGL} and Theorem \ref{irred},   there exist $x_2,y_2 \in \SL(V_2)$ such that $\tilde{R} \cap \tilde{R}^{x_2} \cap \tilde{R}^{y_2} \cap \GL(V_2)= Z(\GL(V_2)).$ Now the claim follows by Lemma \ref{scfield}.

By the same argument, $\tilde{K}=K \cap \GaL_{m_1}(q)$ acts irreducibly on $U_1$, and, by Lemma \ref{GammairGL},   $K \cap \GL_{m_1}(q)$ lies in a subgroup $T$ of $\GL_{m_1}(q)$ such that  either $T$ is an  irreducible maximal solvable subgroup  or there exists $x_1 \in \SL_{m_1}(q)$ such that $T \cap T^{x_1} \le Z(\GL_{m_1}(q))$.

\medskip

  Further, our proof of $(1)$ splits into two cases: when $U_1 \le W_1$ and  $U_1\cap W_1=\{0\}$ respectively. We provide the proof for the first case, the proof for the second case is analogous. 
\medskip

 Let $U_1 \le W_1,$ so $n_2=n-2m_1$ and $S$ lies in a  subgroup $M_1$ of $A$ of type $P_{m_1, n-m_1}.$ Recall that, by \eqref{Mstr}, elements in $M_1$ have shape 
\begin{equation}
\label{GRPmnmdiag}
(\iota_{\beta} a)^l \cdot (\phi_{\beta})^j \cdot 
\begin{pmatrix}
h_1' & * & *\\
 0   & h_2 & *\\
 0   & 0 & h_1
\end{pmatrix}
\end{equation}
with $l \in \{0,1\}$, $j \in \{0, 1, \ldots, f-1\}$, $h_2 \in \GL_{n_2}(q)$ and $h_1, h_1' \in \GL_{m_1}(q).$ Let $K^\dagger$ be the restriction of $S$ on $V/W_1$ and let $\hat{K}$ and $\hat{K}^\dagger$ be $K \cap \GL_{m_1}(q)$ and $K^\dagger \cap \GL_{m_1}(q)$ respectively.

 Recall that $S=\langle \tilde{S}, \varphi \rangle,$ where $\varphi = \iota_{\beta} a \cdot g$ with $g \in \tilde{M_1}=M_1 \cap \GaL_n(q),$ so 
$$g= (\phi_{\beta})^{j_g} \cdot 
\begin{pmatrix}
g_3 & * & * \\
0  & g_2 & *\\
0 & 0 & g_1 \\
\end{pmatrix}
$$ with $g_1, g_3 \in \GL_{m_1}(q),$ $g_2 \in \GL_{n_2}(q)$ and $j_g \in \{0,1, \ldots, f-1\}.$ 

Since $(S \cap \GL_n(q))^{\varphi}=S \cap \GL_n(q),$  $$\hat{K}^\dagger=(\hat{K}^{\top})^{\phi^{j_g} g_3}.$$ Recall that $\hat{K}$ lies in a subgroup $T$ of $\GL_{m_1}(q)$ as described in the third paragraph of the proof.  Let $T^\dagger=(T^{\top})^{\phi^{j_g} g_3},$ so $\hat{K}^\dagger\le T^\dagger.$ 

By Lemma \ref{GammairGL} and Theorem \ref{irred}, there exist $x_1, y_1 \in \SL_{m_1}(q)$ such that $$T \cap T^{x_1} \cap T^{y_1}\le Z(\GL_{m_1}(q)).$$ Let $x_1'= ((x_1^{-1})^{\top})^{\phi^{j_g} g_3}$ and let $y_1'= ((y_1^{-1})^{\top})^{\phi^{j_g} g_3}$,  so $$T^\dagger \cap (T^\dagger)^{x_1'} \cap (T^\dagger)^{x_1'} \le Z(\GL_{m_1}(q)).$$  If $T$ is not as $S$ in  \ref{irred11} -- \ref{irred15} of Theorem \ref{irred}, then we assume $y_1=y_1'=I_{m_1}.$ So, by Lemma \ref{scfield}, we can assume that 
\begin{align*}
K \cap K^{x_1} \cap K^{y_1} \cap \GaL_{m_1}(q) \le \langle \phi_{\beta_1} \rangle Z(\GL_{m_1}(q));\\
K^\dagger \cap (K^\dagger)^{x_1} \cap (K^\dagger)^{y_1} \cap \GaL_{m_1}(q) \le \langle \phi_{\beta_3} \rangle Z(\GL_{m_1}(q)),
\end{align*}
 where $\beta_1$ is the basis of $U_1$ consisting of the last $m_1$ vectors of $\beta$ and $\beta_{3}=\{v_1 +W_1, \ldots, v_{m_1} + W_1\}$ is a basis of $V/W_1.$

If $T$ is as $S$ in  \ref{irred11} -- \ref{irred15} of Theorem \ref{irred}, then we claim that $T$ and $T^\dagger$ are conjugate by an element of $\GL_{m_1}(q).$ Indeed, if $T$ is as $S$ in \ref{irred11}, \ref{irred12}, \ref{irred14}  then both $T$ and $T^\dagger$ are normalisers of Singer cycles which are conjugate by Lemma \ref{singconj}. If $T$ is as $S$ in \ref{irred135} or \ref{irred13},  then $T$ and $T^\dagger$ are conjugate by \cite[\S 21, Theorem 6 $(1)$]{sup}.  If $T$ is as $S$ in  $(5),$ then $T=(T^{\top})^{\phi^{j_g}},$ so $T^\dagger=T^{g_3}$ with $g_3 \in \GL_{m_1}(q).$ Notice that $\Det(T)=\mathbb{F}_q^*$ if $T$ is not as $S$ in \ref{irred135}:  it is easy to check directly. If $T$ is as $S$ in \ref{irred135}, then $\Det(T)$ consists of all squares in $\mathbb{F}_q^*,$ in particular $\pm 1 \in \Det(T).$ Hence, for $\alpha = \pm 1$ and a given  $t_1 \in \GL_{m_1}(q)$ such that  $T^\dagger=T^{t_1}$ there exists $t_2 \in \GL_{m_1}(q)$ such that  $(T^\dagger)^{t_2}=T$ and $\det(t_1)\det(t_2)=\alpha$: we can take $t_2=\alpha t_1^{-1}.$

Let $x= \diag[x_1', x_2, x_1]$, let 
$$y=
\begin{pmatrix}
0 & 0 & t_2 y_1 \\
0 & y_2 & 0 \\
t_1 y_1'& 0 & 0
\end{pmatrix}$$
where $t_1, t_2 \in \GL_{m_1}(q)$ are such that $\det(y)=1$, and 
$$T^{t_1}=T^\dagger \text{ and } (T^\dagger)^{t_2}=T \text{ for } T \text{ as } S \text{ in \ref{irred11} -- \ref{irred15} of Theorem \ref{irred}.}$$
 So $x,y \in \SL_n(q).$ Such $t_1,t_2$ exist by the previous paragraph and since $\det(y_1)= \det(y_2)=1.$

Consider $h \in S \cap S^x \cap S^y \cap \Gamma.$ Using \eqref{GRPmnmdiag}, we obtain that elements in $S^y$ have shape 
\begin{equation}
\label{GRSy}
(\iota_{\beta} a)^l \cdot (\phi_{\beta})^j \cdot 
\begin{pmatrix}
h_1' & 0 & 0\\
 *   & h_2 & 0\\
 *   & * & h_1
\end{pmatrix}
\end{equation}
with $l \in \{0,1\}$, $j \in \{0, 1, \ldots, f-1\}$, $h_2 \in \GL_{n_2}(q)$ and $h_1, h_1' \in \GL_{m_1}(q).$ Therefore,  
by \eqref{GRPmnmdiag},
$$h= (\phi_{\beta})^j \cdot \diag[h_1',h_2,h_1],$$ so
\begin{align*}
h|_{_{V_2}} & =(\phi_{\beta_2})^j \cdot h_2 \le (R \cap R^{x_2} \cap R^{y_2}) \cap \GaL_{n_2}(q);\\  
h|_{_{U_1}} & =(\phi_{\beta_1})^j \cdot h_1 \le (K \cap K^{x_1} \cap K^{y_1}) \cap \GaL_{m_1}(q);\\
h|_{_{V/W_1}} & =(\phi_{\beta_3})^j \cdot h_1' \le (K^\dagger \cap (K^\dagger)^{x_1'} \cap (K^\dagger)^{y_1'}) \cap \GaL_{m_1}(q)
\end{align*} 
where $\beta_2=\{v_{m_1+1} +U_1, \ldots, v_{n-m_1}+U_1\}$ is a basis of $V_2.$ Therefore, $h_1$ and $h_1'$ are scalar matrices, and  $h_2$ is as $\varphi$ in \eqref{GRvarphdiagc1}, so $h$ has the shape claimed by the lemma. Hence $S$ is not a counterexample which contradicts the assumption in the beginning of the proof.
 This concludes the proof of part $(1)$ of Lemma \ref{GRdiag}.

\bigskip

{\bf (2)} Now we start the proof of part $(2)$ of Lemma \ref{GRdiag}. So $q \in \{2,3,5\}$ and at least one of $m_i$ and $n_i$ is $2$ for $i \in \{1, \ldots, k\}$. Notice that $f=1$, so $\phi_{\beta}$ is trivial and $\Gamma =\GL_n(q)$. Therefore, matrices in $S \cap \Gamma$ are block-upper-triangular and have shape \eqref{GRguppdiag} where 
\begin{align*}
g_i \in S_i & :=S|_{_{U_i/U_{i-1}}} \cap \GL_{m_i} (q)  &&\text{ for } i\in \{1, \ldots, k\};\\
g_{k+1} \in S_k & := S|_{_{V_{n_{k+1}}}} \cap \GL_{n_{k+1}} (q); &&\\
g_i \in S_i^\dagger & :=S|_{_{V_i/W_{i}}} \cap \GL_{m_i} (q) && \text{ for } i\in \{i_1, \ldots, i_s\}.
 \end{align*}
Denote $\GL_{m_i}(q)$ by $G_i$ for $i \in \{1, \ldots, k\}$ and $\GL_{n_{k+1}}(q)$ by $G_{k+1}.$ By Lemma \ref{GRirred2}, $S_i$ and $S_i^\dagger$ are irreducible solvable subgroup of $G_i$ for $i \in \{1, \ldots, k+1\}$. By Theorem \ref{irred}, for $i \in \{1, \ldots, k+1\}$ one of the following holds:
\begin{enumerate}[label=(\roman*)]
\item $G_i=\GL_2(q);$ \label{RGlistA} 
\item there exists $x_i \in G_i$ of determinant $1$ such that $S_i \cap S_i^{x_i} \le Z(G_i);$ \label{RGlistB}
\item $(G_i, S_i) \in \{(\GL_4(3),\GL_2(3) \wr \Sym(2)), (\GL_3(2), N_{\GL_3(2)}(T))\}$ where $T$ is a Singer cycle of $\GL_3(2)$ and there exists $x_i, z_i \in G_i$ such that $S_i \cap S_i^{x_i} \cap S_i^{z_i} \le Z(G_i)$.\label{RGlistC}
\end{enumerate}
The same statement is true for $S_i^\dagger$ and we denote corresponding conjugating elements by $x_i'$ and $z_i'$. If $G_i=\GL_2(q)$, then we take $x_i=x_i'=I_2.$

 If condition \ref{RGlistC}  holds for $S_i,$ then, by Theorem \ref{irred}, as discussed in the proof of part $(1)$, $S_i^\dagger$ is conjugate to $S_i$ in $G_i$. 

Let us define $y_i, y_i' \in G_i$ for $i \in \{1, \ldots, k+1\}.$ 
 If $M_i$ is of type $\GL_{m_i}(q) \oplus \GL_{n_i-m_i}(q)$ or $i=k+1$ then let $y_i \in G_i$ be such that $S_i \cap S_i^{x_i} \cap S_i^{y_i} \le Z(G_i).$ If $M_i$ is of type $P_{m_i, n_i-m_i},$ then let $y_i, y_i' \in G_i$ be such that 
\begin{align*}
S_i \cap S_i^{x_i} \cap (S_i^\dagger)^{y_i'} & \le Z(G_i)\\
S_i^\dagger \cap (S_i^\dagger)^{x_i'} \cap (S_i)^{y_i} & \le Z(G_i).
\end{align*} 

If conditions \ref{RGlistA} or \ref{RGlistB} hold for $S_i$ (respectively $S_i^\dagger$), then we take $y_i$ and $y_i'$ to be the identity matrix in $G_i.$ Let 
$$y=\diag[y_{i_1}', \ldots, y_{i_s}', y_{k+1}, \ldots, y_k] \cdot a.$$
If $\det (y) \ne 1$, then $q=3$ and $\det(y)=-1.$ If so, then we pick  $i \in \{1, \ldots, k+1\}$  such that $G_i=\GL_2(3)$ and change $y_i$ to be $\diag(-1,1)$ which is equivalent to multiplying a row of $y$ by $-1.$  Therefore, we can assume that $y \in \SL_n(q).$   

Let $r$ be the number of $G_i$ and $G_i^\dagger$ equal to $\GL_2(q)$. Our proof of $(2)$ splits into two cases: 
when $r \ge 2$
and $r=1$ respectively.

\medskip

{\bf Case (2.1).}  
  Let    $(2 \times 2)$ blocks (corresponding to the $S_i$ and the $S_i^\dagger$ lying in $\GL_2(q)$) on the diagonal in matrices of $\tilde{S}=S \cap \Gamma$ occur in the rows $$(j_1, j_1+1), (j_2, j_2 +1), \ldots, (j_r, j_r+1).$$ 
Let $\tilde{x}=\diag(\sgn(\sigma), 1, \ldots, 1) \cdot {\rm perm}(\sigma) \in \SL_n(q),$ where  $$\sigma=(j_1, j_1+1, j_2, j_2 +1, \ldots, j_r, j_r+1)(j_1, j_1+1).$$  
Let $x= \diag[x_{i_1}',\ldots, x_{i_s}', x_{k+1},  \ldots, x_1]\tilde{x}$. Notice that $\det x=1.$ Calculations show that if $h \in \tilde{S} \cap \tilde{S}^x \cap \tilde{S}^y$, then 
$$h=\diag[h'_{i_1}, \ldots, h'_{i_s}, h_{k+1}, \ldots, h_k]$$
where
\begin{itemize}
\item $h_i \in S_i \cap S_i^{x_i} \le Z(G_i)$ (respectively $h_i' \in S_i^\dagger \cap (S_i^\dagger)^{x_i} \le Z(G_i)$) if condition \ref{RGlistB} holds for $S_i$ (respectively $S_i^\dagger$);
\item $h_i \in S_i \cap S_i^{x_i} \cap (S_i^\dagger)^{y_i'} \le Z(G_i)$ (respectively $h_i' \in S_i^\dagger \cap (S_i^\dagger)^{x'_i} \cap (S_i)^{y_i} \le Z(G_i)$) if condition \ref{RGlistC} holds for $S_i$ (respectively $S_i^\dagger$);
\item $h_i \in \GL_2(q)$ (respectively $h_i' \in \GL_2(q)$) is upper-triangular if condition \ref{RGlistA} holds
$S_i$ (respectively $S_i^\dagger$).
\end{itemize}
 Therefore, $h_i$ and $h_i'$ are either scalar or upper-triangular, so  $(2a)$   of Lemma \ref{GRdiag} holds.

\medskip

{\bf Case (2.2).} Assume that the number $r$ of $G_i$ and $G_i^\dagger$ equal to $\GL_2(q)$ is $1$. Notice that if $G_i^\dagger=\GL_2(q)$, then $G_i=\GL_2(q)$ and $r \ge 2.$ Hence there exists unique $j \in \{1, \ldots, k+1\} \backslash \{i_1, \ldots, i_s\}$ such that $G_j$ is $\GL_2(q)$.  Let $y_j=I_2$  and let $x_i, x'_i, y_i, y'_i$ for $i \in \{1, \ldots, k\} \backslash \{j\}$  be as defined before {\bf Case (2.1)}. Let $x=\diag[x_{i_1}',\ldots, x_{i_s}', x_{k+1}  \ldots, x_1]$ and $y=\diag[y_{i_1}', \ldots, y_{i_s}', y_{k+1}, \ldots, y_k] \cdot a$. It is easy to see that $(2b)$   of Lemma \ref{GRdiag} holds.  
\end{proof}

Now we prove Theorem \ref{theoremGR}.

\begin{proof}[Proof of Theorem {\rm \ref{theoremGR}}]
Let $U=U_1$, $W=W_1$ and $m=m_1.$ If  $Q \le V$ has dimension $r$, then we write 
\begin{equation} \label{Qnotation}
Q= \left \langle 
\begin{pmatrix}
u_1 \\
 \vdots \\
u_r
\end{pmatrix}
\right \rangle,
\end{equation}
where $u_1, \ldots, u_r \in V$ form a basis of $Q$.

The proof splits into two cases: when $(1)$ and $(2)$ of Lemma \ref{GRdiag} holds respectively.

\medskip

{\bf Case 1.} Assume that $(1)$ of Lemma \ref{GRdiag} holds. We study two subcases:
\begin{description}[before={\renewcommand\makelabel[1]{\bfseries ##1}}]
\item[{\bf Case (1.1)}] $m_i=1$ for $i \in \{1, \ldots, k\}$ and $n_{k+1} \in \{0,1\}$;
\item[{\bf Case (1.2)}] $m_i\ge 2$ for some $i \in \{1, \ldots, k\}$ or $n_{k+1}\ge 2.$
\end{description}

\medskip

{\bf Case (1.1).} Assume that $m_i=1$ for $i \in \{1, \ldots, k\}$ and $n_{k+1} \in \{0,1\}.$ Let $x=a$ where $a$ is as in \eqref{adefGR}, so $\det x = \pm 1$. Notice that if $\varphi \in S^x,$ then $\varphi$  has shape \eqref{GRvarphishape} with 
\begin{equation}
\label{GRginsy}
g= \begin{pmatrix}
g_{i_1}' & & & & & & \\
* & \ddots & & & & &\\
* & *&g_{i_s}' &  & & &\\
* & &* &g_{i_{k+1}} & & & \\
* &\ldots & &* & g_{k} & & \\
* &\ldots &\ldots& \ldots &*  & \ddots &\\
* &\ldots &\ldots &\ldots & &* &  g_{1}
\end{pmatrix}.
\end{equation}
So, if $\varphi \in S \cap S^x,$ then $g$ is diagonal.

 Let $y ,z \in \SL_n(q)$ be as in the proof of Proposition \ref{ni1}. Let us show that if $\varphi \in S \cap S^x \cap S^y$, then $\varphi \in S \cap \Gamma.$ Assume that $\varphi \notin \Gamma,$ so $l=1$ in  \eqref{GRvarphishape}.  Since $S$ stabilises $(U,W),$ $S^y$ stabilises $(U,W)y$. Therefore, $$(U,W)y \varphi =(U,W)y$$
and $Uy=Wy\varphi$ since $\dim Uy= \dim Wy\varphi.$ With respect to $\beta$,
$$Uy= \left \langle 
\begin{pmatrix}
1, \ldots, 1 \\
\end{pmatrix}
\right \rangle.$$ On the other hand, 
$$Uy=Wy \varphi =W y (\iota_{\beta} a) \cdot (\phi_{\beta})^j \cdot g =W \iota_{\beta} y^{\iota_{\beta}}a\cdot (\phi_{\beta})^j \cdot g= W'y^{\iota_{\beta}}a\cdot (\phi_{\beta})^j \cdot g,$$
where $W'$ is spanned by $\beta \backslash (W \cap \beta).$ 

 If $U \cap W = \{0\},$ then $W'=U.$ It is now easy to see that 
$$W'y^{\iota_{\beta}}a\cdot (\phi_{\beta})^j \cdot g = 
\begin{cases} \left \langle 
\begin{pmatrix}
1,0, \ldots, 0 \\
\end{pmatrix}
\right \rangle \text{ if } M_2 \text{ is of type } \GL_1(q) \oplus \GL_{n_2-1}(q); \\
\left \langle 
\begin{pmatrix}
0, \ldots,1, 0 \\
\end{pmatrix}
\right \rangle \text{ if } M_2 \text{ is of type } P_{1, n_2-1}
\end{cases} $$
since $g$ is diagonal. So, $Uy \ne Wy \varphi$ which is a contradiction. Hence $\varphi \in \Gamma.$

If $U\le W$, then $W'= \left \langle 
\begin{pmatrix}
1, 0, \ldots, 0 \\
\end{pmatrix}
\right \rangle$, $W'y^{\iota_{\beta}}= \left \langle 
\begin{pmatrix}
1, 0, \ldots, -1 \\
\end{pmatrix}
\right \rangle$ and  $$W'y^{\iota_{\beta}}a\cdot (\phi_{\beta})^j \cdot g=\left \langle 
\begin{pmatrix}
\alpha_1, 0, \ldots,0, \alpha_2 \\
\end{pmatrix}
\right \rangle$$ for some $\alpha_1, \alpha_2 \in \mathbb{F}_q^*.$ So, $Uy \ne Wy \varphi$ which is a contradiction. Hence $\varphi \in \Gamma.$

Therefore, $S\cap S^x \cap S^y \cap S^z = \tilde{S}\cap \tilde{S}^x \cap \tilde{S}^y \cap \tilde{S}^z \le Z(\GL_n(q))$ by Proposition \ref{ni1}. Notice that if $\det a=-1$, then we can take $x= \diag(-1, 1, \ldots, 1) \cdot a$ and the argument above still works, so we can assume $x,y,z \in \SL_n(q).$

\medskip

{\bf Case (1.2).}  Let $x,y$ be as in $(1)$ of Lemma \ref{GRdiag}. Assume that  $m_i\ge 2$ for some $i \in \{1, \ldots, k\}$ or $n_{k+1}\ge 2.$ So, if $\varphi \in (S \cap S^x \cap S^y) \cap  \Gamma$, then  there exists $r \in \{1, \ldots, n\}$ such that  $\alpha_r =\alpha_{r+1}$ in \eqref{GRvarphdiagc1}. We choose $r$ to be minimal such that $\alpha_r =\alpha_{r+1}$ for all such $\varphi.$

 By  \eqref{GRPmnmdiag}, \eqref{GRSy} and similar arguments for the situation when $U \cap W = \{0\}$ , if $\varphi \in S \cap S^y$, then 
\begin{equation}
\label{GRvarphisy}
\varphi = (\iota_{\beta} a )^l \cdot (\phi_{\beta}^j) \cdot g
\end{equation}
where $l \in \{0,1\},$ $j \in \{0, 1, \ldots, f-1\},$
$$a = \begin{cases}
I_n  &  \text{ if } M_1 \text{ is of type } \GL_m(q) \oplus \GL_{n-m(q)};\\
a(n,m) & \text{ if } M_1 \text{ is of type } P_{m, n-m}

\end{cases}
$$
and
$$g = 
\begin{cases}
\diag[g_2, g_1] &  \text{ if } M_1 \text{ is of type } \GL_m(q) \oplus \GL_{n-m(q)};\\
\diag[g_1', g_2, g_1] & \text{ if } M_1 \text{ is of type } P_{m, n-m}.
\end{cases}$$
Here $g_1 \in \GL_m(q)$, $g_2 \in \GL_{n-m}(q)$ in the first option and  $g_1, g_1' \in \GL_m(q)$, ${g_2 \in \GL_{n-2m}(q)}$ in the second. Our consideration of {\bf Case (1.2)} splits into two subcases:
when  $U \ne W$
and  $U=W$ respectively.

\medskip

{\bf Case (1.2.1).} Assume that $U \ne W$. Let $\theta$ be a generator of $\mathbb{F}_q^*$ and let $z \in \SL_n(q)$ be defined as follows:
\begin{equation}
\label{GRzdef1}
\begin{aligned}
(v_i)z & = v_i &&\text{ for } i \in \{1, \ldots, n-m\}; \\
(v_{n-m+1})z & =  \sum_{i=1}^{n-m} v_i -v_r + \theta v_{r} + v_{n-m+1}; && \text{ if } r \le n-m\\
(v_{n-m+1})z & =  \sum_{i=1}^{n-m-1} v_i  + \theta v_{n-m} + v_{n-m+1}; && \text{ if } r \ge n-m+1\\
(v_i)z & =  v_{n-m} + v_{i} && \text{ for } i \in \{n-m+2, \ldots, n\}.
\end{aligned} 
\end{equation}

Let $\varphi \in S \cap S^x \cap S^y \cap S^z$, so $\varphi$ has shape \eqref{GRvarphisy}.   Assume that $\varphi \notin \Gamma,$ so $l=1.$ Since $\varphi \in S^z$, it stabilises $(U,W)z,$ so $(U,W)z \varphi =(U,W)z$ and, therefore, 
$$Uz=Wz \varphi = W'z^{\iota_{\beta}}a\cdot (\phi_{\beta})^j \cdot g,$$
where $W'$ is spanned by $\beta \backslash (W \cap \beta).$  

 With respect to $\beta$, using the notation \eqref{Qnotation},
$$Uz= \left \langle 
\left(\begin{array}{ccccccc|ccccc}
1 & \ldots & 1 & \theta &1 & \ldots & 1     &             1 & & &  \\
1 & \ldots & 1 &  1    &1 & \ldots & 1       &           & 1 & &  \\
\vdots &   &   &       &  &        &\vdots    &          & &\ddots & \\
1 & \ldots & 1 &  1    &1 & \ldots & 1         &         & & & 1 
\end{array}\right)
\right \rangle$$
where $\theta$ in the first line is either in the $r$-th or $(m-n)$-th column, and the part after the vertical line forms $I_m.$

 If $U \cap W=\{0\},$ then $W'=U$ and it is easy to see that $Wz \varphi = W'z^{\iota_{\beta}}a\cdot (\phi_{\beta})^j \cdot g= U,$ since $g$ stabilises $U$ by \eqref{GRvarphisy}. So, $Uz \ne Wz \varphi$ which is a contradiction. Hence $\varphi \in \Gamma.$

If $U \le W$, then $W'= \langle (I_m \mid 0_{ m \times (n-m)}) \rangle$, so 
$$W'z^{\iota_{\beta}}a= \left \langle 
(A \mid 0_{n \times (n-2m)} \mid I_n)
\right \rangle$$
where $A$ is $m \times m$ matrix with entries $-1$ or $- \theta$, and $-\theta$ can occur at most once.
Therefore, $W'z^{\iota_{\beta}}a \cdot g = \langle (A g_1'\mid0_{m \times (n-2m)}\mid g_1) \rangle$. Notice that $n-2m\ge 1$ since $U \ne W.$ So, $Uz \ne Wz \varphi$ which is a contradiction. Hence $\varphi \in \Gamma.$

Therefore, $\varphi = (\phi_{\beta})^j \cdot \diag(\alpha_1, \ldots, \alpha_n)$ as in $(1)$ of Lemma
\ref{GRdiag}. Since $\varphi \in S \cap S^z \cap \Gamma,$ it stabilises $U$ and $Uz.$  

Consider $((v_{n-m+1})z) \varphi.$ First, let $r\le n-m,$ so 
\begin{equation*}((v_{n-m+1})z) \varphi=
\begin{cases}
\left(\sum_{\substack{i \in \{1, \ldots, n-m\}\\ i \ne r}} \alpha_i v_i \right) + \alpha_r \theta^{p^j} v_r + \alpha_{n-m+1} v_{n-m+1};\\
\sum_{i=1}^m \delta_i (v_{n-m+i})z
\end{cases}
\end{equation*}
for some $\delta_i \in \mathbb{F}_q.$ Since $((v_{n-m+1})z) \varphi$ has no $v_i$ for $i \in \{n-m+2, \ldots, n\}$ in the decomposition with respect to $\beta$, 
$$\delta_2 = \ldots = \delta_m =0$$
and $((v_{n-m+1})z) \varphi= \delta_1 ((v_{n-m+1})z).$ Hence 
$$\delta_1 = \alpha_r \theta^{p^j-1}=\alpha_{r+1} = \alpha_i \text{ for } i \in \{1, \ldots, n-m+1\} \backslash \{r,r+1\},$$
so $j=0$ and $\varphi \in Z(\GL_n(q)).$

Now let $r\ge n-m+1,$ so $m\ge 2$ and 
\begin{equation*}((v_{n-m+1})z) \varphi=
\begin{cases}
\left(\sum_{i=1}^{n-m-1} \alpha_i v_i \right) + \alpha_{n-m} \theta^{p^j} v_{n-m} + \alpha_{n-m+1} v_{n-m+1};\\
\sum_{i=1}^m \delta_i (v_{n-m+i})z
\end{cases}
\end{equation*}
for some $\delta_i \in \mathbb{F}_q.$ Since $((v_{n-m+1})z) \varphi$ has no $v_i$ for $i \in \{n-m+2, \ldots, n\}$ in the decomposition with respect to $\beta$, 
$$\delta_2 = \ldots = \delta_m =0$$
and $((v_{n-m+1})z) \varphi= \delta_1 ((v_{n-m+1})z).$ Hence 
$$\delta_1 = \alpha_{n-m} \theta^{p^j-1}=\alpha_{n-m+1} = \alpha_i \text{ for } i \in \{1, \ldots, n-m-1\}.$$
The same arguments for $((v_{n-m+2})z) \varphi$ show that $\alpha_{n-m+2}=\alpha_{n-m},$ so, since $\alpha_{n-m+1}=\alpha_{n-m+2}$ by $(1)$ of Lemma \ref{GRdiag}, we obtain $\theta^{p^j-1}=1.$ Hence $j=0$ and $\varphi \in Z(\GL_n(q)).$

\medskip
{\bf Case (1.2.2).} Assume that $U=W$, so $m=n/2$ and $M$ is of type $P_{n/2,n/2}.$ In particular, $n \ge 4,$ since $n \ge 3$ by the assumption of the theorem.  Let $\theta$ be a generator of $\mathbb{F}_q^*$ and let $z \in \SL_n(q)$ be defined as follows: 
\begin{equation}
\label{GRzdef2}
\begin{aligned}
(v_i)z= &v_i &&\text{ for } i \in \{1, \ldots, n-m\}; \\
(v_{m-n+1})z= & \theta v_{n-m} + v_{m-n+1}; &&\\
(v_i)z= & v_{n-m} + v_{i} && \text{ for } i \in \{n-m+2, \ldots, n\}.
\end{aligned} 
\end{equation}
Arguments, analogous to  ones in {\bf Case (1.2.1)},  show that  $S \cap S^x \cap S^y \cap S^z \le Z(\GL_n(q))$.
\bigskip

{\bf Case 2.}  Assume that $(2)$ of Lemma \ref{GRdiag} holds.
 For $n \in \{3,4\}$ the theorem follows by computation, so  we may assume that $n \ge 5$.

We adopt notation from the proof of Lemma \ref{GRdiag}, in particular $x,y$, $S_i$, $G_i$.
Let $(2 \times 2)$ blocks (corresponding to the $S_i$ and the $S_i^\dagger$ lying in $\GL_2(q)$) on the diagonal in matrices of $\tilde{S}=S \cap \Gamma$ occur in the rows $$(j_1, j_1+1), (j_2, j_2 +1), \ldots, (j_r, j_r+1).$$ Let $\Lambda=\Lambda_1 \cup \Lambda_2$ where 
$\Lambda_1=\{j_1, j_2, \ldots, j_r\}$ and $\Lambda_2=\{j_1+1, j_2+1, \ldots, j_r+1\}.$
 Let $U=U_1$, $W=W_1$ and $m=m_1.$  
 Notice, that if $\varphi \in S^y$, then it has shape \eqref{GRvarphishape} where $g$ has shape \eqref{GRginsy} with $g_i, g_i' \in \GL_{m_i}$ for $i \in \{1, \ldots, k\}$ and $g_{k+1} \in \GL_{n_{k+1}}(q).$ So, if $\varphi \in S \cap S^y$, then it has shape \eqref{GRvarphishape} with 
 \begin{equation}
 \label{GRgdiagcase2}
 g = \diag[g_{i_1}', \ldots, g_{i_s}', g_{k+1}, g_k, \ldots, g_1].
 \end{equation}
 
 Our consideration of {\bf Case 2} splits into two subcases: when $(2a)$ and $(2b)$ of Lemma \ref{GRdiag} holds respectively.

\medskip

{\bf Case (2.1).} Assume that $(2a)$ of Lemma \ref{GRdiag} holds. We consider two subcases: when $m \ge 2$ and $m=1$.

{\bf Case (2.1.1).} Assume that $m \ge 2.$ Let $z \in \SL_n(q)$ be defined as follows
\begin{equation}
\label{GRzdefq23}
\begin{aligned}
(v_i)z & =  v_i &&\text{ for } i \in \{1, \ldots, n-m\}; \\
(v_{n-m+1})z & =  \left(\underset{i \in \{1, \ldots, n-m\} \backslash \{j_1\}}{\sum} v_i \right) + v_{n-m+1};\\
(v_{n-m+2})z & =  \left(\underset{i \in \{1, \ldots, n-m\} \backslash \Lambda_2}{\sum} v_i \right)   + v_{n-m+2};\\
(v_i)z & =  \sum_{j=1}^{n-m} v_{j} + v_{i} && \text{ for } i \in \{n-m+3, \ldots, n\}.
\end{aligned} 
\end{equation}

Let $\varphi \in S \cap S^x \cap S^y \cap S^z$, so $\varphi$ has shape \eqref{GRvarphishape} where $g$ has shape \eqref{GRgdiagcase2}.   Assume that $\varphi \notin \Gamma,$ so $l=1.$ Since $\varphi \in S^z$, it stabilises $(U,W)z,$ so $(U,W)z \varphi =(U,W)z$ and, therefore, 
$$Uz=Wz \varphi = W'z^{\iota_{\beta}}a \cdot g,$$
where $W'$ is spanned by $\beta \backslash (W \cap \beta).$  

 With respect to $\beta$,
$$Uz= \left \langle 
\left(\begin{array}{ccc|ccccc}
\lambda_1 & \ldots & \lambda_{n-m}   &           1 & & & & \\
\mu_1     & \ldots & \mu_{n-m}       &           & 1 & & & \\
  1       & \ldots &  1              &           &   &1 & &  \\
\vdots    &        &\vdots           &           &  & &\ddots & \\
1         & \ldots &  1              &           &  & & & 1 
\end{array}\right)
\right \rangle$$
where $\lambda_i, \mu_i \in \{0,1\}$ according to \eqref{GRzdefq23}, so for each $i \in \{1, \ldots, n-m\}$ at least one of $\lambda_i$ and $\mu_i$ is $1$,  and the part after the vertical line forms $I_m.$

 If $U \cap W=\{0\},$ then $W'=U$ and it is easy to see that $Wz \varphi = W'z^{\iota_{\beta}}a\cdot (\phi_{\beta})^j \cdot g= U,$ since $g$ stabilises $U$ by \eqref{GRvarphishape}. So, $Uz \ne Wz \varphi$ which is a contradiction. Hence $\varphi \in \Gamma.$

If $U \le W$, then $W'= \langle (I_m  \mid  0_{(m \times n-m)}) \rangle$, so 
$$W'z^{\iota_{\beta}}a= \left \langle 
(A \mid 0_{m \times (n-2m)} \mid I_n)
\right \rangle$$
where $A$ is $m \times m$ matrix with entries $-1$ ,$-\lambda_i$ and $- \mu_i$.
Therefore, $$W'z^{\iota_{\beta}}a \cdot g = \langle (A g_1'\mid 0_{m \times (n-2m)}\mid g_1) \rangle.$$ Notice that $n-2m\ge 1,$ since otherwise  $U = W,$ so $m=2$ and $n=4.$ Thus, $Uz \ne Wz \varphi$ which is a contradiction. Hence $\varphi \in \Gamma.$

Therefore, $\varphi = g= \diag[g_{i_1}, \ldots,g_{i_s}, g_{k+1}, g_k, \ldots,  g_1]$ as in $(a)$ of $(2)$ of Lemma
\ref{GRdiag}. Specifically, let $(v_i)\varphi =\alpha_i v_i$ for $i \in \{1, \ldots, n\} \backslash \Lambda_1$ and let $(v_i)\varphi = \alpha_i v_i + \gamma_i v_{i+1}$ for $i \in \Lambda_1$ with $\alpha_i, \gamma_i \in \mathbb{F}_q.$

 Since $\varphi \in S \cap S^z \cap \Gamma,$ it stabilises $U$ and $Uz.$  
Therefore, $((v_{n-m+2})z) \varphi$ is
\begin{equation*}
\label{vnm2zph}
\left(\underset{i \in \{1, \ldots, n-m\} \backslash \Lambda_2}{\sum} \alpha_i v_i \right) +
\left(\underset{i \in \{1, \ldots, n-m\} \cap \Lambda_2}{\sum} \gamma_{i-1} v_i \right) + \alpha_{n-m+2} v_{n-m+2},
\end{equation*}
and
$$((v_{n-m+2})z) \varphi= \sum_{i=1}^m \delta_i (v_{n-m+i})z$$
for some $\delta_i \in \mathbb{F}_q.$ Since $((v_{n-m+2})z) \varphi$ does not contain $v_i$ for $i \in \{n-m+1,n-m+3, \ldots, n\}$ in the decomposition with respect to $\beta$, 
$$\delta_1= \delta_3 = \ldots = \delta_m =0$$
and $((v_{n-m+1})z) \varphi= \delta_1 ((v_{n-m+1})z).$ Hence 
$$\gamma_{j_1} = \ldots = \gamma_{j_r}=0,$$
so $\varphi= \diag(\alpha_1, \ldots, \alpha_n)$ where 
\begin{equation}
\label{GRalphasvnm2}
\alpha_i=\alpha_{n-m+2} \text{ for } i \in \{1, \ldots, n-m\} \backslash \Lambda_2.
\end{equation} 

Consider 
\begin{equation*}((v_{n-m+1})z) \varphi=
\begin{cases}
\underset{i \in \{1, \ldots, n-m\} \backslash \{j_1\}}{\sum} \alpha_i v_i + \alpha_{n-m+1} v_{n-m+1} +\delta_{2,m}{\gamma_{n-m+1} v_{n-m+1}}\\
\sum_{i=1}^m \delta_i (v_{n-m+i})z
\end{cases}
\end{equation*}
for some $\delta_i \in \mathbb{F}_q.$ 
 Since $((v_{n-m+1})z) \varphi$ contains neither $v_{j_1}$ (notice that $j_1>n-m$ since if $(2a)$ of Lemma \ref{GRdiag} holds, then $r \ge 2$) nor $v_i$ for $i \in \{n-m+3, \ldots, n\}$ in the decomposition with respect to $\beta$, 
$$\delta_2 = \ldots = \delta_m =0.$$
Here $\delta_2=0$ since $((v_{n-m+2})z) $ contains $v_{j_1}$ in the decomposition with respect to $\beta$ and $((v_{n-m+1})z)\varphi$ does not.
Thus $((v_{n-m+1})z) \varphi= \delta_1 ((v_{n-m+1})z)$ and 
$$\alpha_i = \alpha_{n-m+1} \text{ for } i \in \{1, \ldots, n-m\} \backslash \{j_1\}.$$
Combined with \eqref{GRalphasvnm2}, it implies $\varphi \in Z(\GL_n(q)).$ 

\medskip

{\bf Case (2.1.2).} Assume that $m_1=1$ and let $t$ be the smallest $i \in \{2, \ldots, k\}$ such that $m_i\ge 2.$ Such $t$ exists since otherwise $m_i=1$ for $i \in \{1, \ldots, k\}$ and $n_{k+1}=2,$ so $(2b)$ of Lemma \ref{GRdiag} holds. Let $d_i=\sum_{j=1}^i m_j$ for $i \in \{1, \ldots, k\}$. Let $z \in \SL_n(q)$ be defined as follows:
\begin{equation}
\label{GRzdefq23mr}
\begin{aligned}
(v_i)z & =  v_i &&\text{ for } i \in \{1, \ldots, n-d_t\}; \\
(v_{n-d_t+1})z & =  \left(\underset{i \in \{1, \ldots, n-d_t\} \backslash \{j_1\}}{\sum} v_i \right) + v_{n-d_t+1};\\
(v_{n-d_t+2})z & =  \left(\underset{i \in \{1, \ldots, n-d_t\} \backslash \Lambda_2}{\sum} v_i \right)   + v_{n-d_t+2};\\
(v_i)z & =  v_i && \text{ for } i \in \{n-d_{t}+3, \ldots, n-1\};\\
(v_n)z & =  \sum_{i=1}^n v_i.
\end{aligned} 
\end{equation}
Arguments, analogous to  ones in {\bf Case (2.1.1)},  show that  $S \cap S^x \cap S^y \cap S^z \le Z(\GL_n(q))$.

\bigskip

{\bf Case (2.2).} Finally, to complete the proof, let us assume that $(2b)$ of Lemma \ref{GRdiag} holds. We consider two subcases: when $m=2$ and $m \ne 2$.

\medskip

{\bf Case (2.2.1).} Assume that $m=2,$ so if $g \in S \cap S^x \cap S^y \cap \GL_n(q),$ then $g_1 \in \GL_2(q)$ and $g_i$, $g_i'$ are scalar for $i \in \{2, \ldots, k+1\}.$ We may assume $U \cap W=0$ since otherwise the number of $G_i$ and $G_i^\dagger$ equal to $\GL_2(q)$ is at least 2 and $(2a)$  of Lemma \ref{GRdiag} holds. Let $z \in \SL_n(q)$ be defined as follows:
\begin{equation}
\label{GRzdefcase221}
\begin{aligned}
(v_i)z & =  v_i &&\text{ for } i \in \{1, \ldots, n-2\}; \\
(v_{n-1})z & =  \sum_{i=2}^{n-2} v_i + v_{n-1};\\
(v_{n})z & =  v_1+ \sum_{i=3}^{n-2} v_i   + v_{n}.\\
\end{aligned} 
\end{equation}
Let $\varphi \in S \cap S^x \cap S^y \cap S^z$, so $\varphi= (\iota_{\beta}a)^l \cdot g$ where $g=\diag[g_{i_1}',\ldots, g_{i_s}', g_{k+1}, \ldots, g_1]$ where $g_i$ and $g_i'$ are as in \eqref{GRguppdiag}.  
 Assume that $\varphi \notin \Gamma,$ so $l=1.$ Since $\varphi \in S^z$, it stabilises $(U,W)z,$ so $(U,W)z \varphi =(U,W)z$ and, therefore, 
$$Uz=Wz \varphi = W'z^{\iota_{\beta}}a \cdot g,$$
where $W'$ is spanned by $\beta \backslash (W \cap \beta).$  

 With respect to $\beta$,
$$Uz= \left \langle 
\left(\begin{array}{ccccc|cc}
0  & 1& 1 & \ldots & 1   &           1 &0  \\
1  & 0& 1     & \ldots & 1       &    0       & 1 
\end{array}\right)
\right \rangle.$$
 Since $U \cap W=\{0\},$ we obtain $W'=U$ and it is easy to see that $Wz \varphi = W'z^{\iota_{\beta}}a\cdot g= U,$ since $g$ stabilises $U$. So, $Uz \ne Wz \varphi$ which is a contradiction. Hence $\varphi =g\in \Gamma.$ Therefore, $g=\diag[A,g_1]$ where $A=\diag(\alpha_1, \ldots, \alpha_{n-2})$ for some $\alpha_i \in \mathbb{F}_q^*$ and $$g_1 = \left( \begin{matrix} \delta_1 & \delta_2 \\ \delta_3 & \delta_4  \end{matrix} \right) \in \GL_2(q).$$ Since $g \in S^z$, it stabilises $Uz.$
 
 Consider 
\begin{equation}
\label{GRc221vn1z}
 ((v_{n-1})z)g =
 \begin{cases}
 \sum_{i=2}^{n-2} \alpha_i v_i + \delta_1 v_{n-1} + \delta_2 v_n;\\
 \lambda_1 (v_{n-1})z + \lambda_2 (v_{n})z
 \end{cases}
 \end{equation}
 for some $\lambda_{i} \in \mathbb{F}_q.$ Since there is no $v_1$ in the first line of \eqref{GRc221vn1z}, $\lambda_2=0,$ so
 $$\alpha_2 = \ldots = \alpha_{n-2}=\delta_1 \text{ and } \delta_2=0.$$
 The same arguments for $((v_{n})z)g$ show that $\delta_3=0$ and, since $n \ge 5$,
 $$\alpha_1=\alpha_{n-2}=\delta_4.$$
 Hence $g$ is scalar and $\varphi \in Z(\GL_n(q)).$
 
 \medskip
 
 {\bf Case (2.2.2).} Assume $m\ne 2.$  
  We may assume that $n_{k+1}\ne 2$. Indeed, if $n_{k+1}=2$, then  $n_{k}=3$ since $m_k\le n_k/2$ and there is only one $G_i$ equal to $\GL_2(q)$ for $i \in \{1, \ldots, k+1\}.$ Therefore, $$S_k=S|_{_{V_k}}\cap \GL(V_k)=\GL_2(q) \times \GL_1(q) \le \GL(V_k)=\GL_3(q)$$ and, using computation, we obtain that  there are $x_k, y_k \in \SL_n(q)$ such that $S_k \cap S_k^{x_k} \cap S_k^{y_k} \le Z(\GL_3(q))$, so the conclusion of  $(1)$ of Lemma \ref{GRdiag} holds and the theorem  holds by {\bf Case  1}. 
  
  Therefore, $j_1\ge 3$ where $j_1$ is as defined in the beginning of {\bf Case  2}, so the $(2 \times 2)$ block (corresponding to the $S_i$ lying in $\GL_2(q)$) on the diagonal in matrices of $\tilde{S}=S \cap \Gamma$ occurs in the rows $(j_1, j_1+1)$. Recall that if $h \in  S \cap S^x \cap S^y \cap  \GL_n(q)$, then $$h= \diag[\alpha_1, \ldots, \alpha_{j_1-1},A, \alpha_{j_1+2}, \ldots, \alpha_n]$$ where $\alpha_i \in \mathbb{F}_q^*$ and $A= \left( \begin{smallmatrix} \delta_1 & \delta_2 \\ \delta_3 & \delta_4  \end{smallmatrix} \right) \in \GL_2(q).$ Let $z \in \SL_n(q)$ be defined as follows:
\begin{equation}
\label{GRzdefcase222}
\begin{aligned}
(v_i)z & =  v_i &&\text{ for } i \in \{1, \ldots, n\} \backslash \{j_1, j_1+1, n\}; \\
(v_{j_1})z & =  v_1 + v_{j_1};\\
(v_{j_1+1})z & =  v_2 + v_{j_1+1};\\
(v_{n})z & =   \sum_{i=1}^{n-m} v_i   + v_{n}.\\
\end{aligned} 
\end{equation}
Let $\varphi \in S \cap S^x \cap S^y \cap S^z$, so $\varphi= (\iota_{\beta}a)^l \cdot g$ with $g=\diag[g_{i_1},\ldots, g_{i_s}, g_{k+1}, \ldots, g_1]$ where $g_i$ and $g_i'$ are as in \eqref{GRguppdiag}.  
 Assume that $\varphi \notin \Gamma,$ so $l=1.$ Since $\varphi \in S^z$, it stabilises $(U,W)z,$ so $(U,W)z \varphi =(U,W)z$ and, therefore, 
$$Uz=Wz \varphi = W'z^{\iota_{\beta}}a \cdot g,$$
where $W'$ is spanned by $\beta \backslash (W \cap \beta).$  

 With respect to $\beta$,
$$Uz= \left \langle 
\left(\begin{array}{ccc|ccccc}
 0        & \ldots & 0               &           1 & & &  \\
\vdots    &        &\vdots           &           & \ddots  & &  \\
0         & \ldots & 0               &            & &\ddots &  \\
1         & \ldots &  1              &           &  & & 1 
\end{array}\right)
\right \rangle$$
where the part after the vertical line forms $I_m.$

 If $U \cap W=\{0\},$ then $W'=U$ and it is easy to see that $Wz \varphi = W'z^{\iota_{\beta}}a \cdot g= U,$ since $g$ stabilises $U$ by \eqref{GRvarphisy}. So $Uz \ne Wz \varphi$ which is a contradiction. Hence $\varphi \in \Gamma.$
 
 If $U \le W,$ then 
 \begingroup
\allowdisplaybreaks
 \begin{align*}
 Wz\varphi & = W'z^{\iota_{\beta}}ag \\ & = 
 \left\langle I_m\mid 0_{m \times (n-m)} \right\rangle \, z^{\iota_{\beta}}ag \\
& =\Scale[0.95]{\left \langle \left(\begin{array}{ccccc|cccccc|cccc}
1 &  &   &        &          & 0 & \ldots 0 &-1 & 0 & 0 & \ldots 0 &       0  & 0 & \ldots &  0 \\         
  &1 &   &        &          & 0 & \ldots 0 & 0 &-1 & 0 & \ldots 0 &       0  & 0 & \ldots &  0 \\
  &  &1  &        &          & 0 & \ldots 0 & 0 &0  & 0 & \ldots 0 &       -1 & 0 & \ldots &  0 \\
  &  &   & \ddots &          & \vdots &     &   &   &   &          &   \vdots &   &        &    \\
  &  &   &        &1         & 0 & \ldots 0 & 0 &0  & 0 & \ldots 0 &       -1 & 0 & \ldots &  0 \\ 
\end{array}\right)
\right \rangle}\, a g\\
& = \Scale[0.95]{ \left \langle \left(\begin{array}{cccc|cccccc|ccccc}
 0  & 0 & \ldots &  0          & 0 & \ldots 0 &-1 & 0 & 0 & \ldots 0 &       1 &  &   &        & \\         
 0  & 0 & \ldots &  0        & 0 & \ldots 0 & 0 &-1 & 0 & \ldots 0 &         &1 &   &        &   \\
 -1 & 0 & \ldots &  0         & 0 & \ldots 0 & 0 &0  & 0 & \ldots 0 &         &  &1  &        &  \\
  &   \vdots &   &        &            & \vdots &     &   &   &   &           &  &   & \ddots &  \\
  -1 & 0 & \ldots &  0        & 0 & \ldots 0 & 0 &0  & 0 & \ldots 0 &        &  &   &        &1  \\ 
\end{array}\right)
\right \rangle}\, g
 \end{align*}
 \endgroup
 where $-1$ in the first row is in the $j_1$ entry, and $-1$ in the second row is in the $j_1+1$ entry. The result of the action of $g$ on the first two rows is 
 $$\langle 0_{2\times 1}, \ldots, 0_{2\times 1}, -g_t, 0_{2\times 1}, \ldots, 0_{2\times 1}, g_1^{(1,2)} \rangle $$ where $t \in \{1, \ldots, k+1\}$ is such that $G_t=\GL_2(q)$ and $g_1^{(1,2)}$ is the matrix formed by the first two rows of $g_1.$ It is easy to see that  two such vectors cannot lie in $Uz$, so  $Uz \ne Wz \varphi$ which is a contradiction. Hence $\varphi =g\in \Gamma.$ 
 
 Therefore $g=\diag[A,g_t, B]$ where $$A=\diag(\alpha_1, \ldots, \alpha_{j_1-1}), \text{ }B=\diag(\alpha_{j_1+2}, \ldots, \alpha_{n}),$$  for some $\alpha_i \in \mathbb{F}_q^*$ and $$g_t = \left( \begin{matrix} \delta_1 & \delta_2 \\ \delta_3 & \delta_4  \end{matrix} \right) \in \GL_2(q).$$ Since $g \in S^z$, it stabilises $Uz.$
 Notice that $S \cap \GL_n(q)$ also stabilises $\langle v_{j_1}, \ldots, v_n \rangle,$ so $g \in S^z \cap \GL_n(q)$ stabilises $\langle v_{j_1}, \ldots, v_n \rangle z$. 
 
 Consider 
\begin{equation}
\label{GR222vj1}
((v_{j_1})z)g= 
 \begin{cases}
 \alpha_1 v_1 + \delta_1 v_{j_1} + \delta_2 v_{j_1+1};\\
 \sum_{i=j_1}^n \lambda_i v_i
 \end{cases}
\end{equation}
 for some $\lambda_i \in \mathbb{F}_q.$ Since the first line of \eqref{GR222vj1} contains no terms with $v_2$ and $v_i$ for $i \ge j_1+2$, we obtain   
 $((v_{j_1})z)g= \alpha (v_{j_1})z $ for some $\alpha \in \mathbb{F}_q^*.$ Therefore, $\delta_2=0$ and $\alpha= \alpha_1 = \delta_1.$ The same arguments applied to $(v_{j_1+1})z$ show that
 $\delta_4=0$ and $\alpha_2=\delta_3.$ 
 
   The same arguments applied to $(v_{n})z$ show that
$\alpha= \alpha_1 = \ldots =\alpha_n,$ so $g$ is scalar and $\varphi \in Z(\GL_n(q)).$
\end{proof}

\section*{Acknowledgements}
Much of the work presented in this paper was done while I was a PhD student at the University of Auckland. I thank my supervisors Eamonn O'Brien and Jianbei An for all their guidance and inspiration. I also thank Professor Timothy  Burness and Professor Peter Cameron, who were  examiners of my PhD thesis, and the anonymous referees of this paper for their constructive comments and corrections.

\bibliographystyle{abbrv}
\bibliography{SolLin.bib}

\end{document}